\documentclass[11pt]{amsart}

\usepackage{etex}
\usepackage{graphicx,amsthm}
\usepackage{enumerate}
\usepackage{empheq}
\usepackage{amssymb}
\usepackage[margin=1in]{geometry}
\usepackage{pst-all}
\usepackage{array}
\usepackage{tikz}
\usepackage{tikz-cd}
\usepackage{amsmath,amscd,amsfonts,epic,eepic,bbm, mathabx,graphicx,ytableau}
\allowdisplaybreaks[4]
\usepackage[onehalfspacing]{setspace}

\usepackage[all]{xy}

\usepackage{caption}
\usepackage{subcaption}
\usepackage{diagbox}
\usepackage{url}

\usepackage[mathscr]{euscript}
\usepackage{booktabs}
\usepackage{ amssymb }
\usetikzlibrary{decorations.markings, decorations.pathreplacing, calc, tikzmark,positioning}
\usetikzlibrary{cd,arrows}
\usetikzlibrary{arrows.meta}

\usepackage{hyperref}

\usepackage{colonequals}
\usepackage{ragged2e}

\DeclareMathOperator{\Stab}{Stab}

\newtheorem{theorem}{Theorem}[section]
\newtheorem{lemma}[theorem]{Lemma}
\newtheorem{proposition}[theorem]{Proposition}
\newtheorem{corollary}[theorem]{Corollary}
\newtheorem{conjecture}[theorem]{Conjecture}

\theoremstyle{definition}
\newtheorem{definition}[theorem]{Definition}
\newtheorem{example}[theorem]{Example}

\newtheorem{remark}[theorem]{Remark}

\numberwithin{equation}{section}

\newtheorem{thmy}{Theorem}
\newenvironment{thmx}{\stepcounter{theorem}\begin{thmy}}{\end{thmy}}

\newcommand{\C}{\mathbb{C}}
\newcommand{\Z}{\mathbb{Z}}
\newcommand{\R}{\mathbb{R}}
\DeclareMathOperator{\GL}{GL}
\newcommand{\Cstar}{\mathbb{C}^{\ast}}
\newcommand{\flag}{F\ell}
\newcommand{\Xwo}[1]{X_{#1}^{\circ}}

\newcommand{\Owo}[1]{\Omega_{#1}^{\circ}}
\newcommand{\Ow}[1]{\Omega_{#1}}
\newcommand{\Owho}[1]{\Omega_{#1,h}^{\circ}}
\newcommand{\Owh}[1]{\Omega_{#1,h}}
\newcommand{\Awh}[1]{A_{#1,h}}
\newcommand{\swh}[1]{\sigma_{#1,h}}
\DeclareMathOperator{\Hess}{Hess}
\DeclareMathOperator{\Hom}{Hom}
\newcommand{\Gh}{\Gamma_h}
\DeclareMathOperator{\codim}{\codim}
\newcommand{\Sym}{\mathcal{S}}
\DeclareMathOperator{\supp}{supp}

\newcommand{\uni}{\underline{i}}
\newcommand{\Gwh}{G_{w,h}}
\newcommand{\Ghw}{G_{w,h}}	
\newcommand{\vPQi}{w_{P,Q}^{(i)}}
\newcommand{\tvPQi}{\widetilde{w}_{P,Q}^{(i)}}
\newcommand{\tsvPQi}{\widetilde{\sigma}_{P,Q}^{(i)}}
\allowdisplaybreaks

\newcommand{\La}{\mathcal{G}_k}
\newcommand{\Gk}{\mathcal{G}_k}
\newcommand{\hats}{\widehat{\sigma}}

\tikzstyle{triple line} = [
	double distance = 2pt,
	double=\pgfkeysvalueof{/tikz/commutative diagrams/background color},
	postaction={decorate}
]

\newcommand{\la}{\lambda}
\newcommand{\x}{\mathbf{x}}

\newlength\cellsize 

\savebox2{%
\begin{picture}(18,18)\linethickness{0.6pt} %
\put(0,0){\line(1,0){18}}
\put(0,0){\line(0,1){18}}
\put(18,0){\line(0,1){18}}
\put(0,18){\line(1,0){18}}
\end{picture}}

\newcommand\cellify[1]{\def\thearg{#1}\def\nothing{}%
\ifx\thearg\nothing\vrule width0pt height\cellsize depth0pt%
  \else\hbox to 0pt{\usebox2\hss}\fi%
  \vbox to \cellsize{\vss\hbox to \cellsize{\hss$_{#1}$\hss}\vss}}

\newcommand\tableau[1]{\vtop{\let\\=\cr
\setlength\baselineskip{-1000pt}
\setlength\lineskiplimit{1000pt}
\setlength\lineskip{0pt}
\ialign{&\cellify{##}\cr#1\crcr}}}

\newlength{\celldim} 
\newsavebox{\cell}

\sbox{\cell}{%
\begin{picture}(22,22)\linethickness{0.6pt} %
  \put(0,0){\line(1,0){22}} \put(0,0){\line(0,1){22}}
  \put(22,0){\line(0,1){22}} \put(0,22){\line(1,0){22}}
\end{picture}}

\newcommand\cellifying[1]{%
  \def\thearg{#1}\def\nothing{}%
  \ifx\thearg\nothing \vrule width0pt height\celldim depth0pt\else
  \hbox to 0pt{\usebox{\cell} \hss}\fi%
  \vbox to \celldim{ \vss \hbox to
  \celldim{\hss$#1$\hss} \vss}
}
\newcommand\ttableau[1]{\vtop{\let\\\cr
\baselineskip -16000pt \lineskiplimit 16000pt \lineskip 0pt
\ialign{&\cellifying{##}\cr#1\crcr}}}

\newcommand\bas[1]{\omit \vbox to \cellsize{ \vss \hbox to \celldim{\hss$#1$\hss} \vss}}
\begin{document}

\title[Bases of the  equivariant cohomologies  of Hessenberg varieties]{Bases of the  equivariant cohomologies of \\ regular semisimple Hessenberg varieties}

\author{Soojin Cho}
\address{Department of Mathematics, Ajou University, Suwon  16499, Republic of Korea}
\email{chosj@ajou.ac.kr}

\author{Jaehyun Hong}
\address{Center for Complex Geometry,  Institute for Basic Science, Daejeon 34126, Republic of Korea}
\email{jhhong00@ibs.re.kr}

\author{Eunjeong Lee}
\address{Center for Geometry and Physics, Institute for Basic Science (IBS), Pohang 37673, Korea}
\curraddr{Department of Mathematics, Chungbuk National University, Cheongju 28644, Republic of Korea}
\email{eunjeong.lee@chungbuk.ac.kr}

\thanks{Cho was supported by the National Research Foundation of Korea (NRF-2020R1A2C1A01011045).
	Hong was supported by the Institute for Basic Science(IBS-R032-D1).
	Lee was supported by the Institute for Basic Science(IBS-R003-D1) and the National Research Foundation of Korea(NRF) grant funded by the Korea government(MSIT) (RS-2022-00165641).}

\begin{abstract}  We consider bases for the cohomology space of  regular semisimple  Hessenberg varieties, consisting of the classes that naturally arise from the  Białynicki-Birula decomposition of the Hessenberg varieties. We give an explicit combinatorial description of the support of each class, which enables us to compute the symmetric group actions on the classes in our bases. We then successfully apply the results to the permutohedral varieties to explicitly write down each class and to construct permutation submodules that  constitute summands of a decomposition of cohomology space of each degree. This resolves the problem posed by Stembridge on the geometric construction of permutation module decomposition  and also the conjecture posed by Chow on the construction of bases for the equivariant cohomology spaces of permutohedral varieties.
\end{abstract}

\keywords{Hessenberg varieties, equivariant cohomologies, representations of symmetric groups, permutohedral varieties}

\subjclass[2010]{Primary 14M15, 14C15; Secondary 05E05, 14L30}
\maketitle

\setcounter{tocdepth}{1}
\tableofcontents

\section{Introduction} \label{sec:intro}

\emph{Hessenberg varieties} were introduced and investigated by De Mari,  Shayman  and Procesi~\cite{DeMari_87, DeMari_Shayman_88, DPS_Hessenberg_var} around 1990 as a  generalization of subvarieties of the full flag variety  $\flag(\C^n)$ that arise in the study of Hessenberg matrices. Important varieties such as Springer fibers, Peterson varieties, and toric varieties associated with the Weyl chambers, as well as the flag varieties appear as examples of  Hessenberg varieties. In addition,  there have been many interesting results on the structure of the Hessenberg varieties.  We refer the reader to the survey article \cite{AbeHoriguchi} and references therein for more details.

Tymoczko applied GKM theory to \emph{regular semisimple} Hessenberg varieties in \cite{T1, T2} to obtain a combinatorial description of the (equivariant) cohomology spaces and to define Weyl group actions, called \emph{dot actions} on them.

Our main concern is the structure of the equivariant  cohomology\footnote{Throughout the whole paper, we will use the $\C$-coefficients for cohomology rings unless otherwise specified. } of regular semisimple Hessenberg varieties  $\Hess(S, h)$ of type $A$, described as follows:
\[
\Hess(S, h)= \{ \{0\} \subsetneq V_1 \subsetneq V_2 \subsetneq \cdots \subsetneq V_{n-1} \subsetneq \C^n  \mid S V_{i} \subset V_{h(i)} \, \text{ for all }1 \leq i \leq n \}\subset  \flag(\C^n)\,,
\]
where $S$ is a regular semisimple $n \times n$ matrix and $h: \{1, \dots, n\} \rightarrow \{1, \dots, n\}$ is a nondecreasing function satisfying $h(i)\geq i$ for all $i$, called a \emph{Hessenberg function}. The symmetric group $\mathfrak{S}_n$ is the Weyl group of type $A_{n-1}$, and Tymoczko's dot action admits an $\mathfrak{S}_n$-module structure on the space of equivariant cohomology of $\Hess(S, h)$.

The toric variety  $\mathcal H_n$ associated with the Weyl chambers of type $A_{n-1}$,  also known as the \emph{permutohedral variety}, is $\Hess(S, h)$ with the Hessenberg function $h$ given by $h(i)=i+1$ for $i<n$. The symmetric group $\mathfrak{S}_n$ acts on the Weyl chambers of type $A_{n-1}$, and it naturally defines an action on $H^*(\mathcal H_n)$.  It is known that this action coincides with Tymoczko's dot action.
The $\mathfrak{S}_n$-module structure of $H^*(\mathcal H_n)$ was first considered by Procesi in \cite{P1}, then by Stanley in \cite{S3} in the language of symmetric functions. Following the work of Stanley, Stembridge constructed a graded $\mathfrak{S}_n$-module carrying a permutation representation, which is isomorphic to $H^*(\mathcal H_n)$, by defining combinatorial objects called codes in~\cite{Ste}.  In the same paper, and also in \cite{Ste2}, Stembridge posed the problem to provide a \emph{geometric} explanation of the fact that $H^*(\mathcal H_n)$ can be decomposed as a direct sum of permutation modules. We provide a solution to this problem  in this paper by applying the general theory we develop  in Section~\ref{sec_description_minus_cell} and Section~\ref{sec:hong} to the permutohedral variety.

We note that the \emph{permutation module} $M^\la$  of $\mathfrak{S}_n$ for a partition $\la=(\la_1, \dots, \la_{\ell})$ of $n$ is the vector space of formal linear sums of the ordered tuples  $(J_1, \dots, J_{\ell})$ of disjoint subsets of $\{1, 2, \dots, n\}$ satisfying $|J_s|=\la_s$, $s=1, \dots, \ell$, and $|\bigcup_s J_s|=n$, where the permutations in $\mathfrak{S}_n$ act on $(J_1, \dots, J_{\ell})$ naturally. The permutation modules are direct sums of copies of irreducible $\mathfrak{S}_n$-modules.
Under the Frobenius characteristic map, $M^\la$ maps to the complete homogeneous symmetric function $h_\la$, whereas the irreducible $\mathfrak{S}_n$-modules maps to the Schur functions. We refer the book \cite{Sag} for the representation theory of the symmetric groups.

\emph{Chromatic symmetric functions} were introduced by Stanley in \cite{S1} as a generalization of the chromatic polynomials of graphs, and they were refined to chromatic \emph{quasi}symmetric functions by Shareshian and Wachs in \cite{SW}.  The Stanley--Stembridge conjecture on chromatic symmetric functions states that the chromatic symmetric function of the incomparability graph of a $(3+1)$-free poset  expands positively as a sum of elementary symmetric functions, that is $e$-positive, and it is refined by Shareshian and Wachs to a conjecture on chromatic quasisymmetric functions.    Work by Guay-Paquet \cite{G-P1} shows that for the $e$-positivity conjecture it is enough to consider the incomparability graphs of \emph{natural unit interval orders} which can each be identified with a Hessenberg function.

For a Hessenberg function $h$, let $G(h)$ be the incomparability graph of the natural unit interval order corresponding to $h$.
In \cite{SW}, Shareshian and Wachs made a conjecture that the regular semisimple Hessenberg varieties of type $A$ and the chromatic quasisymmetric functions are related in the following way, which was  independently proved by Brosnan and Chow \cite{BrosnanChow18}, and Guay-Paquet~\cite{G-P2}.

\begin{theorem}\label{prop:BC_G-P} For a Hessenberg function $h$,
	\[ \sum_k \mathrm{ch} H^{2k}(\Hess(S,h))\,t^k=\omega X_{G(h)}(\mathbf x, t)\,,\]
	where $\mathrm{ch}$ is the Frobenius characteristic map and $\omega$ is the involution on the symmetric function algebra, which sends elementary symmetric functions to complete homogeneous symmetric functions; $\omega(e_i)=h_i$.
\end{theorem}

Theorem~\ref{prop:BC_G-P} plays a crucial role in connecting geometry, combinatorics, and representation theory.
The $e$-positivity conjecture on chromatic (quasi)symmetric functions  translates into the statement that $\mathrm{ch} H^*(\Hess(S,h))$ is positively expanded as a sum of complete homogeneous symmetric functions.

\begin{conjecture}[Stanley--Stembridge \cite{SS, S1}, Shareshian--Wachs \cite{SW}]\label{conj:h-positivity}
	Under the dot action of the symmetric group, the $(2k)$th cohomology space $H^{2k}(\Hess(S,h))$ is decomposed as a direct sum of permutation modules $M^\la$ for each degree $2k$.
\end{conjecture}
Conjecture~\ref{conj:h-positivity} is proved to be true for some special cases; see \cite{S1, ChoHong, GS, DW, CH, HP, HNY}. Even in those cases, the geometric construction of permutation module decomposition for $H^{*}(\Hess(S,h))$ has not been provided except for trivial cases.
Our work in this paper is motivated by Conjecture~\ref{conj:h-positivity} particularly on the permutation module decomposition of $H^{*}(\Hess(S,h))$. For the construction of permutation submodules in   $H^{*}(\Hess(S,h))$, it is essential to choose a good basis that behaves well under the dot action. This is not trivial work and there are only some sets of cohomology classes conjectured to form  good bases of  $H^{*}(\Hess(S,h))$ for certain $h$ or for some permutation modules of certain type: the Erasing marks conjecture due to Chow (personal communication, cf.~\cite{Chow_erasing}) is for toric varieties  $\mathcal H_n$ for example. We first consider a basis of $H_T^{*}(\Hess(S,h))$ that  naturally arises from the geometric structure of  $\Hess(S,h)$ and examine the properties of the basis elements. We then compute the dot actions using our bases and finally apply the results to construct a permutation module decomposition of the cohomology space of the permutohedral variety. The rest of this section summarizes our work.

Any regular semisimple Hessenberg variety $\Hess(S,h)$ admits an affine paving, which is a Bia{\l}ynicki--Birula decomposition (see~\cite{DPS_Hessenberg_var}). Considering the closure $\Omega_{w,h}$ of each minus cell $\Omega_{w,h}^{\circ}$, we obtain an equivariant cohomology class $\swh{w}$ for each $w \in \mathfrak{S}_n$ (see Definition~\ref{def_swh_classes}). The set $\{ \swh{w} \mid w \in \mathfrak{S}_n\}$ forms a basis of the equivariant cohomology ring $H^{\ast}_T(\Hess(S,h))$.
Our basis satisfies nice properties which are already considered in other known results.
Indeed, when $\Hess(S,h) = \flag(\C^n)$, our basis coincides with the basis constructed by Tymoczko~\cite{T1}. When $\Hess(S,h)$ is a permutohedral variety, which is a smooth projective toric variety, then our basis is the  `canonical basis' provided by Pabiniak and Sabatini~\cite{PabiniakSabatini}. Moreover, the basis $\swh{w}$ is a `flow-up basis' studied by Teff~\cite{Teff_DDO,Teff_thesis} (see Proposition~\ref{prop_property_of_swh}).

We identify the class $\swh{w}$ as an element of $\bigoplus_{v \in \mathfrak S_n} \C[t_1,\dots,t_n]$ using GKM theory. Then, to
study the dot action of $\mathfrak{S}_n$ on each class $\swh{w}$, we have to specify the \textit{support} of $\swh{w}$, the set of elements $v \in \mathfrak{S}_n$ such that $\swh{w}(v) \neq 0$. We note that the support $\swh{w}$ is the same as the fixed point set $\Omega_{w,h}^T$ (see Proposition~\ref{prop_property_of_swh}). One of the primary goals of this manuscript is to present an explicit description of the support of $\swh{w}$ in terms of $w$ and~$h$. In order to introduce our result, we prepare some terminologies (see Section~\ref{sec_description_minus_cell} for precise definitions).
For a positive integer~$n$, we use $[n]$ to denote the set $\{1, 2, \dots, n\}$. Let $h \colon [n] \to [n]$ be a Hessenberg function and $w \in \mathfrak{S}_n$. We define a directed graph $\Gwh$ with the vertex set $[n]$ such that for each pair of indices $1 \leq j < i \leq n$, there is an edge $j \to i$ in the graph $\Gwh$ if and only if $j < i \leq h(j)$ and $w(j) < w(i)$. For example, when $n = 5$,  the graph $G_{15342,(3,3,4,5,5)}$ is given as follows:
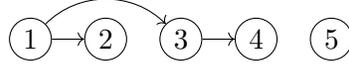
\begin{center}
	\begin{tikzpicture}
	\foreach \x in {1,2,3,4,5}
	\node[circle,draw,inner sep=0pt,text width=5mm,align=center] (\x) at (\x,0) {\x};
	\draw[->] (1)--(2);
	\draw[->] (1) to [bend left = 45] (3);
	\draw[->] (3)--(4);
	\end{tikzpicture}
\end{center}

We say a vertex $i$ is \textit{reachable} from $j$ in the graph $G_{w,h}$ if there exists a directed path from $j$ to $i$.
For subsets $A = \{1 \leq a_1 < \dots < a_k \leq n \}$ and $B = \{1 \leq b_1 < \dots < b_k\leq n\}$ of $[n]$, we say that $A$ is \textit{reachable} from $B$ in the graph $G_{w,h}$ if there exists a permutation $\sigma \in \mathfrak{S}_k$ such that $a_{\sigma(d)}$ is reachable from $b_d$ for all $d = 1,\dots,k$ in the graph $\Gwh$.
For example, $\{3,4\}$ is reachable from $\{1,3\}$, whereas $\{5\}$ is not reachable from $\{3\}$ in the  graph $G_{15342,(3,3,4,5,5)}$. Using these expressions, we state our first main theorem.
\begin{thmx}[Theorem~\ref{thm_reachable_nonzero_minor}]\label{thmx:A}
	Let $h$ be a Hessenberg function and $w \in \mathfrak{S}_n$. An element $u \in \mathfrak{S}_n$ is in $\supp(\swh{w})$ if and only if $\{ w^{-1}(u(1)),\dots,w^{-1}(u(j))\}$ is reachable from $[j]$ in the graph $\Gwh$ for all $j = 1,\dots,n$.
\end{thmx}

The GKM graphs of regular semisimple Hessenberg varieties are subgraphs of the Bruhat graphs of the symmetric group having the same vertex set $\mathfrak{S}_n$. For two permutations $v, w \in \mathfrak{S}_n$ such that
$v=w s_{j,k}$ and $\ell(w)>\ell(v)$, we use $w \rightarrow v$ ($w \dasharrow v$, respectively) if  $v$ and  $w$ are connected (not connected, respectively) by an edge in the GKM graph. Here, $s_{j,k}$ is a transposition that exchanges $j$ and $k$.

We  follow the same line of the computation of the $\frak S_n$-action on the equivariant cohomology space $H_T^*(G/B)$ of the full flag variety done by Brion  \cite{Br1} to compute the dot action on $H_T^*(\Hess(S,h))$ for any Hessenberg variety. To handle the minus cells $\Owho{w}$ with general $h$, we must perform delicate analyses. The main results are stated as follows.
When $w \rightarrow s_iw$, define $\mathcal A_{s_i,w}$ to be the set of all $u \in    \Omega_{s_{i}w,h}^T \cap \Omega_w^T$ such that $\dim (\Omega_u^{\circ} \cap \Omega_{s_{i}w,h})  =\dim \Omega_{w,h}$ and $u \dashrightarrow s_iu$. For $u \in \mathcal A_{s_i,w}$,  $\tau_u$ ($\tau_{s_iu}$, respectively) is the equivariant class induced by the closure of $\Omega_u^{\circ} \cap \Omega_{s_{i}w,h}$ ($\Omega_{s_iu}^{\circ} \cap \Omega_{s_{i}w,h}$, respectively).

\begin{thmx}[Propositions~\ref{prop: edge deleted} and~\ref{prop: edge remaining 2}]\label{thmx:B} Let $h: [n] \rightarrow [n]$ be a Hessenberg function. Let $w \in \mathfrak S_n$ be a permutation and let $s_i = s_{i,i+1}\in \mathfrak S_n$ be a simple reflection.
	\begin{enumerate}
		\item If $ w  \dashrightarrow s_iw$ or $s_iw \dasharrow w$, then $s_i \cdot\sigma_{w,h} =\sigma_{s_i w, h}$.
		\item
		\begin{enumerate}
			\item If $s_i w \rightarrow w$, then $s_i\cdot \sigma_{w,h} - \sigma_{w,h}=0$.
			\item If $w \rightarrow s_iw$, then $\left( s_i \cdot\sigma_{w,h} + \sum_{v \in \mathcal A_{s_i,w}} \tau_{s_iv} \right) - \left(\sigma_{w,h} + \sum_{v \in \mathcal A_{s_i,w} } \tau_v \right)= (t_{i+1}- t_{i})\sigma_{s_i w, h}$, and the intersection $\mathcal A_{s_i,w} \cap s_i\mathcal A_{s_i,w}$ is empty.
		\end{enumerate}
	\end{enumerate}
\end{thmx}

As previously mentioned, Conjecture~\ref{conj:h-positivity} is true for the permutohedral varieties $\mathcal H_n$; however, no geometric explanation (construction) was provided before. We apply Theorem~\ref{thmx:A} and Theorem~\ref{thmx:B} to $H_T^{*}(\mathcal H_n)$ to construct an explicit permutation module decomposition of each  $H^{2k}(\mathcal H_n)$ resolving a question by Stembridge posed in \cite{Ste}.
We present a nice set of permutations of order ${n-1}\choose{k}$;
that is the number of permutation modules appearing in the decomposition of $H^{2k}(\mathcal H_n)$. We then construct one permutation submodule for each selected permutation, where we use the `erasing marks' due to Chow to form a submodule of the right dimension. We use the same notation $\sigma_w\in H^*(\mathcal H_n)$ for the image of the class $\sigma_w\in H_T^*(\mathcal H_n)$ by the  quotient map.

It is well known that the number of descents $\mathrm{des}(w)\colonequals \# \{ i \mid w(i)>w(i+1)\}$ of a permutation $w\in \frak{S}_n$ determines the degree of the corresponding cohomology class of $w$; $\sigma_{w} \in H_T^{2k}(\mathcal H_n)$ if  $\mathrm{des}(w)=k$, where we use $\sigma_w$ instead of $\sigma_{w, h}$ because we fix a Hessenberg function $h$ defined by $h(i) =i+1$ for $i<n$.
We compute the dot actions given in Theorem~\ref{thmx:B} in more explicit ways for each class  $\sigma_{w} \in H_T^{2k}(\mathcal H_n)$, in which the descents of $w$ play important roles; see Proposition~\ref{prop:s_i action on S_n(a)}  and Proposition~\ref{prop_si_action_on_sigma_i}.
The final step in the construction of permutation module decomposition of $H^*(\mathcal H_n)$ is to choose an appropriate module generator of each permutation submodule. Let $\mathcal G_k$ be the set of $w \in \mathfrak S_n$ with $\mathrm{des}(w)=k$  such that  $\supp(\sigma_w)$ contains $w_0$, the longest element in $\mathfrak S_n$. Then $|\mathcal G_k| = {n-1 \choose k}$ so $\sigma_w\in H^*(\mathcal H_n)$, $w\in\mathcal G_k$, are reasonable candidates for the generators; however, their isotropy subgroups do not make the right dimension. We `erase marks (descents)' of each candidates by   symmetrization, where `erasing marks' is defined by Chow in his conjecture on bases for $H_T^*(\mathcal H_n)$ (see~\cite{Chow_erasing}).
The \emph{erasing descents} from a permutation $w$ erases the descent $1$ if $1$ is a descent of $w$, and descent $d$ if both $(d-1)$ and $d$ are descents of $w$. We then symmetrize $\sigma_w$ to obtain  $\widehat{\sigma}_w=\sum_{u\in\mathfrak{S}_w} u\sigma_w$, where $\mathfrak{S}_w$ is the subgroup  of permutations that permute the elements in each newly created block by erasing (see Definition~\ref{def:generator}).
As a $\mathbb C \mathfrak S_n$-module,
$H^{2k}(\mathcal{H}_n)$ is decomposed as permutation submodules:

\begin{thmx}[Theorem~\ref{thm:erasing conjecture}]\label{thmx:E} For each $k$,
	$$H^{2k}(\mathcal{H}_n)=\bigoplus_{w\in \mathcal G_k} \mathbb C \mathfrak S_n (\widehat{\sigma}_w).$$
\end{thmx}

This paper is organized as follows. In Section~\ref{sec:preliminaries}, we recall  Bia{\l}ynicki-Birula decompositions on regular semisimple Hessenberg varieties and the basis of the equivariant cohomology of a regular semisimple Hessenberg variety constructed by them.
In Section~\ref{sec_description_minus_cell}, we provide an explicit description of the support of the basis we constructed in terms of the reachability of a certain acyclic directed graph. The dot actions of the symmetric group on the basis elements are calculated in Section~\ref{sec:hong}.
Sections~\ref{sec:cho} and~\ref{sec:lee_permutohedral} are devoted to the construction of a decomposition of the cohomology space of the permutohedral variety into permutation modules.

\section{Preliminaries}\label{sec:preliminaries}
\subsection{Hessenberg varieties and Bia{\l}ynicki-Birula decompositions}
\label{sec_preliminaries}

In this subsection, we first review some properties of the flag variety and then we present the definition of Hessenberg varieties and Bia{\l}ynicki-Birula decompositions on them.

Let $G$ be the general linear group $\GL_n(\C)$ and let $B$ be the set of upper triangular matrices in $G$. Let $T$ be the set of diagonal matrices in $G$, i.e., $T \cong (\Cstar)^n$. We denote the set of lower triangular matrices in $G$ by $B^-$. The flag variety $\flag(\C^n)$ is isomorphic to the quotient space~$G/B$:
\[
\flag(\C^n) = \{ V_{\bullet} = (\{0\} \subsetneq V_1 \subsetneq V_2 \subsetneq \cdots \subsetneq V_{n-1} \subsetneq \C^n) \mid \dim_{\C} V_i = i \quad \text{ for all }  1\leq i \leq n\} \cong G/B.
\]
For $g \in G$, the corresponding element in $\flag(\C^n)$ to $gB \in G/B$ is given by the column vectors~$v_1,\dots,v_n$ of the matrix $g$:
\[
gB = (\{0\} \subsetneq \langle v_1 \rangle \subsetneq \langle v_1, v_2 \rangle \subsetneq \cdots \subsetneq \langle v_1,\dots,v_i \rangle \subsetneq \cdots \subsetneq \C^n).
\]

The left multiplication of $T$ on $G$ induces an action of $T$ on the flag variety $\flag(\C^n)$, which is indeed the same as that induced by  coordinate-wise multiplication of $T$ on $\C^n$.
The set of $T$-fixed points in $\flag(\C^n)$ is identified with the set $\mathfrak{S}_n$ of permutations on $[n] \colonequals \{1,\dots,n\}$, which is the Weyl group of $G$ (cf.~\cite[Lemma~2 in~\S10.1]{Fulton_YoungT}). More precisely, if we denote the standard basis vectors in $\C^n$ by $e_1,\dots,e_n$, for a permutation $w = w(1) w(2) \cdots w(n) \in \mathfrak{S}_n$, the corresponding element in $\flag(\C^n)$ is a coordinate flag
\begin{equation}\label{eq_wB}
\dot{w}B = (\{0\} \subsetneq \langle e_{w(1)} \rangle \subsetneq \langle  e_{w(1)}, e_{w(2)} \rangle \subsetneq
\cdots \subsetneq \langle e_{w(1)}, \dots, e_{w(i)} \rangle \subsetneq \cdots \subsetneq \C^n)\,.
\end{equation}
Here, $\dot{w}$ is the column permutation matrix of $w$, i.e., it has $1$ on $(w(i),i)$-entry and the others are all zero.

The flag variety $G/B$ admits a cell decomposition called the \textit{Bruhat decomposition}
\begin{equation}
G/B = \bigsqcup_{w \in \mathfrak{S}_n} B\dot{w}B/B = \bigsqcup_{w \in \mathfrak{S}_n} B^- \dot{w} B/B.
\end{equation}
For each $w \in \mathfrak{S}_n$, we denote by
\begin{equation}\label{eq_def_Xwo}
\Xwo{w} = B\dot{w}B/B \quad \text{ and } \quad \Owo{w} = B^- \dot{w}B/B
\end{equation}
  the \textit{Schubert cell} and the \textit{opposite Schubert cell}, respectively. Both are indeed affine cells of dimension $\ell(w)$ and codimension $\ell(w)$, respectively.
Here, $\ell(w)$ is the \textit{length} of $w$ defined as follows. The symmetric group $\mathfrak{S}_n$ is generated by \textit{simple reflections} $s_i$, which are the adjacent transpositions exchanging $i$ and $i+1$ for $i = 1,\dots,n-1$. Any element $w \in \mathfrak{S}_n$ can be written as a product of generators $w = s_{i_1} s_{i_2} \cdots s_{i_k}$ for $i_1,\dots,i_k \in [n-1]$. If $k$ is minimal among all such expressions for $w$, then $k$ is called the \textit{length} of $w$. We note that the length $\ell(w)$ is the same as the number of inversions of $w$, that is,
\begin{equation}\label{eq_length_and_inversions}
\ell(w) = \# \{(j,i) \mid 1 \leq j < i \leq n, w(j) > w(i)\}.
\end{equation}

To define Hessenberg varieties, we first present the definition of Hessenberg functions.
A function $h \colon [n] \to [n]$ is called  a \textit{Hessenberg function} if
\begin{itemize}
	\item it is weakly increasing: $1 \leq h(1) \leq h(2) \leq \cdots \leq h(n) \leq n$, and
	\item $h(i) \geq i$ for all $1 \leq i \leq n$.
\end{itemize}
\begin{definition}\label{def_Hess_var}
	Let $h$ be a Hessenberg function and let $L$ be a linear operator.
	Then the \textit{Hessenberg variety} $\Hess(L,h)$ is a subvariety of the flag variety defined as follows:
	\[
	\Hess(L,h) = \{ V_{\bullet} \in \flag(\C^n) \mid L V_{i} \subset V_{h(i)} \quad \text{ for all }1 \leq i \leq n \}.
	\]
	When the linear operator $L$ is regular (semisimple, or nilpotent) then we call $\Hess(L, h)$ regular (semisimple, or nilpotent).
\end{definition}

In this article, we  focus on \textit{regular semisimple Hessenberg varieties}, i.e., the Jordan canonical form of the linear operator $L$ is the diagonal matrix
\begin{equation}\label{eq_Jordan_canonical_form}
S\colonequals
\begin{pmatrix}
c_1 & 0 & \cdots & 0 \\
0 & c_2 & \cdots & 0 \\
\vdots & \vdots & \rotatebox[origin=c]{-45}{$\cdots$} & \vdots \\
0 & 0 & \cdots & c_n
\end{pmatrix}
\end{equation}
with distinct eigenvalues $c_1,c_2,\dots,c_n$.
There is an isomorphism
\[
\Hess(L,h) \cong \Hess(gLg^{-1}, h)
\]
for any  $g \in G$ by sending $V_{\bullet}$ to $g V_{\bullet}$. Accordingly, we may assume that the linear operator~$L$ is of the form~\eqref{eq_Jordan_canonical_form}. To emphasize, we consider only regular semisimple Hessenberg varieties in this manuscript, and we denote the linear operator by $S$.

\begin{example}
	\begin{enumerate}
		\item Let $h = (n,n,\dots,n)$. Then $\Hess(S,h)$ is the flag variety $\flag(\C^n)$.
		\item Let $h = (2,3,4,\dots,n,n)$. Then $\Hess(S,h)$ is a toric variety $\mathcal H_n$ called \textit{permutohedral variety}, with the fan consisting of Weyl chambers of the root system of type $A_{n-1}$ \textup{(}see~\cite[Theorem~11]{DPS_Hessenberg_var}\textup{)}.
	\end{enumerate}
\end{example}

The torus action on the flag variety preserves a regular semisimple Hessenberg variety.
Using this torus action, De Mari, Procesi, and Shayman~\cite{DPS_Hessenberg_var} computed the Poincar\'e polynomial of $\Hess(S,h)$.
\begin{proposition}[{\cite{DPS_Hessenberg_var}}]\label{prop_Betti_numbers}
	Let $\Hess(S,h)$ be a regular semisimple Hessenberg variety. Then $\Hess(S,h)$ is smooth of $\C$-dimension $\sum_{i=1}^n (h(i) - i)$. The Poincar\'e polynomial $\mathscr{P}\textup{oin}(\Hess(S,h), q)$ is given by
	\[
	\mathscr{P}\textup{oin}(\Hess(S,h) ,q )
	\colonequals \sum_{k \geq 0} \dim_{\mathbb{C}} H^{k}(\Hess(S,h))  q^k
	= \sum_{w \in \mathfrak{S}_n} q^{2\ell_h(w)}.
	\]
	Here, the number $\ell_h(w)$ is defined by
	\begin{equation}\label{eq_def_of_lhw}
	\ell_h(w) \colonequals \# \{ (j,i) \mid 1 \leq j < i \leq h(j), w(j) > w(i)\}.
	\end{equation}
\end{proposition}
We notice that the number $\ell_h(w)$ is the same as $\ell(w)$ if $h = (n,\dots,n)$ (see~\eqref{eq_length_and_inversions}).
To prove the previous proposition, De Mari, Procesi, and Shayman applied the theory of Bia{\l}ynicki-Birula~\cite{BB} to Hessenberg varieties.
We briefly review their proof. The Hessenberg variety is $T$-stable and the fixed point set is again identified with $\mathfrak{S}_n$ as in the case of the flag variety  (see~\eqref{eq_wB}).
Fix a homomorphism $\lambda \colon \Cstar \to T, t \mapsto \textup{diag}(t^{\lambda_1},t^{\lambda_2},\dots,t^{\lambda_n})$ determined by integers $\lambda_1 > \lambda_2 > \cdots > \lambda_n$, that is, $\lambda$ is a regular dominant weight. For each $w \in \mathfrak{S}_n,$ consider the \textit{plus cells}  and \textit{minus cells}.
\begin{align}
X_{w,h}^{\circ} &\colonequals \left \{ gB \in \Hess(S,h) \mid \lim_{t \to 0} \lambda(t) gB = \dot{w} B \right\},  \\
\Owho{w} &\colonequals \left \{ gB \in \Hess(S,h) \mid \lim_{t \to \infty} \lambda(t) gB = \dot{w} B \right\}. \label{eq_def_m_cell}
\end{align}
Then, both $X_{w,h}^{\circ}$ and $\Owho{w}$ are isomorphic to affine spaces and they define the plus and minus decompositions. We note that when we take $h = (n,\dots,n)$ then $\Hess(S,h)$ is the flag variety, and moreover, we have  $X_{w,h}^{\circ} = \Xwo{w}$ and $\Owho{w} = \Owo{w}$ (cf.~\eqref{eq_def_Xwo}). Counting the numbers of positive and negative weights in the  tangent space of $\Hess(S,h)$ at $\dot{w}B$, we obtain
\begin{equation}\label{eq_dim_codim_X_Omega}
\dim_{\C} X_{w,h}^{\circ} = \ell_h(w)  = \dim_{\C}(\Hess(S,h)) - \dim_{\C} \Owho{w}
\end{equation}
which implies the aforementioned proposition.

\begin{example}
	Suppose that $n = 3$ and $h = (2,3,3)$.
	For a permutation $w = 231$, there are two inversions on locations $(1,3)$ and $(2,3)$ because $w(1) > w(3)$ and $w(2) > w(3)$. However, because of the condition $j \leq h(i)$, we only count $(2,3)$ to compute $ \ell_h(w)$. Therefore, we obtain $ \ell_h(w) = \# \{(2,3)\} = 1$. By a similar computation, we present $ \ell_h(w)$.
	\begin{center}
		\begin{tabular}{c|cccccc}
			\toprule
			$w$ & $123$ & $132$ & $213$ & $231$ & $312$ & $321$ \\
			\midrule
			$ \ell_h(w)$ & $0$ & $1$ & $1$ & $1$ & $1$ & $2$ \\
			\bottomrule
		\end{tabular}
	\end{center}
\end{example}

\begin{remark}\label{rmk_not_cellular}
	It is known that the closure of a Schubert cell is a union of other Schubert cells:
	\[
	X_w\colonequals \overline{\Xwo{w}} = \bigsqcup_{v \leq w } \Xwo{v}, \qquad
	\Omega_w \colonequals \overline{\Owo{w}} = \bigsqcup_{v \geq w } \Owo{v},
	\]
where $\leq$ is the Bruhat order. We refer the reader to \cite[\S 10.2 and \S10.5]{Fulton_YoungT} for the definition of the Bruhat order and the above description. As explained in~\cite[Remark in Section~5]{DeMari_Shayman_88}, the boundary of a cell of dimension $k$ in an arbitrary regular semisimple Hessenburg variety is not always contained in a union of cells of dimension $\leq (k-1)$.
\end{remark}

\begin{remark}\label{rmk_every_hess_affine_paving}
	The Hessenberg varieties are defined for any reductive linear algebraic group and any linear operator. Particularly, it is proved that Hessenberg varieties corresponding to a reductive linear group  and a regular element are paved by affine spaces (see~\cite{Precup_affine, Tymoczko_paving} and references therein).
\end{remark}

\subsection{Basis of the equivariant cohomology}\label{sec:basis_of_equivariant_cohomology}

In this subsection, we consider the torus action on regular semisimple Hessenberg varieties, and then define a basis $\{ \swh{w} \mid w\in \mathfrak{S}_n\}$ of the  equivariant cohomology ring in Definition~\ref{def_swh_classes}.
Moreover, we study their combinatorial properties in Proposition~\ref{prop_property_of_swh}.

We first briefly recall the GKM theory from~\cite{GKM98, T2}.
Let $T$ be a complex torus $(\Cstar)^n$. Let $\Sym = \Z[t_1,\dots,t_n]$ be the symmetric algebra over $\Z$ of the character group $M \colonequals \Hom(T,\Cstar)$.
Let $X$ be a complex projective variety with an action of $T$. Then $X$ is called a \textit{GKM space} if it satisfies the following three conditions:
\begin{enumerate}
	\item $X$ has finitely many $T$-fixed points,
	\item $X$ has finitely many (complex) one-dimensional $T$-orbits, and
	\item $H^{\text{odd}}(X)$ vanishes.\footnote{In the original definition, the third condition (3) is `$X$ is equivariantly formal', which is a technical property that holds when $H^{\text{odd}}(X)$ vanishes.}
\end{enumerate}
Toric varieties are GKM spaces, and the flag variety, Grassmannians, and Schubert varieties are also GKM spaces.
Because there are finitely many fixed points, the closure $\overline{\mathcal{O}}$ of a one-dimensional orbit $\mathcal{O}$ is homeomorphic to $\C P^1$.

The \textit{GKM graph} $\Gamma = (V, E, \alpha)$ of a GKM space is a directed graph with label $\alpha$ on edges defined as follows:
\begin{enumerate}
	\item $V = X^T$ is the set of $T$-fixed points of $X$.
	\item $ (v \to w) \in E$ if and only if there exists a one-dimensional orbit $\mathcal{O}_{v,w}$ whose closure $\overline{\mathcal{O}_{v,w}}$ contains $v$ and $w$.
	\item We label each edge $v \to w $ with the weight $ \alpha({v \to w }) \in \Sym$ of the $T$-action corresponding to  the closure $\overline{\mathcal{O}_{v,w}}$  at the fixed point  $v$.
\end{enumerate}
According to the second condition, we have that $(v \to w)  \in E$ if and only if $(w \to v) \in E$. We consider an orientation on edge because of the labelling $\alpha(v \to w)$ on each edge $v \to w$. Indeed,
\[
\alpha({v \to w }) = - \alpha({w \to v}).
\]
Moreover, at each fixed point $v \in V = X^T$, the set $\{\alpha(v \to w) \mid (v \to w) \in E \}$ consists of the weights of $T_vX$.

\begin{example}\label{example_Hess_GKM}
	Let $h \colon [n] \to [n]$ be a Hessenberg function.
	The regular semisimple Hessenberg variety $\Hess(S,h)$ is a GKM space (cf.~\cite[Section~III]{DPS_Hessenberg_var} or~\cite[Proposition~5.4(1)]{T2}).
	The GKM graph $\Gh = (V,E,\alpha)$ of $\Hess(S,h)$ is given as follows:
	\begin{enumerate}
		\item $V = \mathfrak{S}_n$.
		\item $(w \to w s_{j,i}) \in E$ if and only if $j < i \leq h(j)$.
		\item $\alpha(w \to w s_{j,i}) = t_{w(i)} - t_{w(j)}$.
	\end{enumerate}
	Here, $s_{j,i}$ is the  transposition in $\mathfrak{S}_n$, exchanging $j$ and $i$.
	Because of the  equality $s_{j,i} w =  w s_{w^{-1}(j), w^{-1} (i)}$,
	the second condition is equivalent to the following:
	\begin{itemize}
		\item $(w \to s_{j,i}w )\in E$ if and only if $w^{-1}(j) < w^{-1}(i) \leq h(w^{-1}(j))$. \\
		With this description, $\alpha(w \to s_{j,i} w) = t_{i} - t_{j}$.
	\end{itemize}
	For example, assume that $n = 3$ and $h = (2,3,3)$. For $w = 132$, there are two edges starting at $w$.
	\[
	(w \to w s_{1}) = (132 \to 312), \quad (w \to w s_2) = (132 \to 123).
	\]
	For each edge, the labeling is given by $\alpha{(132 \to 312)} = t_{3}- t_1$ and $\alpha{(132 \to 123)} = t_2 - t_3$.
There also exist edges $(312 \to 132)$ and $(123 \to 132)$ having opposite labels.
	In Figure~\ref{fig_GKM_n3}, we present GKM graphs of regular semisimple Hessenberg varieties for $n = 3$, and $h = (2,3,3)$, $h = (3,3,3)$. The indegree and outdegree of each vertex coincide with the $\C$-dimension of $\Hess(S,h)$.
\end{example}
\begin{figure}
	\begin{tikzpicture}[scale = 0.5,
arrowmark/.style 2 args={decoration={markings,mark=at position #1 with {\arrow[line width = 0.5pt,scale=1]{#2}}}}]

	\node[shape = circle, fill=black, scale = 0.2] (312) at (30:2) {};
	\node[shape = circle, fill=black, scale = 0.2] (321) at (90:2) {};
	\node[shape = circle, fill=black, scale = 0.2] (231) at (150:2) {};
	\node[shape = circle, fill=black, scale = 0.2] (132) at (-30:2) {};
	\node[shape = circle, fill=black, scale = 0.2] (123) at (-90:2) {};
	\node[shape = circle, fill=black, scale = 0.2] (213) at (-150:2) {};
	
	\node[left] at (213) {$213$};
	\node[left] at (231) {$231$};
	\node[below] at (123) {$123$};
	\node[right] at (132) {$132$};
	\node[right] at (312)  {$312$};
	\node[above] at (321) {$321$};

\foreach \x/\y in {123/213, 312/321}{
\draw[postaction={decorate},
    arrowmark={.3}{>},
	arrowmark={.7}{<}] (\x) to (\y);
}

\foreach \x/\y in {123/132, 231/321}{
\draw[postaction={decorate},
    arrowmark={.3}{>},
	arrowmark={.7}{<}, blue, double] (\x) to (\y);
}

\foreach \x/\y in {132/312, 213/231}{
\draw[postaction={decorate},
    arrowmark={.3}{>},
	arrowmark={.7}{<}, red, triple line] (\x) to (\y);
\draw[red] (\x) to (\y);
\draw[red] (\y) to (\x);

}
	\node[below of = 2cm] at (123) {$h = (2,3,3)$};
	\end{tikzpicture}
	\hspace{1cm}
	\begin{tikzpicture}[scale = 0.5,
arrowmark/.style 2 args={decoration={markings,mark=at position #1 with {\arrow[line width = 0.5pt,scale=1]{#2}}}}]
	\node[shape = circle, fill=black, scale = 0.2] (312) at (30:2) {};
	\node[shape = circle, fill=black, scale = 0.2] (321) at (90:2) {};
	\node[shape = circle, fill=black, scale = 0.2] (231) at (150:2) {};
	\node[shape = circle, fill=black, scale = 0.2] (132) at (-30:2) {};
	\node[shape = circle, fill=black, scale = 0.2] (123) at (-90:2) {};
	\node[shape = circle, fill=black, scale = 0.2] (213) at (-150:2) {};
	
	\node[left] at (213) {$213$};
	\node[left] at (231) {$231$};
	\node[below] at (123) {$123$};
	\node[right] at (132) {$132$};
	\node[right] at (312)  {$312$};
	\node[above] at (321) {$321$};

\foreach \x/\y in {123/213, 312/321, 132/231}{
\draw[postaction={decorate},
    arrowmark={.3}{>},
	arrowmark={.7}{<}] (\x) to (\y);
}

\foreach \x/\y in {123/132, 231/321, 213/312}{
\draw[postaction={decorate},
    arrowmark={.3}{>},
	arrowmark={.7}{<}, blue, double] (\x) to (\y);
}

\foreach \x/\y in {132/312, 213/231, 123/321}{
\draw[postaction={decorate},
    arrowmark={.3}{>},
	arrowmark={.7}{<}, red, triple line] (\x) to (\y);
\draw[red] (\x) to (\y);
\draw[red] (\y) to (\x);

}

	\node[below of = 2cm] at (123) {$h = (3,3,3)$};
	\end{tikzpicture}
	\hspace{1cm}
	\begin{tikzpicture}[
arrowmark/.style 2 args={decoration={markings,mark=at position #1 with {\arrow[line width = 0.5pt,scale=1]{angle 90}}}}]
	\draw[postaction={decorate}, arrowmark={.55}{}] (0,2)--(0.7,2);
	\node at(2,2) {$\pm (t_2 - t_1)$};
	\draw[blue, double, postaction={decorate}, arrowmark={.55}{}] (0,1.3)--(0.7,1.3);
	\node at (2,1.3) {$\pm (t_3-t_2)$};
	\draw[postaction={decorate},
    arrowmark={.55}{>}, red, triple line] (0,0.6)--(0.7,0.6);
	\draw[red] (0,0.6)--(0.7,0.6);
	\node at (2,0.6) {$\pm (t_3 - t_1)$};
	\end{tikzpicture}
	\caption{GKM graphs of $\Hess(S,h)$ for $h = (2,3,3)$ and $h = (3,3,3)$.} \label{fig_GKM_n3}
\end{figure}
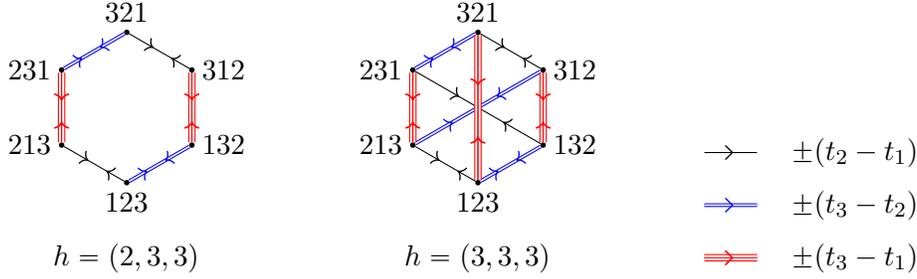

Goresky--Kotwitz--MacPherson~\cite{GKM98} provides a description of the equivariant cohomology ring of a GKM space.
\begin{theorem}[{\cite{GKM98}}]\label{thm_GKM}
	Let $T = (\Cstar)^n$ and let $(X,T)$ be a GKM space. Let $\Gamma = (V, E, \alpha)$ be the GKM graph. Then,
	\[
	H^{\ast}_T(X;\C) \cong
	\left\{
	(p(v)) \in \bigoplus_{v \in V} \C[t_1,\dots,t_n] ~\Bigg|~ \alpha{(v \to w)} |(p(v) - p(w)) \quad \text{ for all }(v \to w) \in E
	\right\}.
	\]
\end{theorem}

Now we concentrate on the equivariant cohomology ring of a regular semisimple Hessenberg variety.
One may wonder whether there exist nice classes which form a $H^{\ast}(BT)$-basis of the equivariant cohomology ring.
This paper studies a basis defined using the minus cell decomposition:
\begin{equation}\label{eq_decomp_Hess}
\Hess(S,h) = \bigsqcup_{w \in \mathfrak{S}_n} \Owho{w}.
\end{equation}
When $h = (n,\dots,n)$, each cell $\Owho{w}$ agrees with the opposite Schubert cell $\Owo{w}$.

The closure $\Owh{w}  \colonequals \overline{\Owho{w}}$ of a minus cell $\Owho{w}$ might not be smooth even though any regular semisimple Hessenberg variety $\Hess(S,h)$ is smooth.
Indeed, Schubert varieties (and opposite Schubert varieties) are not always smooth.
The closure  $\Owh{w} $ defines a class $[\Owh{w}]$ in the equivariant Chow ring $A^{\ast}_T(\Hess(S,h))$, which is graded by codimension. (See~\cite{Br1} for more details on \textit{equivariant Chow rings}.)
Furthermore, since the $T$-fixed points are isolated and $\Hess(S,h)$ is smooth, the cycle map is an isomorphism by~\cite[Corollary~2 in Section~3.2]{Br1}:
\begin{equation}\label{eq_cycle_map_is_iso}
cl_{\Hess(S,h)}^T \colon A_{T}^{\ast}(\Hess(S,h))_{\mathbb{Q}} \stackrel{\cong}{\longrightarrow} H_T^{2\ast}(\Hess(S,h); \mathbb{Q}).
\end{equation}
 Using the cycle map, we provide the following definition.
\begin{definition}\label{def_swh_classes}
	Let $\Hess(S,h)$ be a regular semisimple Hessenberg variety.
	For $w \in \mathfrak{S}_n$, we define an equivariant class $\swh{w} \in H^{2 \ell_h(w)}_{T}(\Hess(S,h);\C)$ to be the
	image of the class $[\Owh{w}] \in  A^{\ell_h(w)}_T(\Hess(S,h))_{\mathbb{Q}}$ under the cycle map~\eqref{eq_cycle_map_is_iso}.
\end{definition}
Because the Hessenberg variety is decomposed by the minus cells as described in~\eqref{eq_decomp_Hess},
the equivariant classes $\swh{w}$ form a basis of the equivariant cohomology.
\begin{proposition}\label{prop_swh_form_a_basis}
	Let $\Hess(S,h)$ be a regular semisimple Hessenberg variety. Then the classes $\{ \swh{w} \mid w  \in \mathfrak{S}_n \} $ form a basis of the equivariant cohomology ring $H^{\ast}_T(\Hess(S,h); \C)$.
\end{proposition}

Let $\Gh = (V, E, \alpha)$ be the GKM graph of a regular semisimple Hessenberg variety $\Hess(S,h)$.
We consider a subgraph $\Gh^o = (V, E^o, \alpha)$, where $E^{o} = \{(v \to w) \in E \mid \ell(v) > \ell(w)\}$. That is, we choose one edge from a pair of edges $v \to w$ and $w \to v$ comparing the lengths of elements.
With abuse of notation, we write $(v \to w) \in \Gh^o$ if $(v \to w)$ is an edge of the graph $\Gh^o$.
We present $\Gh^o$ for $h = (2,3,3)$ and $h = (3,3,3)$ in Figure~\ref{fig_GKM_n3_o}.
\begin{figure}
	\centering
	\begin{tikzpicture}[scale = 0.5,
arrowmark/.style 2 args={decoration={markings,mark=at position #1 with {\arrow[line width = 0.5pt,scale=1]{angle 90}}}}]
	\node[shape = circle, fill=black, scale = 0.2] (312) at (30:2) {};
	\node[shape = circle, fill=black, scale = 0.2] (321) at (90:2) {};
	\node[shape = circle, fill=black, scale = 0.2] (231) at (150:2) {};
	\node[shape = circle, fill=black, scale = 0.2] (132) at (-30:2) {};
	\node[shape = circle, fill=black, scale = 0.2] (123) at (-90:2) {};
	\node[shape = circle, fill=black, scale = 0.2] (213) at (-150:2) {};
	
	\node[left] at (213) {$213$};
	\node[left] at (231) {$231$};
	\node[below] at (123) {$123$};
	\node[right] at (132) {$132$};
	\node[right] at (312)  {$312$};
	\node[above] at (321) {$321$};
	
\foreach \x/\y in {213/123, 321/312}{
\draw[postaction={decorate},
    arrowmark={.55}{}] (\x) to (\y);
}

\foreach \x/\y in {321/231, 132/123}{
\draw[postaction={decorate},
    arrowmark={.55}{}, blue, double] (\x) to (\y);
}

\foreach \x/\y in {231/213, 312/132}{
\draw[postaction={decorate},
    arrowmark={.55}{>}, red, triple line] (\x) to (\y);
\draw[red] (\x) to (\y);
}
	
	\node[below of = 2cm] at (123) {$h = (2,3,3)$};
	\end{tikzpicture}
	\hspace{1cm}		
	\begin{tikzpicture}[scale = 0.5,
arrowmark/.style 2 args={decoration={markings,mark=at position #1 with {\arrow[line width = 0.5pt,scale=1]{angle 90}}}}]
	\node[shape = circle, fill=black, scale = 0.2] (312) at (30:2) {};
	\node[shape = circle, fill=black, scale = 0.2] (321) at (90:2) {};
	\node[shape = circle, fill=black, scale = 0.2] (231) at (150:2) {};
	\node[shape = circle, fill=black, scale = 0.2] (132) at (-30:2) {};
	\node[shape = circle, fill=black, scale = 0.2] (123) at (-90:2) {};
	\node[shape = circle, fill=black, scale = 0.2] (213) at (-150:2) {};
	
	\node[left] at (213) {$213$};
	\node[left] at (231) {$231$};
	\node[below] at (123) {$123$};
	\node[right] at (132) {$132$};
	\node[right] at (312)  {$312$};
	\node[above] at (321) {$321$};

\foreach \x/\y in {213/123, 321/312}{
\draw[postaction={decorate},
    arrowmark={.55}{}] (\x) to (\y);
}

\draw[postaction={decorate},
    arrowmark={.3}{}] (231) to (132);

\foreach \x/\y in {321/231, 132/123}{
\draw[postaction={decorate},
    arrowmark={.55}{}, blue, double] (\x) to (\y);
}
\draw[postaction={decorate},
    arrowmark={.3}{}, blue, double] (312) to (213);

\foreach \x/\y in {231/213, 312/132}{
\draw[postaction={decorate},
    arrowmark={.55}{>}, red, triple line] (\x) to (\y);
\draw[red] (\x) to (\y);
}

\draw[postaction={decorate},
    arrowmark={.3}{>}, red, triple line] (321) to (123);
\draw[red] (321) to (123);

	\node[below of = 2cm] at (123) {$h = (3,3,3)$};
	\end{tikzpicture}
	\caption{$\Gh^o$ for $h = (2,3,3)$ and $h = (3,3,3)$}\label{fig_GKM_n3_o}
\end{figure}
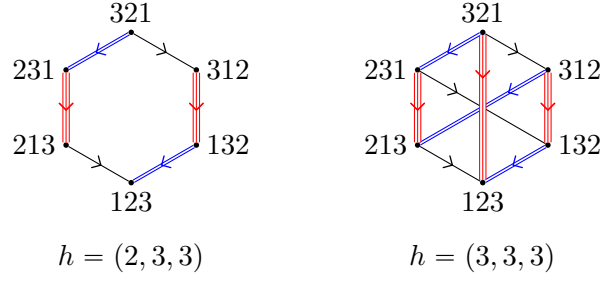
For an equivariant cohomology class $p = (p(v))_{v \in X^T}$, the \textit{support $\supp(p)$ of $p$} is defined to be
\begin{equation}
\supp(p) \colonequals \{ v \in X^T \mid p(v) \neq 0\}.
\end{equation}
From the definition of the equivariant class $\swh{w}$, we obtain subsequent properties.
\begin{proposition}\label{prop_property_of_swh}
	\begin{enumerate}
		\item For each $v \in \Owh{w}^T$, there exists a descending chain $v \to \cdots \to w$ in the graph $\Gh^o$.
		\item The support $\supp(\swh{w})$ of $\swh{w}$ is $\Owh{w}^T$, and $\swh{w}(v)$ is homogeneous of degree $\ell_h(w)$ for each $v \in \Owh{w}^T$.
		\item $\displaystyle\swh{w}(w) = \prod_{(w \to v) \in \Gh^o} \alpha(w \to v)$.
	\end{enumerate}
\end{proposition}
Here and from now on, we identify $v \in \mathfrak S_n$ with $\dot{v}B \in G/B$. \\

In the forthcoming section, we will provide a concrete description of the support~$\supp(\swh{w})$ in Theorem~\ref{thm_support_of_Owh}.
To present a proof of Proposition~\ref{prop_property_of_swh}, we use the following result of Brion~\cite{Br1}:
\begin{lemma}[{\cite[Theorems~4.2 and~4.5, Proposition~4.4]{Br1}}]\label{lemma_Brion_equivariant_mutliplicities}
	Let $X$ be a complex projective variety with an action of $T$, let $x \in X$ be an isolated fixed point.
	Let $\chi_1,\dots,\chi_m$ be the weights of $T_xX$, where $m = \dim_{\C}(X)$.
	\begin{enumerate}
		\item There exists a unique $\Sym$-linear map
		\[
		e_x \colon A^{\ast}_T(X) \to \frac{1}{\chi_1 \cdots \chi_m} \Sym
		\]
		such that $e_x[x] = 1$ and that $e_x[Y] = 0$ for any $T$-invariant subvariety $Y \subset X$ which does not contain $x$.
		\item If $x$ is an attractive point in $X$, that is, all weights in the tangent space $T_xX$ are contained in some open half-space of $M \otimes_{\Z} \R$,  then $e_x[X] \neq 0$.

		\item For any $T$-invariant subvariety $Y \subset X$, the rational function $e_x[Y]$ is homogeneous of degree $(-\dim_{\C}(Y))$.
		\item The point $x$ is nonsingular in $X$ if and only if
		\[
		e_x[X] = \frac{1}{\chi_1 \cdots \chi_m}.
		\]
		Moreover, in this case, for any $T$-invariant subvariety $Y \subset X$, we have
		\[
		[Y]_x = e_x[Y] \chi_1 \cdots \chi_m\,,
		\]
		where $[Y]_x$ denotes the pull-back of $[Y]$ by the inclusion of $x$ into $X$.
	\end{enumerate}
\end{lemma}

\begin{proof}[Proof of Proposition~\ref{prop_property_of_swh}]
	\noindent (1) The first statement follows from the result of Carrell and Sommese~\cite[Lemma~1 in Section~IV]{CarrellSommese83}.
	\smallskip
	
	\noindent (2)
	It is enough to show that the claim holds for $[\Owh{w}]$, that is, $\supp([\Owh{w}]) = \Owh{w}^T$, and $[\Owh{w}]_v$ is homogeneous of degree~$\ell_h(w)$.
	If $v \notin \Owh{w}^T$, then
	\[
	[\Owh{w}]_{v} \stackrel{(4)}{=} e_{v}[\Owh{w}] \chi_1 \cdots \chi_m \stackrel{(1)}{=} 0
	\]
	using (1) and (4) in Lemma~\ref{lemma_Brion_equivariant_mutliplicities}, where $\chi_1,\dots,\chi_m$ are weights in the tangent space $T_v \Hess(S,h)$.
	Moreover, since all $v \in \Owh{w}^T$ are attractive in $\Owh{w}$ (cf.~\cite[\S 6.5]{Br1}), we obtain $e_{v} [\Owh{w}] \neq 0$ by Lemma~\ref{lemma_Brion_equivariant_mutliplicities}(2). Then by Lemma~\ref{lemma_Brion_equivariant_mutliplicities}(3) and (4), we have $[\Owh{w}]_v \neq 0$ and is homogeneous of degree $\dim_{\C}(\Hess(S,h)) - \dim_{\C} \Owh{w} = \ell_h(w)$.  Therefore, we obtain the second statement.
	
	\smallskip
	\noindent(3) Note that the point $\dot{w}B$ is nonsingular in $\Owh{w}$.
	By Lemma~\ref{lemma_Brion_equivariant_mutliplicities}(4), we obtain
	\[
	[\Owh{w}]_w = e_w[\Owh{w}] \chi_1 \cdots \chi_m
	= \frac{ \prod_{(w \to v) \in \Gh} \alpha(w \to v)}{ \prod_{(w \to v) \notin \Gh^o} \alpha(w \to v)}
	= \prod_{(w \to v) \in \Gh^o} \alpha(w \to v).
	\]
	Here, the second equality is deduced from the fact the $\Owh{w}$ is the closure of the minus cell and the weights of $T_w \Owho{w}$ are given by $\alpha(w \to v)$ for $(w \to v) \notin \Gh^o$. This completes the proof of the third statement.
\end{proof}
\begin{remark}
	There are several attempts to find bases of the equivariant cohomology of a complex variety with an action of $T$. Guillemin and Zara~\cite{GZ02,GZ03} introduced `equivariant Thom classes' which can be considered as the equivariant Poincar\'e duals of the closures of the minus cells when the closures are smooth. Goldin and Tolman~\cite{GoldinTolman} considered a similar problem for Hamiltonian $(S^1)^k$-manifolds.
	However, their bases are not uniquely defined in general. Recently, Pabiniak and Sabatini~\cite{PabiniakSabatini} defined `canonical bases' for symplectic toric manifolds which are uniquely defined.
	For regular semisimple Hessenberg varieties, Tymoczko~\cite{T1} focused on the case of $\Hess(S,h) = \flag(\C^n)$ and studied a basis called `Knutson--Tao classes' which are uniquely determined. Moreover, Teff~\cite{Teff_DDO,Teff_thesis} studied bases of the equivariant cohomology rings of arbitrary regular semisimple Hessenberg varieties called `flow-up bases'.
	
	Our basis $\swh{w}$ satisfies nice properties which are already considered in the known results.
	Indeed, when $\Hess(S,h) = \flag(\C^n)$, our basis coincides with the basis constructed by Tymoczko. When $\Hess(S,h)$ is toric, then our basis is the  `canonical basis' provided by Pabiniak and Sabatini. Moreover, by Proposition~\ref{prop_property_of_swh}, the basis $\swh{w}$ is a `flow-up basis' studied by Teff.
\end{remark}

We close this subsection with a property of an equivariant cohomology class, which will be used in Section~\ref{sec:lee_permutohedral}. For a nonsingular GKM space $(X,T)$ with the GKM graph $(V,E,\alpha)$, suppose that $Y \subset X$ is a $T$-invariant nonsingular subvariety.
Since the set $\{\alpha(v \to w) \mid (v \to w)\in E\}$ forms the weights of $T_vX$ and the subset $\{\alpha(v \to w) \mid (v \to w)\in E \text{ and }w \in \supp([Y])\}$ forms the weights of $T_vY$, by Lemma~\ref{lemma_Brion_equivariant_mutliplicities}(4), we obtain the following:
\begin{proposition}\label{cor:regular class}
	Let $(X,T)$ be a nonsingular GKM space  and let $\Gamma = (V, E, \alpha)$ be the GKM graph.
	Let $Y \subset X$ be a $T$-invariant nonsingular subvariety.
	For $v \in \supp([Y])$, if $Y$ is smooth at $v$, then we have
	\[
	[Y](v) = \prod_{\substack{(v \to w) \in E, \\ w \notin \supp([Y])}} \alpha(v \to w).
	\]
Here, $[Y]$ is the equivariant cohomology class determined by $Y$ via the cycle map~\eqref{eq_cycle_map_is_iso}.
\end{proposition}

\section{Fixed points in the closure of a minus cell}\label{sec_description_minus_cell}
In this section, we provide a concrete formula for the support of each class $\swh{w}$ in Theorem~\ref{thm_support_of_Owh} using the \textit{reachability} of a certain graph.
To prove the theorem, we study an explicit description of the minus cell $\Owho{w}$ and its properties in Proposition~\ref{prop_reachable_nonzero} and Theorem~\ref{thm_reachable_nonzero_minor}.

\subsection{Support of $\swh{w}$}
\begin{definition}\label{def_graph_Ghw}
	Let $h$ be a Hessenberg function and $w \in \mathfrak{S}_n$.
	We define a directed graph~$\Gwh$ with the vertex set~$[n]$ such that for each pair of indices $1 \leq j < i \leq n$, there is an edge $j \to i$ in $\Gwh$ if and only if
	\begin{equation}\label{eq_def_of_Ghw_edges}
	j < i \leq h(j), \quad w(j) < w(i).
	\end{equation}
\end{definition}

\begin{example}\label{example_Gwh}
	We assume that $n = 5$ and $h = (3,3,4,5,5)$.
	We present the graphs $G_{24135, h}$,  $G_{15342, h}$, and $G_{12345,h}$ in Figure~\ref{fig_graphs_Gwh}.
	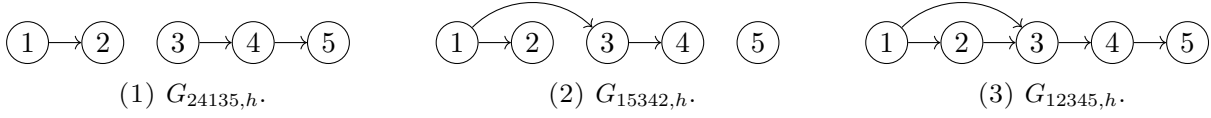
\begin{figure}[t]
		\begin{subfigure}[b]{0.3\textwidth}
			\begin{tikzpicture}
			\foreach \x in {1,2,3,4,5}
			\node[circle,draw,inner sep=0pt,text width=5mm,align=center] (\x) at (\x,0) {\x};
			\draw[->] (1)--(2);
			\draw[->] (3)--(4);
			\draw[->] (4)--(5);
			\end{tikzpicture}
			\caption{$G_{24135,h}$.}\label{subfigure_24135}
		\end{subfigure}
		\hspace{0.5cm}
		\begin{subfigure}[b]{0.3\textwidth}
			\begin{tikzpicture}
			\foreach \x in {1,2,3,4,5}
			\node[circle,draw,inner sep=0pt,text width=5mm,align=center] (\x) at (\x,0) {\x};
			\draw[->] (1)--(2);
			\draw[->] (1) to [bend left = 45] (3);
			\draw[->] (3)--(4);
			\end{tikzpicture}
			\caption{$G_{15342,h}$.}\label{subfigure_15342}
		\end{subfigure}
		\hspace{0.5cm}
		\begin{subfigure}[b]{0.3\textwidth}
			\begin{tikzpicture}
			\foreach \x in {1,2,3,4,5}
			\node[circle,draw,inner sep=0pt,text width=5mm,align=center] (\x) at (\x,0) {\x};
			\draw[->] (1)--(2);
			\draw[->] (1) to [bend left = 45] (3);
			\draw[->] (2)--(3);
			\draw[->] (3)--(4);
			\draw[->] (4)--(5);
			\end{tikzpicture}
			\caption{$G_{12345,h}$.}\label{subfigure_12345}
		\end{subfigure}
		\caption{Graphs $\Gwh$ for $h = (3,3,4,5,5)$.}\label{fig_graphs_Gwh}
	\end{figure}
\end{example}
By the definitions of the graph $\Gwh$ and the length $\ell_h(w)$ (see~Definition~\ref{def_graph_Ghw} and~\eqref{eq_def_of_lhw}), and by the dimension formula in~\eqref{eq_dim_codim_X_Omega}, we obtain
\[
\# E(\Gwh) = \dim_{\C} \Owho{w}.
\]
We note that the graph $G_{e,h}$ agrees with an acyclic orientation of the \textit{incomparability graph} of the natural unit interval order determined by the Hessenberg function $h$, where $e$ is the identity element in $\mathfrak{S}_n$; see \cite{SW}.

We say a vertex $i$ is \textit{reachable} from $j$ in~$\Gwh$ if there exists a sequence of vertices $v_0, v_1,\dots,v_k$ such that $ j = v_0 \to v_1 \to \dots \to v_{k-1} \to v_k = i$. We allow the length of a sequence to be $0$, that is, $j$ is reachable from $j$.
Let $A = \{a_1 < a_2 < \dots < a_k\}$ and $B  = \{ b_1 < b_2 < \dots < b_k\}$ be subsets of $[n]$.
We say that $A$ is \textit{reachable} from $B$ in~$\Gwh$ if there exists a permutation $\sigma \in \mathfrak{S}_k$ such that $a_{\sigma(d)}$ is reachable from $b_d$ for all $d = 1,\dots,k$.
We will simply say `$i$ is reachable from $j$' without mentioning the graph $\Gwh$ when no confusion can arise.

For two vertices $j$ and~$i$ in the graph $\Gwh$, the \textit{distance} $d(j,i)$ is defined to be the  length of the shortest sequences  $j = v_0 \to v_1 \to \dots \to v_{k-1} \to v_k = i$. We set $d(j,j) = 0$ and if $i$ is not reachable from $j$, then $d(j,i) = \infty$.

For $ 1 \leq j \leq n$, we define the following set of ordered tuples:
\[
I_{j,n} \colonequals  \{(i_1,\dots,i_j) \in \Z^j \mid 1 \leq i_1 < \dots < i_j \leq n \}.
\]
For $u \in \mathfrak{S}_n$, and $(i_1,\dots,i_j) \in I_{j,n}$, we set
\[
u \cdot (i_1,\dots,i_j) \colonequals \{u(i_1),\dots,u(i_j)\}\!\!\uparrow.
\]
Here, for a $j$-tuple $(i_1,\dots,i_j)$ of distinct integers, $\{i_1,\dots,i_j\}\!\!\uparrow$
denotes the ordered $j$-tuple obtained from $i_1,\dots,i_j$ by arranging numbers in ascending order.
Furthermore, we set
\[
u^{(j)} = u \cdot (1,\dots,j) \quad \text{ for }1 \leq j \leq n.
\]

Using the reachability of the  graph $\Gwh$, we define the set $J_{w,h,j} \subset I_{j,n}$:
\begin{equation}
J_{w,h,j} \colonequals \{
(i_1,\dots,i_j) \in I_{j,n} \mid
\{ i_1,\dots,i_j\} \text{ is reachable from }[j] \text{ in~$\Gwh$} \}.
\end{equation}
\begin{example}\label{example_Lwh}
	We continue with Example~\ref{example_Gwh}. Let $h = (3,3,4,5,5)$.
	Using graphs in Figure~\ref{fig_graphs_Gwh}, we obtain the subsets $J_{w,h,j}$ as follows.
	\begin{enumerate}
		\item When $w = 24135$, we have
		\begin{gather*}
		J_{w,h,1} = \{ (1), (2)\},
		J_{w,h,2} = \{ (1,2)\},
		J_{w,h,3} = \{(1,2,3), (1,2,4), (1,2,5)\}, \\
		J_{w,h,4} = \{ (1,2,3,4), (1,2,3,5), (1,2,4,5)\},
		J_{w,h,5} = \{(1,2,3,4,5)\}.
		\end{gather*}
		\item When $w = 15342$, we have
		\begin{gather*}
		J_{w,h,1} = \{(1),(2),(3),(4)\}, J_{w,h,2} = \{(1,2),(2,3),(2,4)\}, J_{w,h,3} = \{(1,2,3),(1,2,4),(2,3,4)\}, \\
		J_{w,h,4} = \{(1,2,3,4)\}, J_{w,h,5} = \{(1,2,3,4,5)\}.
		\end{gather*}
		\item When $w = 12345$, we have
		\begin{gather*}
		J_{w,h,j} = I_{j,n} \quad \text{ for } 1\leq j \leq 5.
		\end{gather*}
	\end{enumerate}
\end{example}

Using these terminologies, we define a subset of $\mathfrak{S}_n$ which is determined by a Hessenberg function~$h$ and a permutation $w \in \mathfrak{S}_n$.
\begin{definition}\label{def_of_Awh}
	Let $h$ be a Hessenberg function on $[n]$ and $w \in \mathfrak{S}_n$.
	We define a subset~$\Awh{w} \subset \mathfrak{S}_n$ to be
	\[
	\Awh{w} \colonequals \{ u \in \mathfrak{S}_n \mid u^{(j)} \in \{
	w \cdot (i_1,\dots,i_j) \mid (i_1,\dots,i_j) \in J_{w,h,j}\}
	\quad \text{ for all }1 \leq j \leq n\}.
	\]
\end{definition}

The subsequent theorem is the main result of this section whose proof will be provided in Subsection~\ref{sec_proof_of_support_theorem}.
\begin{theorem}\label{thm_support_of_Owh}
	Let $h$ be a Hessenberg function and $w \in \mathfrak{S}_n$. Then the $T$-fixed points in the closure~$\Owh{w}$  of the minus cell is  $\Awh{w}$, that is,  $\Awh{w} = \Owh{w}^T = \supp(\swh{w})$.
\end{theorem}

\begin{example}
	We continue with Example~\ref{example_Lwh}. Let $h = (3,3,4,5,5)$.
	\begin{enumerate}
		\item Suppose that $w = 24135$.
		Then the subset $\Awh{w}$ is given by
		\[
		\Awh{w} = \{ 24135, 24153, 24351, 24315, 24531, 24513,
		42135, 42153, 42351, 42315, 42531, 42513\}.
		\]
		\item Suppose that $w = 15342$. Then the subset $\Awh{w}$ is given by
		\[
		\Awh{w} = \{ 15342, 15432, 35142, 35412, 45132, 45312,
		51342, 51432, 53142, 53412, 54132, 54312\}.
		\]
		\item Suppose that $w = 12345$. Then the subset $\Awh{w}$ is the same as the set $\mathfrak{S}_n$ of permutations.
	\end{enumerate}
\end{example}

\begin{example}\label{example_reachability_full_flag}
	Suppose that $h = (n,n,\dots,n)$. Then we have $\Hess(S,h) = \flag(\C^n)$. In this case, for $1 \leq j < i \leq n$, the graph $\Gwh$ has an edge $j  \to i$ if and only if   $w(j) < w(i)$. Accordingly, for $\underline{i} = (i_1,\dots,i_j) \in I_{j,n}$, the set $\{i_1,\dots,i_j\}$ is  reachable from $[j]$ in~$G_{w,h}$ if and only if $w \cdot \underline{i} \geq w^{(j)}$. Here, we use an order on $I_{j,n}$, defined as follows: For $(a_1,\dots,a_j), (b_1,\dots,b_j) \in I_{j,n}$,
	\[
	(a_1,\dots,a_j) \geq (b_1,\dots,b_j) \iff a_t \geq b_t \quad \text{ for all }1 \leq t \leq j.
	\]
	Therefore, we have $J_{w,h,j} = \{ \underline{i} \mid w \cdot \underline{i} \geq w^{(j)}\}$, and moreover,
	\[
	A_{w,h} = \{ u \in \mathfrak{S}_n \mid u^{(j)} \geq w^{(j)} \quad \text{ for all }j\}.
	\]
	This proves that $A_{w,h} = \{ u \in \mathfrak{S}_n \mid u \geq w \text{ (in the Bruhat order)}\}$ in this case because  $u \geq w$ in the Bruhat order if and only if $u^{(j)} \geq w^{(j)}$ for all $j$ (see, for example,~\cite[\S 3.2]{BilleyLakshmibai}).
\end{example}

\subsection{Description of minus cells}
By the definition of the minus cell $\Owho{w}$, we have
\[
\Owho{w} = \Hess(S,h) \cap \Owo{w}.
\]
We note that elements in the opposite Schubert cell $\Owo{w}$ are described by
\begin{equation}\label{eq_Owo}
\Owo{w} =
\left\{
g = (g_{i,j}) \in \GL_n(\C) \mid
g_{w(j),j} = 1, \quad
g_{i,j} = 0 \text{ if $i < w(j)$ or $w^{-1}(i)<j$}
\right\}B \subset G/B.
\end{equation}

We recall the result in~\cite{DeMari_Shayman_88} which describes the minus cell $\Owho{w}$ explicitly. The authors of \cite{DeMari_Shayman_88} consider only the case  when a Hessenberg function has the form $h(i) = i+ p$, while their results can be extended to all Hessenberg functions.  Recall that $n$ distinct numbers $c_1, \dots, c_n$ are the diagonal entries of our regular semisimple matrix $S$.
\begin{proposition}{\cite[\S 5]{DeMari_Shayman_88}}\label{prop_minus_cell_description}
	Let $x = (x_{i,j})$ be a lower triangular matrix having $1$ on its diagonal.
	Then $\dot{w}x \in \Owho{w}$ if and only if
	\begin{enumerate}
		\item 	$x_{i,j} = 0$   if $w(i) < w(j)$;
		\item if $w(i) > w(j)$  and $i > h(j)$, we have
		\[
		x_{i,j} = \frac{-1}{c_{w(i)}- c_{w(j)}}
		\left[
		\sum_{t=1}^{i-j-1} \sum_{ i > \gamma_1 > \gamma_2 > \dots > \gamma_t > j} (-1)^t (c_{w(\gamma_t)} - c_{w(j)}) x_{i, \gamma_1} x_{\gamma_1, \gamma_2} \cdots x_{\gamma_t, j}
		\right].
		\]
	\end{enumerate}
\end{proposition}
From the description of the minus cell, the entries $x_{i,j}$ for $w(i) > w(j)$ and  $i \leq h(j)$ are free, i.e., there is no restriction on these entries. The number of such entries agrees with the dimension of the minus cell.
Using Proposition~\ref{prop_minus_cell_description}, we obtain the following corollary.
\begin{corollary}\label{cor_xij_description_using_the_graph}
	Let $x = (x_{i,j})$ be a lower triangular matrix having $1$ on its diagonal. Then $\dot{w}x \in \Owho{w}$ if and only if
	\begin{enumerate}
		\item $x_{i,j} = 0$ if $i$ is not reachable from $j$ in~$G_{w,h}$;
		\item if $i$ is reachable from $j$ in~$G_{w,h}$ with $d(j,i) > 1$, then
		\[
		x_{i,j} = \frac{-1}{c_{w(i)}- c_{w(j)}}
\sum
(-1)^t (c_{w(\gamma_t)} - c_{w(j)}) x_{i, \gamma_1} x_{\gamma_1, \gamma_2} \cdots x_{\gamma_t, j}.
		\]
Here, the sum is over sequences $\gamma_0 = i > \gamma_1 > \gamma_2 > \dots > \gamma_t > j= \gamma_{t+1}$ such that $t > 0$ and $\gamma_{\ell}$ is reachable from $\gamma_{\ell+1}$ in~$G_{w,h}$ for $\ell=0, \dots, t$.
	\end{enumerate}
\end{corollary}

\begin{example}\label{example_xij_not_free}
	We continue with Example~\ref{example_Gwh}. Let $h = (3,3,4,5,5)$ and $x = (x_{i,j})$ be a lower triangular matrix having $1$ on its diagonal.
	\begin{enumerate}
		\item Suppose that $w = 24135$. By considering the reachability of $\Gwh$ (see Figure~\ref{subfigure_24135}), the entries $x_{3,1}, x_{4,1}, x_{5,1}, x_{3,2}, x_{4,2}, x_{5,2}$ vanish. 		
		Moreover, three edges $(1 \to 2)$, $(3 \to 4)$, and $(4 \to 5)$ define free variables $x_{2,1}$, $x_{4,3}$, and $x_{5,4}$.
		The variable $x_{5,3}$ is given by the following equation.
		\[
		x_{5,3}  = \frac{-1}{c_{5}-c_{1}} \left[
		(-1)^1(c_{3} - c_{1}) x_{5,4}x_{4,3}
		\right] = \frac{(c_3-c_1) x_{5,4}x_{4,3}}{c_5-c_1}.
		\]
		
		\item Suppose that $w = 15342$. By considering the reachability of $\Gwh$ (see Figure~\ref{subfigure_15342}), the entries $x_{5,1}$, $x_{3,2}$, $x_{4,2}$, $x_{5,2}$, $x_{5,3}$, $x_{5,4}$ vanish, and there are three free variables $x_{2,1}, x_{3,1}$, and $x_{4,3}$. The variable $x_{4,1}$ is given by the following equation.
		\[
		x_{4,1} = \frac{-1}{c_{4} - c_{1}} \left[
		(-1)^{1}(c_3-c_1) x_{4,3}x_{3,1}
		\right] = \frac{(c_3-c_1)x_{4,3}x_{3,1}}{c_4-c_1}.
		\]
		
		\item Suppose that $w = 12345$. Then there are  five free variables $x_{2,1}, x_{3,1}, x_{3,2}, x_{4,3}$, and $x_{5,4}$, and the other variables are given as follows:
		\[
		\begin{split}
		x_{4,1} &= \frac{-1}{c_4 -c_1} \left[
		(-1)^1 (c_3-c_1) x_{4,3}x_{3,1} + (-1)^1(c_2-c_1) x_{4,2}x_{21} + (-1)^2(c_2-c_1) x_{4,3}x_{3,2}x_{2,1}
		\right], \\
		x_{5,1} &= \frac{-1}{c_5-c_1} [		(-1)^2(c_3-c_1) x_{5,4}x_{4,3}x_{3,1} + (-1)^1(c_3-c_1) x_{5,3}x_{3,1} + (-1)^1 (c_4-c_1) x_{5,4}x_{4,1} \\
		&\qquad + (-1)^3 (c_2-c_1) x_{5,4}x_{4,3}x_{3,2}x_{2,1} + (-1)^2 (c_2-c_1) x_{5,4}x_{4,2}x_{2,1} \\
		&\qquad \quad  + (-1)^2 (c_2-c_1) x_{5,3}x_{3,2}x_{2,1} + (-1)^1 (c_2-c_1) x_{5,2}x_{2,1}  ],\\
		x_{4,2} &= \frac{-1}{c_4-c_2}[(-1)^1(c_3-c_2)x_{4,3}x_{3,2}] = \frac{c_3-c_2}{c_4-c_2} x_{4,3}x_{3,2},\\
		x_{5,2} &= \frac{-1}{c_5-c_2}[ (-1)^2 (c_3-c_2) x_{5,4}x_{4,3}x_{3,2} + (-1)^1(c_4-c_2) x_{5,4}x_{4,2} + (-1)^1 (c_3-c_2) x_{5,3}x_{3,2}],\\
		x_{5,3} &= \frac{-1}{c_5-c_3}[(-1)^1 (c_4-c_3) x_{5,4}x_{4,3}] = \frac{c_4-c_3}{c_5-c_3}x_{5,4}x_{4,3}.
		\end{split}
		\]
	\end{enumerate}	
	
\end{example}

\subsection{Minors and reachability}\label{section_reachability}
In this subsection, we study the expression of $x_{i,j}$ in terms of free variables.
Applying Corollary~\ref{cor_xij_description_using_the_graph} iteratively, one may get this expression of $x_{i,j}$ in terms of free variables.
For notational simplicity, we define a set of paths connecting $i$ and $j$, which have length greater than $1$.
\[
\mathcal P_{w,h}(i,j) \colonequals
\left\{ \gamma_{\bullet} = (\gamma_{t+1} \to \gamma_t \to \dots \to \gamma_1 \to \gamma_0) ~~\middle|~~
\begin{array}{l}
i = \gamma_0, j = \gamma_{t+1}, t >0, \text{ and } \\
 \gamma_{\ell} \to \gamma_{\ell-1} \text{ is an edge in } G_{w,h}  \text{ for all }\ell
\end{array} \right\}.
\]
Each path $\gamma_{\bullet} \in \mathcal P_{w,h}(i,j)$ indexes a monomial  in the expression of $x_{i,j}$.
Furthermore, we have
\begin{equation}\label{eq_xij_monomial_expression}
x_{i,j} = \sum_{\gamma_{\bullet} \in \mathcal P_{w,h}(i,j)}
C(i, \gamma_1,\dots,\gamma_t,j) x_{i,\gamma_1} x_{\gamma_1,\gamma_2} \cdots x_{\gamma_t,j}
\end{equation}
Here, $C(i, \gamma_1,\dots,\gamma_t,j)$ is the coefficient of the monomial $x_{\gamma_1,\gamma_2} \cdots x_{\gamma_t,j}$. Note that each coefficient $C(i, \gamma_1,\dots,\gamma_t,j)$ is a Laurent polynomial in variables $c_1,\dots,c_n$.

For a given path $\gamma_{\bullet} \in \mathcal P_{w,h}(i,j)$, we call it \textit{minimal} if there are no  edges $\gamma_{\ell} \to \gamma_{\ell'}$  for $\ell'-\ell>1$. For example, the path $1 \to 2 \to 3 \to 4 \,\,\in \mathcal{P}_{12345, (3,3,4,5,5)}(4,1)$  is not minimal because there is an edge $1 \to 3$ (see Figure~\ref{subfigure_12345}).  We note that there is a minimal path from $j$ to $i$ if $i$ is reachable from $j$.
We will see in Proposition~\ref{prop_reachable_nonzero} that if $i$ is reachable from $j$, then $x_{i,j}$ is not identically zero; one of the coefficients in the expression~\eqref{eq_xij_monomial_expression} is not identically zero as a Laurent polynomial in variables $c_1, \dots,c_n$. Before presenting the statement, we provide an example.
\begin{example}\label{example_reachability}
	We continue with Example~\ref{example_xij_not_free}(3).
	Substituting the expression of $x_{4,2}$ in $x_{4,1}$, we obtain the following.
	\begin{align*}
	x_{4,1} &= \frac{-1}{c_4 -c_1} \left[
	(-1)^1 (c_3-c_1) x_{4,3}x_{3,1} + (-1)^1(c_2-c_1) x_{4,2}x_{2,1} + (-1)^2(c_2-c_1) x_{4,3}x_{3,2}x_{2,1}
	\right] \\
	&= \frac{-1}{c_4 -c_1} \left[
	(-1)^1 (c_3-c_1) x_{4,3}x_{3,1} + (-1)^1(c_2-c_1) \frac{(c_3-c_2)x_{4,3}x_{3,2}}{c_4-c_2} x_{2,1} + (-1)^2(c_2-c_1) x_{4,3}x_{3,2}x_{2,1}
	\right]  \\
	&= \frac{c_3-c_1}{c_4-c_1} x_{4,3}x_{3,1} - \frac{(c_2-c_1)(c_4-c_3)}{(c_4-c_1)(c_4-c_2)}x_{4,3}x_{3,2}x_{2,1}.
	\end{align*}
	There are two paths $1 \to 3 \to 4$ and $1 \to 2 \to 3 \to 4$ in $\mathcal P_{w,h}(4, 1)$ (see Figure~\ref{subfigure_12345}), and these paths define monomials $x_{4,3}x_{3,1}$ and $x_{4,3}x_{3,2}x_{2,1}$. Considering the coefficients $C(i,\gamma_1,\dots,\gamma_t, j)$ in~\eqref{eq_xij_monomial_expression}, we obtain
	\[
	C(4, 3, 1) = \frac{c_3-c_1}{c_4-c_1}, \quad C(4, 3, 2, 1) = - \frac{(c_2-c_1)(c_4-c_3)}{(c_4-c_1)(c_4-c_2)}.
	\]
	Here, we notice that the path $1 \to 3 \to 4$ is minimal while $1 \to 2 \to 3 \to 4$ is not.
	
	Substituting the expression of $x_{4,2}$ and $x_{5,3}$ in $x_{5,2}$, we have
	\begin{align*}
	x_{5,2} &= \frac{-1}{c_5-c_2}[ (-1)^2 (c_3-c_2) x_{5,4}x_{4,3}x_{3,2} + (-1)^1(c_4-c_2) x_{5,4}x_{4,2} + (-1)^1 (c_3-c_2) x_{5,3}x_{3,2}] \\
	&= \frac{(c_3-c_2)(c_4-c_3)}{(c_5-c_2)(c_5-c_3)} x_{5,4}x_{4,3}x_{3,2}.
	\end{align*}	
	There is one path $2 \to 3 \to 4 \to 5$ connecting $2$ and $5$ in the graph $\Gwh$, and the corresponding coefficient $C(5, 4, 3, 2)$ is given by
	\[
	C(5, 4, 3, 2) = \frac{(c_4-c_3)(c_3-c_2)}{(c_5-c_2)(c_5-c_3)}.
	\]

	Substituting $x_{5,3}$, $x_{5,2}$, $x_{4,1}$ and $x_{4,2}$ to $x_{5,1}$, we obtain
\[
\begin{split}
	x_{5,1} &= \frac{-1}{c_5-c_1} [		(-1)^2(c_3-c_1) x_{5,4}x_{4,3}x_{3,1} + (-1)^1(c_3-c_1) x_{5,3}x_{3,1} + (-1)^1 (c_4-c_1) x_{5,4}x_{4,1} \\
	&\qquad + (-1)^3 (c_2-c_1) x_{5,4}x_{4,3}x_{3,2}x_{2,1} + (-1)^2 (c_2-c_1) x_{5,4}x_{4,2}x_{2,1} \\
	&\qquad \quad + (-1)^2 (c_2-c_1) x_{5,3}x_{3,2}x_{2,1} + (-1)^1 (c_2-c_1) x_{5,2}x_{2,1} ] \\
&= \frac{(c_4-c_3)(c_3-c_1)}{(c_5-c_1)(c_5-c_3)} x_{5,4}x_{4,3}x_{3,1} - \frac{(c_4-c_3)(c_2-c_1)}{(c_5-c_1)(c_5-c_2)} x_{5,4}x_{4,3}x_{3,2}x_{2,1}.
\end{split}
\]
	There are two paths $1 \to 3 \to 4 \to 5$ and $1 \to 2 \to 3 \to 4 \to 5$ in $\mathcal P_{w,h}(5, 1)$, and they define monomials $x_{5,4}x_{4,3}x_{3,1}$ and $x_{5,4}x_{4,3}x_{3,2}x_{2,1}$, respectively. Their coefficients are given by
	\[
	C(5, 4, 3, 1) = \frac{(c_4-c_3)(c_3-c_1)}{(c_5-c_1)(c_5-c_3)}, \quad C(5, 4, 3, 2, 1) = - \frac{(c_4-c_3)(c_2-c_1)}{(c_5-c_1)(c_5-c_2)}.
	\]
\end{example}

\begin{proposition}\label{prop_reachable_nonzero}
	Let $x = (x_{i,j})$ be a lower triangular matrix having $1$ on its diagonal. Suppose that  $\dot{w}x \in \Owho{w}$ and  $i$ is reachable from $j$.
	For a minimal path $\gamma_{\bullet} \in \mathcal P_{w,h}(i,j)$,
the coefficient $C(i,\gamma_1,\dots,\gamma_t,j)$ of the monomial $x_{i,\gamma_1} x_{\gamma_1,\gamma_2} \cdots x_{\gamma_t,j}$ in the expression~\eqref{eq_xij_monomial_expression} is given as follows:
	\[
	\frac{(c_{w(\gamma_1)} - c_{w(\gamma_2)}) (c_{w(\gamma_2)} - c_{w(\gamma_3)})\cdots (c_{w(\gamma_t)} - c_{w(\gamma_{t+1})})}{(c_{w(\gamma_0)} - c_{w(\gamma_{t+1})}) (c_{w(\gamma_0)} - c_{w(\gamma_{t})}) \cdots (c_{w(\gamma_0)} - c_{w(\gamma_{2})})},
	\]
where $\gamma_0=i$ and $\gamma_{t+1}=j$.
	Thus, $x_{i,j}$ is not identically zero if $i$ is reachable from $j$.
\end{proposition}
\begin{proof}
	To obtain the coefficient of a certain monomial, we substitute $x_{k,l} = 0$ for variables which do not appear in the monomial. Furthermore, for a given minimal path $\gamma_{\bullet} = (j \to \gamma_t \to \cdots \to \gamma_{1} \to i) \in \mathcal P_{w,h}(i,j)$, we set $\mathcal I = \{ x_{i,\gamma_1}, x_{\gamma_1,\gamma_2},\dots,x_{\gamma_t,j}\}$. Then, we have
	\[
	x_{i,j}|_{\{x_{k,l} = 0 \mid x_{k,l} \notin \mathcal I\}} = C(i, \gamma_1,\dots,\gamma_t,j) x_{i,\gamma_1} x_{\gamma_1,\gamma_2} \cdots x_{\gamma_t,j}.
	\]

	We use an induction on $t$, which is  one less than the length of the minimal path.
We assume that $t = 0$, i.e., there is an edge $ j \to i$ in the graph $\Gwh$. In this case, $x_{i,j}$ itself is a free variable, and the result follows.
	We assume that $t = 1$. Then, by Corollary~\ref{cor_xij_description_using_the_graph}(2), we have
	\[
	x_{i,j}|_{\{x_{k,l} = 0 \mid x_{k,l} \notin \mathcal I\}} = \frac{-1}{c_{w(i)}- c_{w(j)}} (-1)^1 (c_{w(\gamma_1)} - c_{w(j)}) x_{i, \gamma_1} x_{\gamma_1, j}
	= \frac{c_{w(\gamma_1)} - c_{w(j)}}{c_{w(i)}- c_{w(j)}}  x_{i, \gamma_1} x_{\gamma_1, j}.
	\]
	Therefore, the proposition holds.
	
	Now,  for $t>1$, we assume that the proposition holds for any path having the length less than~$t$. We claim the following. Suppose that we have a minimal path $\gamma_{\bullet} = (j \to \gamma_t \to \cdots \to \gamma_1 \to i)\in \mathcal P_{w,h}(i,j)$. Then, for every nonempty subset $S = \{s_1,\dots,s_k\} \subset [t-1]$, we have
	\begin{equation}\label{eq_prop3.11_1}
	[(c_{w(\gamma_t)}-c_{w(j)}) x_{i,\gamma_{s_1}}x_{\gamma_{s_1},\gamma_{s_2}}\cdots x_{\gamma_{s_k},\gamma_t} x_{\gamma_t,j} - (c_{w(\gamma_{s_k})} - c_{w(j)}) x_{i,\gamma_{s_1}} x_{\gamma_{s_1},\gamma_{s_2}} \cdots x_{\gamma_{s_k},j}]\big|_{\{x_{k,l} = 0 \mid x_{k,l} \notin \mathcal I\}} = 0.
	\end{equation}
	By the induction assumption, using the minimal path $j \to \gamma_t \to \gamma_{t-1} \to \dots \to \gamma_{s_k+1} \to \gamma_{s_k}$, we obtain
	\begin{equation}\label{eq_prop3.11_2}
	\begin{split}
	x_{\gamma_{s_k},j}|_{\{x_{k,l} = 0 \mid x_{k,l} \notin \mathcal I\}} &= \frac{(c_{w(\gamma_{s_k+1})} - c_{w(\gamma_{s_k+2})}) \cdots (c_{w(\gamma_t)} - c_{w(j)})} {(c_{w(\gamma_{s_k})} - c_{w(j)})(c_{w(\gamma_{s_k})} - c_{w(\gamma_t)}) \cdots (c_{w(\gamma_{s_k})} -c_{w(\gamma_{s_k+2})})}	x_{\gamma_{s_k}, \gamma_{s_k+1}} \cdots x_{\gamma_t, j} \\
	&= \frac{c_{w(\gamma_t)} - c_{w(j)}}{c_{w(\gamma_{s_k})} - c_{w(j)}} x_{\gamma_{s_k}, \gamma_t} x_{\gamma_t, j}.
	\end{split}
	\end{equation}
	Here, the second equality comes from  applying the induction hypothesis to the path $\gamma_t \to \gamma_{t-1} \dots \to \gamma_{s_k+1} \to \gamma_{s_k}$.
Claim~\eqref{eq_prop3.11_1} follows immediately from   \eqref{eq_prop3.11_2}.

	We complete a proof of the inductive step using the equality~\eqref{eq_prop3.11_1}. Suppose that $j \to \gamma_t \to \gamma_{t-1} \to \dots \to \gamma_1 \to i$ is a minimal path.
	Then the monomials appearing on the right hand side of the formula given in Corollary~\ref{cor_xij_description_using_the_graph}(2)
in which the term $x_{i, \gamma_1} x_{\gamma_1,\gamma_2}\cdots x_{\gamma_t, j}$ appears after substituting to obtain an expression in the free variables
 are given as follows:
	\begin{itemize}
		\item $x_{i,\gamma_t} x_{\gamma_t, j}$;
		\item for any nonempty subset $\{s_1,\dots,s_k\} \subset [t-1]$,
		\[
		x_{i, \gamma_{s_1}}x_{\gamma_{s_1}, \gamma_{s_2}} \cdots x_{\gamma_{s_k}, \gamma_t} x_{\gamma_t, j} \quad \text{ and } \quad x_{i, \gamma_{s_1}} x_{\gamma_{s_1},\gamma_{s_2}} \cdots x_{\gamma_{s_k},j}.\]
	\end{itemize}
	By summing up all terms above with appropriate coefficients, we have
	\[
	\begin{split}
	&\frac{-1}{c_{w(i)}-c_{w(j)}}
	\Bigg[
	\sum_{\emptyset \neq \{s_1,\dots,s_k\} \subset [t-1]} \{
	(-1)^{k+1} (c_{w(\gamma_t)} -c_{w(j)}) x_{i, \gamma_{s_1}} x_{\gamma_{s_1}, \gamma_{s_2}} \cdots x_{\gamma_{s_k}, \gamma_t} x_{\gamma_t, j} \\
	& \qquad \qquad \qquad \qquad \qquad \qquad\qquad+ (-1)^k (c_{w(\gamma_{s_k})} - c_{w(j)}) x_{i, \gamma_{s_1}} x_{\gamma_{s_1}, \gamma_{s_2}} \cdots x_{\gamma_{s_k}, j}
	\} \\
	& \qquad \qquad\qquad\qquad + (-1)(c_{w(\gamma_t)} -c_{w(j)}) x_{i, \gamma_t} x_{\gamma_t, j}
	\Bigg] \Bigg|_{\{x_{k,l} = 0 \mid x_{k,l} \notin \mathcal I\}}\\
	&\quad = \frac{c_{w(\gamma_t)} -c_{w(j)}}{c_{w(i)}-c_{w(j)}} x_{i, \gamma_t} x_{\gamma_t,j} \\
	&\quad =\frac{c_{w(\gamma_t)} -c_{w(j)}}{c_{w(i)}-c_{w(j)}}
	\cdot
	\frac{(c_{w(\gamma_1)} - c_{w(\gamma_2)}) \cdots (c_{w(\gamma_{t-1})} - c_{w(\gamma_t)})}{(c_{w(i)} - c_{w(\gamma_t)}) \cdots (c_{w(i)} - c_{w(\gamma_2)})} x_{i, \gamma_1} x_{\gamma_1,\gamma_2}\cdots x_{\gamma_{t-1},\gamma_t} x_{\gamma_t, j}.
	\end{split}
	\]
	Here, the first equality comes from the equation~\eqref{eq_prop3.11_1}.
	This proves the proposition.
\end{proof}

Let $A = \{   a_1 < a_2 < \dots < a_k \}$ and $B  = \{   b_1 < b_2 < \dots < b_k  \}$ be subsets of $[n]$.
For an element $x = (x_{i,j}) \in \GL_n(\C)$, we define $p_{A,B}(x)$ to be the $k \times k$ minor of $x$ with row indices $A$ and the column indices $B$, that is,
\begin{equation}\label{eq_p_AB_x}
p_{A,B}(x) \colonequals \det(x_{a,b})_{a \in A, b \in B}.
\end{equation}

\begin{theorem}\label{thm_reachable_nonzero_minor}
	Let $h \colon [n] \to [n]$ be a Hessenberg function and $w \in \mathfrak{S}_n$.
	Let $A = \{a_1 < a_2 < \dots < a_k  \}$ and $B = \{   b_1 < b_2 < \dots < b_k    \}$ be subsets of $[n]$.
	Then there exists $n$ distinct numbers $c_1, \dots, c_n$ satisfying the following: If we set a regular semisimple matrix $S$ to be the diagonal matrix $\textup{diag}(c_1,\dots,c_n)$, then for  a lower triangular matrix $x = (x_{i,j})$ having $1$ on its diagonal and satisfying $\dot{w} x \in \Owho{w}$, we have that
	$p_{A,B}(x)$ is not identically zero as a polynomial with the free variables $\{ x_{i,j} \mid (j \to i) \in \Gwh\}$ if and only if $A$ is reachable from~$B$.
\end{theorem}

\begin{proof}
	We first note that by Corollary~\ref{cor_xij_description_using_the_graph}(1) and Proposition~\ref{prop_reachable_nonzero}, we have that
	$x_{i,j} = 0$  if and only if  $i$ is not reachable from $j$.	This proves the theorem for $k =1$.
	
	Now we consider the `only if' part of the statement: if $A$ is not reachable from $B$, then $p_{A,B}(x)$ is identically zero.
	By the definition of reachability, for every $\sigma \in \mathfrak{S}_k$, there exists an index $d_{\sigma} \in [k]$ such that $a_{\sigma(d_{\sigma})}$ is not reachable from $b_{d_{\sigma}}$. Accordingly, we have $x_{a_{\sigma(d_{\sigma})}, b_{d_\sigma}} = 0$.
	Therefore, we obtain
	\[
	p_{A,B}(x) = \sum_{\sigma \in \mathfrak{S}_k} \text{sgn}(\sigma) \prod_{i=1}^k x_{a_{\sigma(i)}, b_{i}}
	= \sum_{\sigma \in \mathfrak{S}_k} 0 = 0.
	\]
	Here, the second equality comes from $x_{a_{\sigma(d_{\sigma})}, b_{d_\sigma}} = 0$. This proves the only if part of the statement.
	
	To prove the `if' part of the statement, we use an induction argument on $|A| = |B|= k$. We already have observed that the statement holds when $k = 1$. Suppose that $k > 1$ and assume that the claim holds for subsets $A$ and $B$ satisfying $|A|= |B| < k$.
	Because of the description in Corollary~\ref{cor_xij_description_using_the_graph}~(2), we have that
	\begin{equation}\label{eq_xij_expressed_by_other_variables}
	\text{the entry $x_{i,j}$ is expressed by free variables $x_{\gamma, \gamma'}$ such that $j \leq \gamma' < \gamma \leq i$. }
	\end{equation}
	Let $a_p$ be the maximal element in $A$ which is reachable from $b_1$ and such that $A \setminus \{a_p\}$ is reachable from $B \setminus \{b_1\}$.
Considering the first column and the respective minors, we have
	\begin{equation}\label{eq_p_AB_x2}
	p_{A,B}(x) = \sum_{i=1}^k (-1)^{i+1} x_{a_i,b_1} \cdot p_{A \setminus \{a_i\}, B\setminus \{b_1\}}(x)
	=  \sum_{i = 1}^p  (-1)^{i+1} x_{a_i,b_1} \cdot p_{A \setminus \{a_i\}, B\setminus \{b_1\}}(x).
	\end{equation}

There are two possibilities: $(b_1 \to a_p) \in \Gwh$ or $(b_1 \to a_p) \notin \Gwh$. We consider the case that $b_1 \to a_p$ is an edge in $\Gwh$ first.
	Then, the term $x_{a_p,b_1} \cdot p_{A \setminus \{a_p\}, B \setminus \{b_1\}}(x)$ is not identically zero by the induction argument and  it is the only term having $x_{a_p, b_1}$ because of~\eqref{eq_xij_expressed_by_other_variables}. 	Therefore, $p_{A,B}(x)$ is not identically zero.

	Before considering the remaining case, we observe the following. Suppose that $(b_1 \to a_p) \notin G_{w,h}$.  Choose
	$\sigma \in \mathfrak{S}_k$ satisfying that $a_{\sigma(d)}$ is reachable from $b_d$ for all $d \in [k]$.
We note that $\sigma(1)$ is not necessarily $p$.
Let
\[
\gamma_{\bullet}^{(\sigma,d)} = (b_d \to \gamma_{t_d}^{(\sigma,d)} \to \dots \to \gamma_{1}^{(\sigma,d)} \to  a_{\sigma(d)})
\]
be a minimal path.
For a given minimal path $\gamma_{\bullet} = (\gamma_{t+1} \to \dots \to \gamma_1 \to \gamma_0)$, denote by
\[
X(\gamma_{\bullet}) \colonequals x_{\gamma_0,\gamma_1} \cdots x_{\gamma_t, \gamma_{t+1}}
\]
the corresponding monomial for notational convenience.
Let $X_d$ be the monomial given by the path~$\gamma_{\bullet}^{(\sigma,d)}$, that is,
$X_d = X(\gamma_{\bullet}^{(\sigma,d)})$.  We obtain the coefficient $C(a_{\sigma(d)},\gamma_{1}^{(\sigma,d)},\dots,\gamma_{t_d}^{(\sigma,d)},b_d)$ of the monomial~$X_d$ in $x_{a_{\sigma(d)},b_d}$
using Proposition~\ref{prop_reachable_nonzero}:
\[
C(a_{\sigma(d)},\gamma_{1}^{(\sigma,d)},\dots,\gamma_{t_d}^{(\sigma,d)},b_d) = \frac{N_{\sigma,d}(c)}{D_{\sigma,d}(c)},
\]
where
\begin{gather*}
N_{\sigma,d}(c) \colonequals \left(c_{w(\gamma_1^{(\sigma,d)})} - c_{w(\gamma_2^{(\sigma,d)})}\right)
\left(c_{w(\gamma_2^{(\sigma,d)})} - c_{w(\gamma_3^{(\sigma,d)})}\right) \cdots
\left(c_{w(\gamma_{t_d}^{(\sigma,d)})} - c_{w(b_d)} \right), \\
D_{\sigma,d}(c) \colonequals \left(c_{w(a_{\sigma(d)})} - c_{w(b_d)}\right)
\left(c_{w(a_{\sigma(d)})} - c_{w(\gamma_{t_d}^{(\sigma,d)})}\right) \cdots
\left(c_{w(a_{\sigma(d)})} - c_{w(\gamma_{2}^{(\sigma,d)})} \right).
\end{gather*}
We note that the polynomials $N_{\sigma,d}(c)$ and $D_{\sigma,d}(c)$ are not identically zero polynomials in variables $c_1,\dots,c_n$.
Moreover, $N_{\sigma,d}(c)$ is not divisible by $(c_{w(a_{p})} - c_{w(b_1)})$ since $N_{\sigma,d}(c)$ does not have a variable $c_{w(b_1)}$ except when $d=1$; $N_{\sigma,1}(c)$ is not divisible by $(c_{w(a_{p})} - c_{w(b_1)})$. Indeed, the fact $b_1 = \min(A \cup B)$ implies that any path $\gamma_{\bullet}^{(\sigma,d)}$ cannot visit $b_1$ except when $d = 1$.
By the similar reason, the variable $c_{w(b_1)}$ cannot appear  in the polynomial $D_{\sigma,d}(c)$ for $d \neq 1$ because $D_{\sigma,d}(c)$ consists of variables corresponding to paths $\gamma_{\bullet}^{(\sigma,d)}$  and $b_1$ is the smallest value in $B$.  Therefore,
\begin{equation}\label{equation_D_sigma_d_not_divisible_by_cwa}
(c_{w(a_{p})} - c_{w(b_1)}) \nmid D_{\sigma,d}(c) \quad d \neq 1.
\end{equation}

Accordingly, the product $\prod_{d=1}^k x_{a_{\sigma(d)}, b_d}$ has a monomial  of the following form:
	\[
	\frac{F_{\sigma}(c)}{G_{\sigma}(c)} X_1 \cdots X_k.
	\]
Here, $F_{\sigma}(c) = \prod_{d=1}^k N_{\sigma,d}(c)$ and $G_{\sigma}(c) =\prod_{d=1}^k D_{\sigma,d}(c)$.
We note that the polynomials $F_{\sigma}(c)$ and $G_{\sigma}(c)$ are not identically zero polynomials in variables $c_1,\dots,c_n$.
Moreover, $F_{\sigma}(c)$ is not divisible by $(c_{w(a_{p})} - c_{w(b_1)})$.
For the polynomial $G_{\sigma}(c)$, if $\sigma(1) = p$, then  $G_{\sigma(c)}$ is
divisible by $c_{w(a_{p})} - c_{w(b_1)}$, which is the factor coming from $D_{\sigma,1}(c)$. Moreover, the power of this factor is one by~\eqref{equation_D_sigma_d_not_divisible_by_cwa}, that is, $G_{\sigma}(c) = (c_{w(a_{p})} - c_{w(b_1)}) \tilde{G}_{\sigma}(c)$ for some polynomial $\tilde{G}_{\sigma}(c)$, and
\(
	(c_{w(a_{p})} - c_{w(b_1)}) \nmid \tilde{G}_{\sigma}(c).
\)
If $\sigma(1) \neq p$, then again by~\eqref{equation_D_sigma_d_not_divisible_by_cwa}, and the definition of $D_{\sigma,1}(c)$, we obtain
\begin{equation}\label{G_not_divisible_by_cwa}
	(c_{w(a_{p})} - c_{w(b_1)}) \nmid G_{\sigma}(c).
\end{equation}
	
	Using the above observation, we consider the remaining case: $(b_1 \to a_p) \notin \Gwh$.
Choose a permutation $\tau \in \mathfrak{S}_k$ such that $a_{\tau(d)}$ is reachable from $b_d$ for all $d  \in [k]$ and $\tau(1) = p$. Moreover, take minimal paths $\gamma_{\bullet}^{(\tau,d)} \in \mathcal P_{w,h}(a_{\tau(d)},b_d)$ for $d \in [k]$, and set
\[
X = X(\gamma_{\bullet}^{(\tau,1)}) \cdots X(\gamma_{\bullet}^{(\tau,k)}).
\]
Using the expression~\eqref{eq_p_AB_x2}, we have a monomial  of the form
	\[
	\frac{(-1)^{p+1}F(c)}{(c_{w(a_p)} - c_{w(b_1)}) \tilde{G}(c)} X
	\]
coming from $(-1)^{p+1} x_{a_p,b_1} \cdot p_{A \setminus \{a_p\}, B\setminus \{b_1\}}(x)$.
Here, $F(c) = F_{\tau}(c)$ and $\tilde{G}(c) = \tilde{G}_{\tau}(c)$.

If the summation $\sum_{  i < p} (-1)^{i+1} x_{a_i,b_1} \cdot p_{A \setminus \{a_i\}, B\setminus \{b_1\}}(x)$ does not have the monomial  $X$, then   this monomial  is not identically zero since $F(c)$ is not identically zero  as a polynomial in variables $c_1,\dots,c_n$.
	
	Suppose that the summation $\sum_{ i<p} (-1)^{i+1} x_{a_i,b_1} \cdot p_{A \setminus \{a_i\}, B\setminus \{b_1\}}(x)$  has the monomial~$X$, that is, there exist permutations $\sigma_1,\dots,\sigma_{\ell} \in \mathfrak{S}_k$ such that $a_{\sigma_{z}(d)}$ is reachable from $b_d$ for all $d \in [k]$ and $z \in [\ell]$, and moreover there exist suitable minimal paths in $\mathcal P_{w,h}(a_{\sigma_z}(d),b_d)$ generating the monomial~$X$.
Here, we note that because of the assumption on $p$, we have
\begin{equation}\label{equation_sigma_less_than_p}
\sigma_1(1),\dots,\sigma_{\ell}(1) < p.
\end{equation}
Then the coefficient of $X$ in $p_{A,B}(x)$ is expressed as
	\[
	\frac{(-1)^{p+1}F(c)}{(c_{w(a_p)} - c_{w(b_1)}) \tilde{G}(c)} + \frac{J(c)}{K(c)}
	= \frac{(-1)^{p+1}F(c) K(c) + (c_{w(a_p)} - c_{w(b_1)}) \tilde{G}(c) J(c)}{(c_{w(a_p)} - c_{w(b_1)}) \tilde{G}(c) K(c)},
	\]
where $J(c) = F_{\sigma_1}(c) \cdots F_{\sigma_{\ell}}(c)$ and $K(c) = G_{\sigma_1}(c) \cdots G_{\sigma_{\ell}}(c)$ which are not identically zero polynomials.
By~\eqref{G_not_divisible_by_cwa} and~\eqref{equation_sigma_less_than_p}, we have 	$(c_{w(a_p)} - c_{w(b_1)}) \nmid K(c)$.
	Because $(c_{w(a_p)} - c_{w(b_1)}) \nmid F(c)$,
	the polynomial $(-1)^{p+1}F(c) K(c) + (c_{w(a_p)} - c_{w(b_1)}) G(c) J(c)$ is not identically zero  as a polynomial in variables $c_1,\dots,c_n$.
	Hence the result follows.
\end{proof}

\begin{remark}
In the proof of Theorem~\ref{thm_reachable_nonzero_minor}, we choose a monomial and  prove that its coefficient  is  nonzero for almost  all $c_1, \dots, c_n$: As a polynomial in $c_1, \dots, c_n$, the coefficient is a nonzero polynomial in $c_1, \dots, c_n$.
Therefore, the set of distinct numbers $c_1,\dots,c_n$ satisfying the theorem is a dense subset. We believe that this subset is the whole set; however we could not find a proof of this.
\end{remark}

\subsection{Proof of Theorem~\ref{thm_support_of_Owh}}\label{sec_proof_of_support_theorem}
We first notice that the equivariant cohomology of Hessenberg variety $\Hess(S,h)$ does not depend on the choice of $S$ as we have observed in the description of the GKM graph in Example~\ref{example_Hess_GKM}. In addition, for any regular semisimple elements $S$ and $S'$, we have $H_{T}^{\ast}(\Hess(S,h)) \cong H_T^{\ast}(\Hess(S',h))$. Accordingly, it is enough to consider a certain regular semisimple element $S$ to provide a proof. From now on, we suppose that $S$ satisfies Theorem~\ref{thm_reachable_nonzero_minor}, that is, for a given lower triangular matrix $x$ satisfying $\dot{w}x \in \Owho{w}$, a minor $p_{A,B}(x)$ is not identically zero if and only if $A$ is reachable from $B$.

To give a proof of Theorem~\ref{thm_support_of_Owh}, we recall  the Pl\"ucker embedding of $G/B$.
For an element $g = (g_{i,j}) \in G = \GL_n(\C)$ and $\uni = (i_1,\dots,i_j) \in I_{j,n}$, the $\uni$th Pl\"ucker coordinate $p_{\uni}(g)$ of $g$ is given by the $j \times j$ minor of $g$ with row indices $i_1,\dots,i_j$ and the column indices $1,\dots,j$. That is, with the notation~\eqref{eq_p_AB_x}, we have
\begin{equation}\label{eq_Plucker_coordinates}
p_{\underline{i}}(g) = p_{\underline{i}, [j]}(g).
\end{equation}
The Pl\"ucker embedding $\psi$ is defined to be
\[
\psi \colon G/B \to \prod_{j=1}^{n-1} \C P^{{n \choose j} -1}, \quad gB \mapsto \prod_{j=1}^{n-1} \left[p_{\uni}(g)\right]_{\uni \in I_{j,n}}.
\]
The Pl\"ucker embedding $\psi$ is well-defined, and it is $T$-equivariant with respect to the action of $T$ on $\prod_{j=1}^{n-1} \C P^{{n \choose j}-1}$ given by
\[
(t_1,\dots,t_n) \cdot \left[p_{\uni}\,\right]_{\uni \in I_{j,n}} \colonequals \left[t_{i_1} \cdots t_{i_j} \cdot p_{\uni}\,\right]_{\uni \in I_{j,n}}
\]
for $(t_1,\dots,t_n) \in T$ and $\uni = (i_1,\dots,i_j)$.
We note that for a fixed point $\dot{w}B \in G/B$ and for $\underline{i} \in I_{j,n}$, we obtain
\begin{equation}\label{eq_plucker_fixed_point}
p_{\underline{i}}(\dot{w}) \neq 0 \iff
\underline{i} = w^{(j)}.
\end{equation}
\begin{example}
	Suppose that $G = \GL_3(\C)$. For an element $g = (g_{i,j}) \in G$, the image $\psi(gB)$ is given by
	\[
	\begin{split}
	&([p_1(g): p_2(g): p_3(g)], [p_{12}(g): p_{13}(g): p_{23}(g)])  \\
	&\qquad = ([g_{1,1}: g_{2,1}: g_{3,1}], [g_{1,1}g_{2,2} - g_{2,1}g_{1,2}: g_{1,1}g_{3,2} - g_{3,1}g_{1,2}: g_{2,1}g_{3,2} - g_{3,1}g_{2,2}]).
	\end{split}
	\]
\end{example}

Before presenting a proof of Theorem~\ref{thm_support_of_Owh}, we recall the following result of Gelfand and Serganova~\cite{GelfandSerganova}.
\begin{lemma}[{\cite[Proposition~5.2.1]{GelfandSerganova}}]\label{lemma_GelfandSerganova}
	Let $gB \in G/B$, and let
	\[
	L_{g} = \left\{ \uni \in \bigcup_{j=1}^{n} I_{j,n} \mid p_{\uni}(g) \neq 0 \right\}.
	\]
	Then, for $u \in \mathfrak{S}_n$, a point $\dot{u}B$ is contained in $\overline{T \cdot gB}$ if and only if $u^{(j)} \in L_g$ for all $1 \leq j\leq n-1$.
\end{lemma}

\begin{proof}[Proof of Theorem~\ref{thm_support_of_Owh}]
	We first recall Theorem~\ref{thm_reachable_nonzero_minor}.
	Let $x = (x_{i,j})$ be a lower triangular matrix having $1$ on its diagonal. Suppose that $\dot{w}x \in \Owho{w}$. Then, for $\underline{i} = (i_1,\dots,i_j) \in I_{j,n}$
\begin{equation}\label{eq_Omega_condition}
p_{\underline{i}}(x) \text{ is not identically zero } \iff (i_1,\dots,i_j) \in J_{w,h,j}.
\end{equation}
 Moreover, $p_{\underline{k}}(\dot{w}x)$ is not identically zero if and only if $\underline{k} = w \cdot (i_1,\dots,i_j)$ for some $(i_1,\dots,i_j) \in J_{w,h,j}$, that is,
\[
\Omega_{w,h}\subset \bigcap_{j=1}^{n-1} \{ gB \in \Hess(S,h) \mid
		p_{w \cdot \underline{i}}(g) = 0 \text{ for }\underline{i}\in I_{j,n} \setminus J_{w,h,j} \}.
\]
Accordingly, we have
\begin{equation}\label{eq_Omega_contained_in_Awh}
\Omega_{w,h}^T \subset \{ u \in \mathfrak{S}_n \mid p_{w \cdot \underline{i}}(\dot{u}) = 0 \text{ for } \underline{i} \in I_{j,n} \setminus J_{w,h,j} \text{ and for }1 \leq j \leq n-1\}.
\end{equation}
We first note that $p_{w \cdot \underline{i}}(\dot{u}) \neq 0$ if and only if $w \cdot \underline{i} = u^{(j)}$ because of~\eqref{eq_plucker_fixed_point}. Then, we can see that a permutation $u$ in the right hand side of~\eqref{eq_Omega_contained_in_Awh} satisfies the condition $u^{(j)}=w\cdot \underline{i}$ for some $\underline{i} \in J_{w,h,j}$ for all $1\leq j\leq n-1$,
and  hence is in the set $A_{w,h}$. This proves $\Omega_{w,h}^T \subset A_{w,h}$.

	To prove the opposite inclusion, by~\eqref{eq_Omega_condition}, there exists a point $\dot{w}x \in \Owho{w}$ such that
	\[
	L_{\dot{w}x} = \bigcup_{j=1}^n \{ w \cdot (i_1,\dots,i_j) \mid (i_1,\dots,i_j) \in J_{w,h,j} \}.
	\]
	By  Lemma~\ref{lemma_GelfandSerganova}, a point $\dot{u}B$ is contained in the closure $\overline{T \cdot (\dot{w}x)B}$ if and only if $u^{(j)} \in  \{ w \cdot (i_1,\dots,i_j)  \mid (i_1,\dots,i_j) \in J_{w,h,j} \}$ for all $1 \leq j \leq n-1$. This proves $\Omega_{w,h}^T \supset A_{w,h}$ and the result follows.
\end{proof}

As an application of Theorem~\ref{thm_support_of_Owh}, we compute some classes $\swh{w}$ in the following example.
\begin{example}\label{example_2444_2143_2413_2341}
Let $h = (2,4,4,4)$. For $w = 2143, 2413, 2341$, one can obtain the supports of classes~$\swh{w}$ by Theorem~\ref{thm_support_of_Owh}. Indeed, one may see that each support forms a \emph{regular} subgraph of the GKM graph of $\Hess(S,h)$. See Figure~\ref{fig_GKM_2444} for the GKM graph of $\Hess(S,h)$. Here, a regular graph is a graph such that each vertex has the same number of neighbors. One can check that the values $\swh{w}(v)$ for $v \in \supp(\swh{w})$ are completely determined by the GKM condition with the value $\swh{w}(w)$ that is given by Proposition~\ref{prop_property_of_swh} as shown in Table~\ref{table_class_description_2444}.
\end{example}
\begin{table}[bt]
\begin{small}
\begin{tabular}{c|cccccc}
\toprule
$v$ & $1234$ & $1243$ & $1324$ & $1342$ & $1423$ & $1432$ \\
\midrule
$\swh{2143}(v)$ & 0 & 0 & 0 & 0 & 0 & 0 \\
$\swh{2413}(v)$ &  0 & 0 & 0 & 0 & 0 & 0 \\
$\swh{2341}(v)$  & 0 & 0 & 0 & 0 & 0 & 0 \\
\midrule
\midrule
$v$ & $2134$ & $2143$ & $2314$ & $2341$ & $2413$ & $2431$ \\
\midrule
$\swh{2143}(v)$	& 0 & $t_{1,2}t_{3,4}$ & 0 & $-t_{2,3}t_{1,4}$ & $-t_{2,4}t_{3,4}$
	& $-t_{2,4}t_{1,4}$ \\
$\swh{2413}(v)$ & 0 & 0 & 0 & 0 & $t_{1,4}t_{3,4}$ & $t_{1,4}t_{3,4}$ \\
$\swh{2341}(v)$  & 0 & 0 & 0 & $t_{1,3}t_{1,4}$ & 0 & $t_{1,3}t_{1,4}$ \\
\midrule
\midrule
$v$ & $3124$ & $3142$ & $3214$ & $3241$ & $3412$ & $3421$ \\
\midrule
$\swh{2143}(v)$  & 0 & 0 & 0 & 0 & 0 & 0 \\
$\swh{2413}(v)$ & 0 & 0 & 0 & 0 & 0 & 0 \\
$\swh{2341}(v)$ & 0 & 0 & 0 & $t_{1,2}t_{1,4}$ & 0 & $t_{1,2}t_{1,4}$\\
\midrule
\midrule
$v$ & $4123$ & $4132$ & $4213$ & $4231$ & $4312$ & $4321$ \\
\midrule
$\swh{2143}(v)$ & 0 & 0 & 0 & 0 & 0 & 0 \\
$\swh{2413}(v)$ & 0 & 0 & $-t_{1,2}t_{2,3}$ & $-t_{1,2}t_{2,3}$ & 0 & 0 \\
$\swh{2341}(v)$ & 0 & 0 & 0 & $t_{1,2}t_{1,3}$ & 0 & $t_{1,2}t_{1,3}$\\
\bottomrule
\end{tabular}
\end{small}
\caption{Description of classes $\swh{w}$ for $w = 2143, 2413, 2341$ and $h = (2,4,4,4)$.
Here, we set $t_{i,j} \colonequals t_i - t_j$.}\label{table_class_description_2444}
\end{table}

\section{Symmetric group action}\label{sec:hong}
 \label{subsec_4.1}
%
\subsection{Geometric interpretation of the dot action}

In~\cite{T1}, Tymoczko introduces an action of the symmetric group $\mathfrak S_n$ on $\bigoplus_{v \in \mathfrak S_n} \mathbb C[t_1, \dots, t_n]$ as follows: For $u \in \mathfrak S_n$ and $ p=(p(v))_{v \in \mathfrak S_n} \in \bigoplus_{v \in \mathfrak S_n} \mathbb C[t_1, \dots, t_n]$,
$$(u\cdot p)(v) \coloneqq
 p(u^{-1}v) (t_{u(1)}, \dots, t_{u(n)}).  $$
Moreover, this action induces an action of the symmetric group $\mathfrak S_n$   on the equivariant cohomology~$H_T^*(G/B)$, which is called the \emph{dot action}.

Geometrically, this action is equivalent to the one induced by the left multiplication of the symmetric group $\mathfrak S_n$ on $G/B$:   $u\cdot gB = u^{-1} gB$ for $u \in \mathfrak S_n$ and $gB \in G/B$ (Corollary~2.10 of \cite{T1}). More precisely, the action of $\mathfrak S_n$ on $ET \times_T G/B$   defined by $u\cdot(z,gB) = (z u, u^{-1}gB)$ for $z \in ET $ and $gB \in G/B$ restricts to $ET \times_T (G/B)^T$ and
we have a commutative diagram
\begin{eqnarray*}
\xymatrix{
H_T^{\ast}  (G/B) \ar[d]^{j^*}\ar[r]^{u^*} & H_T^{\ast} (G/B ) \ar[d]^{j^*} \\
H_T^{\ast}\left((G/B)^T\right) \ar[r]^{u^*} & H_T^{\ast}\left((G/B)^T\right)
}
\end{eqnarray*}
where $u^*$ is the induced map from the action of $u \in \mathfrak S_n$ on $ET \times_T G/B$ and on $ET \times_T (G/B)^T$, and $j^*$ is the induced map from the inclusion $j: ET \times_T (G/B)^T
\rightarrow ET \times_T G/B$. %
Now the action $u^*:H_T^{\ast} \left((G/B)^T\right) \rightarrow  H_T^{\ast}\left((G/B)^T\right)$ is exactly  the one Tymoczko defined on $H^{\ast}_T(G/B)$.

Under the isomorphism
\begin{equation*}
cl_{G/B}^T \colon A_{T}^{\ast}(G/B) \stackrel{\cong}{\longrightarrow} H_T^{2\ast}(G/B)
\end{equation*}
the dot action  corresponds to the action of $\mathfrak S_n$ on the equivariant Chow ring $A_T^{\ast}(G/B)$ given as follows: For an element $u \in \mathfrak S_n$ and  for an element $[Z]$ of $A_{T}^{\ast}(G/B) $ represented by a $T$-invariant subvariety $Z$ of $G/B$, the element $u \cdot [Z]$ of $A_{T}^{\ast}(G/B) $  is represented by the $T$-invariant subvariety $u^{-1}Z$.
This defines an action of $\mathfrak S_n$ on  $A_T^{\ast}(G/B)$  due to the following theorem.

\begin{theorem} [Theorem~2.1 of \cite{Br1}] \label{thm:generators and relations}
	Let $X$ be a variety with an action of a complex torus $T$. Let $M$ be the character group of $T$ and $\mathcal{S}$ be the character ring of $T$.  Then the $\mathcal{S}$-module $A_*^T(X)$ is generated by the classes $[Y]$ of  closed $T$-invariant subvarieties $Y$ of $X$ with relations
	$$[\mathrm{div}_Y(f)] -\chi[Y]$$
	where $f$ is a non-constant rational function on $Y$ which is a $T$-eigenvector of weight $\chi$.
	
\end{theorem}

Here, for a $T$-invariant subvariety $Y$ and a rational function $f$ on it, $[\mathrm{div}_Y(f)]$ denotes the class in $A_{\ast}^T(X)$ defined by  the divisor of $f$:
$$[\mathrm{div}_Y(f)] = \sum \mathrm{ord}_V(f) [V],
$$
the sum  is over all codimension one subvarieties  of $Y$.  For example, if $f$ is  a dominant morphism $Y \rightarrow \mathbb {C}P^1$, then $[\mathrm{div}_Y(f)]$ is the cycle  $[f^{-1}(0)] - [f^{-1}(\infty)]$, where $0=[1:0]$ and $\infty=[0:1]$ are the zero and infinite points of $\mathbb {C}P ^1$ (Example~1.5.1 of \cite{fulton}). \\

The edge set of the GKM graph of the full flag variety $G/B$  consists of $(w \rightarrow v)$ satisfying that  $v=w s_{j,k}$ for some $s_{j,k}$. 
By the symbol $w \rightarrow v$   we mean that the pair $(w,v)$ satisfies $\ell(w) >\ell(v)$ in addition.
By applying Theorem~\ref{thm:generators and relations} to the case $X = G/B$ and Schubert varieties, the following result  was derived.

\begin{proposition} [Proposition~6.2 of \cite{Br1}, Proposition~3.5 of \cite{T1}] \label{prop:si action in full flag}
	Assume that $\Hess(S,h)$ is the full flag variety $G/B$ of type $A$, in other words, $h(i)=n$ for all $ i \in [n]$.
	
	\begin{enumerate}
		\item If $s_iw \rightarrow w$, then $s_i \cdot \sigma_{w,h} - \sigma_{w,h}=0$.
		\item If $w \rightarrow s_i w$, then $s_i \cdot  \sigma_{w,h} - \sigma_{w,h}= (t_{i+1} - t_i) \sigma_{s_iw,h}$.
	\end{enumerate}
	
\end{proposition}

Let $h:[n] \rightarrow [n]$ be a Hessenberg function.
The symmetric group $\mathfrak S_n$ acts also  on $H_T^*(\Hess (S,h))$: For $u \in \mathfrak S_n$ and $p=(p(v))_{v \in \mathfrak S_n} \in \bigoplus_{v \in \mathfrak S_n} \mathbb C[t_1, \dots, t_n]$, if $p \in H_T^*(\Hess(S,h))$,  then $u\cdot p \in H_T^*(\Hess(S,h))$ (cf.~\cite[\S4.2]{T2}).
In the rest of this section, we  extend Proposition~\ref{prop:si action in full flag} to Hessenberg varieties $\Hess(S,h)$.
(See Proposition~\ref{prop: edge deleted} and Proposition~\ref{prop: edge remaining 2} below.)

Recall that we have a coordinate chart $x=(x_{{\alpha} ,{\beta}}) \in U^-  \mapsto \dot{w}xB \in \{\dot{w}x \in G \mid x=(x_{{\alpha} ,{\beta}}) \in U^-\}/B$ of $G/B$ around $\dot{w}B$, where $U^-$ is the set of lower triangular matrices having 1 on its diagonal.
Here and from now on, when we use the coordinate chart $x=(x_{{\alpha}, {\beta}})$, we identify $x \in U^-$ with $xB \in U^-B/B \subset G/B$.
   Then we have
 	$$w^{-1}\Omega_{w }^{\circ}=  \{(x_{{\alpha} ,{\beta}}) \in U^- \mid x_{{\alpha} ,{\beta}} =0  \text{ for all } ({\alpha}, {\beta}) \text{ with }  w({\alpha}) < w({\beta}) \}.$$
  By Proposition~\ref{prop_minus_cell_description}, $w^{-1}(\Omega_{w }^{\circ} \cap \Hess(S,h))$ is given by
 	\begin{eqnarray} \label{eq defining ideal description}
  \{(x_{{\alpha} ,{\beta}}) \in w^{-1}\Omega_w^{\circ} \mid   f_{{\alpha},{\beta}}^w =0  \text{ for all } ({\alpha},{\beta}) \text{ with } {\alpha} > h({\beta})\},
   \end{eqnarray}
 	where
 \begin{equation} \label{eq defining ideal}
   f^w_{{\alpha},{\beta}}= (c_{w({\alpha})}- c_{w({\beta})})x_{{\alpha},{\beta}}  +
 	\sum_{t=1}^{{\alpha}-{\beta}-1} \sum_{ {\alpha} > \gamma_1 > \gamma_2 > \dots > \gamma_t > {\beta}} (-1)^t (c_{w(\gamma_t)} - c_{w({\beta})}) x_{{\alpha}, \gamma_1} x_{\gamma_1, \gamma_2} \cdots x_{\gamma_t, {\beta}}
 \end{equation}
 and $c_1, \dots, c_n$ are eigenvalues of the diagonal matrix $S$.
Furthermore, as in~\eqref{eq_xij_monomial_expression}, we may express $w^{-1}(\Omega_{w }^{\circ} \cap \Hess(S,h))$  as the graph
 \begin{eqnarray} \label{eq graph expression}
 \{(x_{c,d}, x_{a,b}) \mid x_{a ,b} = g_{a,b}^w(x_{c  ,d})\}
 \end{eqnarray}
 of a function $g^w =(g_{a ,b}^w )$ from $\{(x_{c,d}) \mid  c>d, w(c) >w(d) \text{ and }  c \leq h(d) \}$ to $\{(x_{a,b})   \mid  a>b,  w(a) >w(b) \text{ and }  a> h(b) \}$.
 Here and afterwards, the indices for the variables $x_{c,d}$ are elements of the first set, and the indices for the variables $x_{a,b}$ are elements of  the second set.
  More precisely, each $g_{a,b}^w $ is a polynomial in variables $x_{c,d}$,
\begin{eqnarray} \label{eq coefficient}
g_{a,b} ^w    = \sum_{\gamma_{\bullet} \in \mathcal P_{w,h}(a,b)} C^w (a,\gamma_1, \dots, \gamma_t,b)x_{a, \gamma_1} \dots x_{\gamma_t,b},
\end{eqnarray}
whose coefficients $C^w (a,\gamma_1, \dots, \gamma_t,b)$ are Laurent polynomials in the eigenvalues $c_1, \dots, c_n$ of the diagonal matrix $S$.
 Note that monomials appearing in the expression~\eqref{eq coefficient} do  not depend on a choice of $S$ but only on the graph $G_{w,h}$, but their coefficients depend on a choice of $S$.

 For example, if we choose $S'=s_iSs_i=\mathrm{diag}(c_1', \dots, c_n')$, then~\eqref{eq defining ideal} becomes
 \begin{equation*} 
    {f'}^w_{{\alpha},{\beta}}= (c'_{w({\alpha})}- c'_{w({\beta})})x_{{\alpha},{\beta}}  +
 	\sum_{t=1}^{{\alpha}-{\beta}-1} \sum_{ {\alpha} > \gamma_1 > \gamma_2 > \dots > \gamma_t > {\beta}} (-1)^t (c'_{w(\gamma_t)} - c'_{w({\beta})}) x_{{\alpha}, \gamma_1} x_{\gamma_1, \gamma_2} \cdots x_{\gamma_t, {\beta}},
 \end{equation*}
and  $w^{-1}(\Omega_{w }^{\circ} \cap \Hess(S',h))$  can be described as the graph of the function ${g'}^{w}=({g'}^{w}_{a,b})$ obtained from $g^w=(g^w_{a,b})$ by substituting $ c_1', \dots, c_n' $ for $c_1, \dots, c_n$ in $C^w (a,\gamma_1, \dots, \gamma_t,b)$.

\begin{proposition} \label{prop: the induced action on chow ring}
The  action of  $ \mathfrak S_n$ on $H_T^{\ast}(\Hess(S,h))$ induces an action on $A_T^{\ast}(\Hess(S,h))$. In particular, for a simple reflection $s_i$, we have
$$ s_i\cdot [\overline{\Omega_w^{\circ} \cap \Hess(S,h)}]
=[\overline{(s_i^{-1}  \Omega_w^{\circ}) \cap \Hess(  S , h) }]$$
 in  $A_T^{\ast}(\Hess(S,h))$.
\end{proposition}

\begin{proof}
Let $\mathcal X$ be a smooth subvariety of $G/B$ invariant under the action of $T$.  Then the following diagram
\begin{eqnarray*}
\xymatrix{
H_T^{\ast} \left(s_i\mathcal X\right) \ar[d]^{j_{s_i}^*}\ar[r]^{s_i^*} & H_T^{\ast}\left(\mathcal X\right) \ar[d]^{j^*} \\
H_T^{\ast}\left((s_i \mathcal X)^T\right) \ar[r]^{{s_i}^*} & H_T^{\ast}\left(\mathcal X^T\right)
}
\end{eqnarray*}
is commutative, where $j_{s_i}:(s_i \mathcal X)^T \rightarrow s_i\mathcal X$ and $j: \mathcal X ^T\rightarrow \mathcal X $ are the inclusions.

 Putting  $\mathcal X = \Hess(S,h)$ into the above diagram,
we get
\begin{eqnarray*}
\xymatrix{
H_T^{\ast} \left( \Hess(S,h)\right) \ar[d]^{j^*} &H_T^{\ast} \left(s_i \Hess(S,h)\right) \ar[d]^{j_{s_i}^*}\ar[r]^{{s_i}^*} & H_T^{\ast}\left(\Hess(S,h)\right) \ar[d]^{j^*} \\
H_T^{\ast}\left(  \Hess(S,h) ^T\right)\ar@{=}[r] &H_T^{\ast}\left((s_i\Hess(S,h))^T\right) \ar[r]^{{s_i}^*} & H_T^{\ast}\left(\Hess(S,h)^T\right).
}
\end{eqnarray*}
Here, the equality $H_T^{\ast}\left(  \Hess(S,h) ^T\right)=H_T^{\ast}\left((s_i \Hess(S,h))^T\right)$ follows from the fact that $\Hess(S,h) ^T $  is equal to $ (G/B)^T=\mathfrak S_n $.

We claim that   $ \overline{\Omega_w^{\circ} \cap \Hess(S,h)}$  and $  \overline{\Omega_w^{\circ} \cap s_i\Hess(  S , h)}$   define the same class in $A_T^{\ast }(\Hess(S,h)^T)=A_T^{\ast }((s_i\Hess( S ,h))^T)$.
  Since $ \Omega_{w }^{\circ} \cap s_i\Hess(S,h)  = \Omega_w^{\circ} \cap \Hess(s_iSs_i^{-1}, h) $, we get a   description of $w^{-1}(\Omega_{w }^{\circ} \cap s_i\Hess(S,h))$ as a graph of some function $g'=(g'_{a,b})$ similar to~\eqref{eq graph expression}, simply by replacing the entries of the diagonal matrix $S = \text{diag}(c_1,\dots,c_n)$ with the entries of the diagonal matrix  $S'\coloneq s_iSs_i^{-1} = \text{diag} (c_1',\dots,c_n'
)$, where $c_k' = c_{s_i(k)}$.
From this description, we get a $T$-equivariant isomorphism  from $\Omega_{w}^{\circ} \cap \Hess(S,h)$ onto $\Omega_{w}^{\circ} \cap \Hess(S',h)$,
 which extends to a $T$-equivariant isomorphism from $\overline{\Omega_{w}^{\circ} \cap \Hess(S,h)}$ onto $\overline{\Omega_{w}^{\circ} \cap \Hess(S',h)}$.

As we see in Theorem~\ref{thm_support_of_Owh}, the $T$-fixed point set of the closure $\overline{\Omega_w^{\circ} \cap \Hess(S,h)} $ does not depend on a choice of $S$. In particular,
  $ \overline{\Omega_w^{\circ} \cap \Hess(S,h)} $  and $ \overline{\Omega_w^{\circ} \cap  \Hess(  S' , h)} $  have the same $T$-fixed point set.
  Furthermore, the tangent spaces of $\Hess(S,h)$ and $\Hess(S',h)$ at a $T$-fixed point $u$ have the same weights.
   Recall that, by Lemma~\ref{lemma_Brion_equivariant_mutliplicities} (4), $[\overline{\Omega_w^{\circ} \cap \Hess(S,h)}]_u$ is determined by the product $\chi_1 \dots  \chi_m$ of weights of the tangent space of $\Hess(S,h)$ at $u$  and $e_u[\overline{\Omega_w^{\circ} \cap \Hess(S,h)}  ]$.
    Note that for  a $T$-fixed point $x$ of a $T$-invariant subvariety $Y$ of $X$,  $e_x[Y]$ does not depend on the embedding of $Y$ into $X$ but  only on the action of $T$ on $Y$ (\cite[Theorem~4.2 (2)]{Br1}).
     It follows that $e_u[\overline{\Omega_w^{\circ} \cap \Hess(S,h)}  ]$ equals $e_u[\overline{\Omega_w^{\circ} \cap  \Hess(  S' , h)} ] $ for any $T$-fixed point $u$.
Consequently, $ \overline{\Omega_w^{\circ} \cap \Hess(S,h)}  $  and $\overline{\Omega_w^{\circ} \cap s_i\Hess(  S , h)} $   define the same class in $A_T^{\ast }(\Hess(S,h)^T)=A_T^{\ast }((s_i\Hess( S ,h))^T)$.

 Now the translate $s_i^{-1}( \Omega_w^{\circ} \cap s_i\Hess( S , h)) $ of  $\Omega_w^{\circ} \cap s_i\Hess( S , h)$ by $s_i^{-1}$ is $(s_i^{-1}\Omega_w^{\circ}) \cap \Hess(S,h)$, and thus the  action of $s_i \in \mathfrak S_n$ on $A_T^{\ast}(\Hess(S,h))$ is given by
\[ [\overline{\Omega_w^{\circ} \cap \Hess(S,h)}]  \in  A_T^{\ast}(\Hess(S,h)) \mapsto [ (\overline{s_i^{-1}  \Omega_w^{\circ}) \cap \Hess(  S  , h)  }] \in A_T^{\ast}(\Hess(S,h)). \qedhere
\]
 \end{proof}

Let $w,v \in \mathfrak S_n$ be such that $v=w s_{j,k}$ and $\ell(w) >\ell(v)$. Then  $(w \rightarrow v)$ is contained in the edge set of the  GKM graph of the full flag variety $G/B$. As in the case of $G/B$, by   $w \rightarrow v$ we mean that    $(w \rightarrow v)$ is contained in the edge set of the  GKM graph of $\Hess(S,h)$. In addition, by $w \dasharrow v$ we mean that  $(w \rightarrow v)$ is not contained in the edge set of the  GKM graph of $\Hess(S,h)$.

\begin{lemma} \label{lem:si and Gwh}  Let $w $ be an element of $\mathfrak S_n$ and let $s_i  $ be a simple reflection.
	
	\begin{enumerate}
		\item If $ w  \dashrightarrow s_iw$, then $E(G_{w,h}) = E(G_{s_iw,h})$ and $\ell_h(s_iw) = \ell_h(w)$ {\rm(}see Figure~\ref{fig_graphs_Gwh si deleted}{\rm)}. 
		
		\item If $ w  \rightarrow s_iw$, then $E(G_{s_iw,h})=E(G_{w,h}) \cup \{w^{-1}(i+1)\rightarrow w^{-1}(i)\}$ and $\ell_h( w) = \ell_h(s_iw)+1$ {\rm (}see Figure~\ref{fig_graphs_Gwh si remaining}{\rm)}.
	\end{enumerate}
\end{lemma}

\begin{proof}
	Assume that $\ell(w) >\ell(s_iw)$. Then $j \colonequals w^{-1}(i+1) <w^{-1}(i)=:k$. Furthermore, $(w \rightarrow s_iw)$ is contained in the edge set of the GKM graph of $\Hess(S,h)$ if and only if $k \leq h(j)$, or equivalently, there is an edge $j \rightarrow k$ in $G_{s_iw,h}$.  This completes the proof.
\end{proof}

\begin{figure}[t]
	\begin{subfigure}[b]{0.3\textwidth}
		\begin{tikzpicture}
		\foreach \x in {1,2,3,4}
		\node[circle,draw,inner sep=0pt,text width=5mm,align=center] (\x) at (\x,0) {\x};
		\draw[->] (2)--(3);
		\draw[->] (2) to [bend left = 45] (4);
		\draw[->] (3)--(4);
		\end{tikzpicture}
		\caption{$G_{4123,h}$.}\label{subfigure_4123}
	\end{subfigure}
	\hspace{0.5cm}
	\begin{subfigure}[b]{0.3\textwidth}
		\begin{tikzpicture}
		\foreach \x in {1,2,3,4}
		\node[circle,draw,inner sep=0pt,text width=5mm,align=center] (\x) at (\x,0) {\x};
		\draw[->] (2)--(3);
		\draw[->] (2) to [bend left = 45] (4);
		\draw[->] (3)--(4);
		\end{tikzpicture}
		\caption{$G_{3124,h}$.}\label{subfigure_3124}
	\end{subfigure}
	\caption{$h = (2,4,4,4)$ and $4123 \stackrel{s_3}{\dashrightarrow} 3124$.}\label{fig_graphs_Gwh si deleted}
\end{figure}

\begin{figure} [t]
	\begin{subfigure}[b]{0.3\textwidth}
		\begin{tikzpicture}
		\foreach \x in {1,2,3,4}
		\node[circle,draw,inner sep=0pt,text width=5mm,align=center] (\x) at (\x,0) {\x};
		\draw[->] (2)--(3);
		\draw[->] (1)--(2);
		\end{tikzpicture}
		\caption{$G_{1342,h}$.}\label{subfigure_1342}
	\end{subfigure}
	\hspace{0.5cm}
	\begin{subfigure}[b]{0.3\textwidth}
		\begin{tikzpicture}
		\foreach \x in {1,2,3,4}
		\node[circle,draw,inner sep=0pt,text width=5mm,align=center] (\x) at (\x,0) {\x};
		\draw[->] (1)--(2);
		\draw[->] (2) to [bend left = 45] (4);
		\draw[->] (2)--(3);

		\end{tikzpicture}
		\caption{$G_{1243,h}$.}\label{subfigure_1243}
	\end{subfigure}
	\caption{$h = (2,4,4,4)$ and $1342 \stackrel{s_2}{\rightarrow} 1243$.}\label{fig_graphs_Gwh si remaining}
\end{figure}
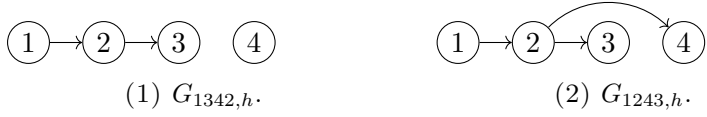

  For a simple reflection $s_i$ and $w \in \mathfrak S_n$, we have
	$$w^{-1}\Omega_{w }^{\circ}=  \{(x_{{\alpha} ,{\beta}}) \in U^- \mid x_{{\alpha} ,{\beta}} =0 \text{ for all } ({\alpha},{\beta}) \text{ with } w({\alpha}) < w({\beta}) \} $$
and
$$(s_iw)^{-1}\Omega_{s_iw }^{\circ}=  \{(x_{{\alpha} ,{\beta}}) \in U^- \mid x_{{\alpha} ,{\beta}} =0 \text{ for all } ({\alpha},{\beta}) \text{ with } s_iw({\alpha}) < s_iw({\beta}) \}.$$
If $\ell(w) > \ell(s_iw)$, or equivalently, $w^{-1}(i+1) < w^{-1}(i )$, then from
$$ \{(\alpha, \beta) \mid  \alpha > \beta,  w(\alpha) >w(\beta) \}  \sqcup \{(w^{-1}(i ), w^{-1}(i+1))\} = \{(\alpha,\beta)\mid \alpha >\beta, s_iw(\alpha) >s_i w(\beta) \}, $$ it follows that  $w^{-1}\Omega_{w }^{\circ}$ is contained in  $(s_iw)^{-1}\Omega_{s_iw }^{\circ}$, and is obtained by taking $x_{w^{-1}(i+1), w^{-1}(i)}=0$.

\begin{lemma} \label{free variables and defining equations}
Let $w $ be an element of $\mathfrak S_n$ and let $s_i$ be a simple reflection.
 \begin{enumerate}
 \item
If $w \dasharrow s_iw$,  then we have $$\Omega_{s_i w}^{\circ} \cap \Hess(S,h)  = s_i(\Omega_{w}^{\circ} \cap s_i \Hess(S,h)) .$$
 \item If $w \rightarrow s_iw$, then
 we have
 $$\Omega_{s_i w}^{\circ} \cap \Hess(S,h)   \supsetneq s_i(\Omega_{w}^{\circ} \cap s_i \Hess(S,h)) .$$
 \end{enumerate}

 \end{lemma}

\begin{proof}

 By replacing $w$ with $s_iw$ in~\eqref{eq defining ideal description} and~\eqref{eq defining ideal}, we get
$$(s_iw)^{-1}\Omega_{s_iw,h}^{\circ} =   \{(x_{\alpha,{\beta}} ) \in (s_iw)^{-1}\Omega_{s_iw}^{\circ} \mid   f_{{\alpha},{\beta}}^{s_iw} =0  \text{ for all } ({\alpha},{\beta}) \text{ with } {\alpha} > h({\beta})\},$$ 
	where $$f_{{\alpha},{\beta}}^{s_iw}= (c_{s_iw({\alpha})}- c_{s_iw({\beta})})x_{{\alpha},{\beta}}  +
	\sum_{t=1}^{{\alpha}-{\beta}-1} \sum_{ {\alpha} > \gamma_1 > \gamma_2 > \dots > \gamma_t > {\beta}} (-1)^t (c_{s_iw(\gamma_t)} - c_{s_iw({\beta})}) x_{a, \gamma_1} x_{\gamma_1, \gamma_2} \cdots x_{\gamma_t, {\beta}}. $$
Note that $c_{s_iw({\alpha})} = c_{w({\alpha})}$ if $w({\alpha}) \not \in \{j,k\}$,
and $c_{s_iw(j)} =c_{w(k)}$ and $c_{s_iw(k)} =c_{w(j)}$, where $j\colonequals w^{-1}(i+1)$ and $k\colonequals w^{-1}(i)$.

\smallskip

(1) If $w \dasharrow s_iw$, i.e., $\ell(w) > \ell(s_iw)$ and  $w^{-1}(i ) > h(w^{-1}(i+1))$, then by Lemma~\ref{lem:si and Gwh}(1), we have
\begin{eqnarray*} \label{eq:same free variables}
\{(\alpha,\beta) \mid \alpha >\beta,  w(\alpha) >w(\beta), \text{ and } \alpha \leq h(\beta)\}   = \{(\alpha,\beta) \mid  \alpha >\beta,  s_iw(\alpha) >s_i w(\beta), \text{ and } \alpha \leq h(\beta) \}.
\end{eqnarray*}
Thus $w^{-1}\Omega^{\circ}_{w,h}$ and $(s_iw)^{-1}\Omega^{\circ}_{s_iw,h}$ have the same free variables,
and  their defining equations $f_{{\alpha} ,{\beta}}^{w}$ and $f_{{\alpha} ,{\beta}}^{s_iw}$ differ  only by the coefficients of monomials.
Furthermore, for $S'=s_iSs_i=\mathrm{diag}(c_1', \dots, c_n')$, we have  $c'_{ w({\alpha})} = c_{w({\alpha})}$ if $w({\alpha}) \not \in \{j,k\}$,
and $c'_{ w(j)} =c_{w(k)}$ and $c'_{ w(k)} =c_{w(j)}$, where $j\colonequals w^{-1}(i+1)$ and $k\colonequals w^{-1}(i)$.
Therefore, we have $(s_iw)^{-1}(\Omega_{s_iw}^{\circ} \cap \Hess(S,h)) = w^{-1}(\Omega_{w}^{\circ} \cap \Hess(s_iSs_i,h)) $ and thus
$$\Omega_{s_i w}^{\circ} \cap \Hess(S,h)  = s_i(\Omega_{w}^{\circ} \cap s_i \Hess(S,h)) .$$

\smallskip

(2) If  $w \rightarrow s_iw$, i.e., $\ell(w) > \ell(s_iw)$ and  $w^{-1}(i ) \leq h(w^{-1}(i+1))$, then by Lemma~\ref{lem:si and Gwh}(2), we have
 \begin{eqnarray*} \label{eq:not same free variables}
 && \{(\alpha,\beta)\mid \alpha >\beta, w(\alpha) >w(\beta), \text{ and } \alpha \leq h(\beta)\}  \sqcup \{(w^{-1}(i ), w^{-1}(i+1))\} \\
   & & \quad =   \{(\alpha,\beta) \mid \alpha >\beta, s_iw(\alpha) >s_i w(\beta), \text{ and } \alpha \leq h(\beta) \}.
 \end{eqnarray*}
 Thus $(s_iw)^{-1}\Omega^{\circ}_{s_iw,h}$ has one more free variable $x_{w^{-1}(i )\, w^{-1}(i+1) }$ than $w^{-1}\Omega_{w,h}^{\circ}$ does,
and  their defining equations $f_{{\alpha} ,{\beta}}^{w}$ and $f_{{\alpha} ,{\beta}}^{s_iw}$ differ  only by the coefficients of monomials.
 Therefore,
 we have $$(s_iw)^{-1}(\Omega_{s_iw}^{\circ} \cap \Hess(S,h))  \supsetneq w^{-1}(\Omega_{w}^{\circ} \cap \Hess(s_iSs_i,h)), $$
 where the latter is defined by the equation $x_{w^{-1}(i+1), w^{-1}(i)}=0$. Consequently,
\[
\Omega_{s_i w}^{\circ} \cap \Hess(S,h)  \supsetneq s_i(\Omega_{w}^{\circ} \cap s_i \Hess(S,h)) . \qedhere
\]
\end{proof}

\begin{proposition}  \label{prop: edge deleted} Let $w $ be an element of $ \mathfrak S_n$ and $s_i$  a simple reflection.
	If $ w  \dashrightarrow s_iw$ or $s_iw \dasharrow w$, then $s_i \cdot \sigma_{w,h} =\sigma_{s_i w, h}$ in $H^{\ast}_T(\Hess(S,h))$.
	
\end{proposition}

\begin{proof}
We may assume that $w \dasharrow s_iw$. 	
	By Lemma~\ref{free variables and defining equations}(1), we have
 $$\Omega_{s_i w}^{\circ} \cap \Hess(S,h)  = s_i(\Omega_{w}^{\circ} \cap s_i \Hess(S,h)) .$$
 By Proposition~\ref{prop: the induced action on chow ring}, we have $s_i \cdot [\Omega_{w,h}] = [\Omega_{s_iw,h}]$.
\end{proof}

Let $w $ be an element of $\mathfrak S_n$ and let $s_i  $ be a simple reflection such that $w \rightarrow s_iw$.
To describe the action of $s_i$ on $\sigma_{w,h}$, we introduce some notations.
The Bia{\l}ynicki-Birula decomposition of $\Omega_{s_{i}w,h}$ is given by
$$\Omega_{s_{i}w,h}= \bigsqcup_{u \in \Omega_{s_{i}w,h}^T}(\Omega_u^{\circ} \cap \Omega_{s_{i}w,h}). $$
For $u \not = s_iw \in \Omega_{s_iw,h}^T$, the dimension of $\Omega_u^{\circ} \cap \Omega_{s_iw,h}$ is less than or equal to $\dim \Omega_{w,h}$ because $\Omega_{s_iw,h}$ is irreducible.
Define a subset $\mathcal A_{s_i,w}$ of $\Omega_{s_{i}w,h}^T $ by
$$ \mathcal A_{s_i,w} \coloneq \left \{u \in    \Omega_{s_{i}w,h}^T \cap \Omega_w^T \mid  \dim (\Omega_{u}^{\circ} \cap \Omega_{s_iw,h}) =\dim \Omega_{w,h} \text{ and  } u \dashrightarrow s_iu \right\}.
$$
For $u \in  \mathcal A_{s_i,w}$, define $\mathcal T_u$ and $\mathcal T_{s_iu}$ by the closure of   $ \Omega_u^{\circ} \cap \Omega_{s_{i}w,h}$ and $ \Omega_{s_iu}^{\circ} \cap \Omega_{s_{i}w,h}$, respectively. Then $\mathcal T_u$ and $\mathcal T_{s_iu}$ are irreducible because each of them is the closure of an affine cell.  Let $\tau_u$ and $\tau_{s_iu}$ denote the equivariant classes  in $H^*_T(\Hess(S,h))$ induced by $\mathcal T_u$ and $\mathcal T_{s_iu}$, respectively.

\begin{proposition} \label{prop: edge remaining 2} Let $w \in \mathfrak S_n$ and let $s_i  $ be a simple reflection.
	
	 \begin{enumerate}
		 \item
 If $s_i w \rightarrow w$, then  we get $s_i \cdot \sigma_{w,h} - \sigma_{w,h}=0$ in $H^{\ast}_T(\Hess(S,h))$.
 	\item If $w \rightarrow s_iw$, then we get
 $$\left( s_i \cdot \sigma_{w,h} + \sum_{u \in \mathcal A_{s_i,w}  } \tau_{s_iu} \right) - \left(\sigma_{w,h} + \sum_{u   \in \mathcal A_{s_i,w} } \tau_u \right)= (t_{i+1 }- t_{i})\sigma_{s_i w, h} \text{ in } H^{\ast}_T(\Hess(S,h)), $$
and the intersection $\mathcal A_{s_i,w} \cap s_i\mathcal A_{s_i,w}$ is empty.
 	\end{enumerate}
\end{proposition}

For a more detailed description  of the action of simple reflections
in the case when $h=(2,3,\dots, n,n)$, see  Proposition~\ref{prop_si_action_on_sigma_i}. We will give a proof of Proposition~\ref{prop: edge remaining 2} in Subsection~\ref{sect: proof of si action}.


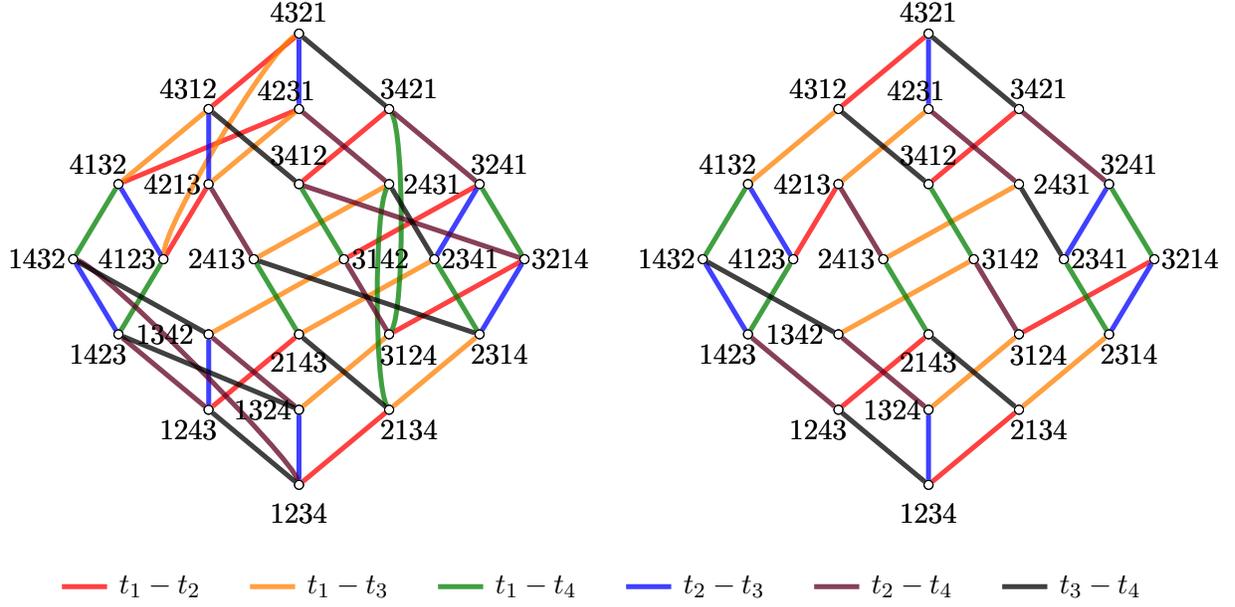
\begin{figure}
	$ 	\begin{array}{cc}
	\begin{tikzpicture}[scale=0.8]
	\tikzstyle{every node}=[font=\normalsize, label distance=-1mm]
	\matrix [matrix of math nodes,column sep={0.6cm,between origins},
	row sep={1cm,between origins},
	nodes={circle, draw, inner sep = 0pt , minimum size=1.2mm}]
	{
		& & & & & \node[label = {[label distance=-0.2cm]90:4321}] (4321) {} ; & & & & & \\
		& & & \node[label = {above left:4312}] (4312) {} ; & &
		\node[label = {[label distance = -0.2cm]160:4231}] (4231) {} ; & &
		\node[label = {above right:3421}] (3421) {} ; & & & \\
		& \node[label = {above left:4132}] (4132) {} ; & &
		\node[label = {[label distance =0cm] left:4213}] (4213) {} ; & &
		\node[label = {above:3412}] (3412) {} ; & &
		\node[label = {[label distance = 0.1cm]0:2431}] (2431) {} ; & &
		\node[label = {above right:3241}] (3241) {} ; & \\
		\node[label = {[label distance =0cm] left:1432}] (1432) {} ; & &
		\node[label = {[label distance =0cm] left:4123}] (4123) {} ; & &
		\node[label = {[label distance = 0.01cm]180:2413}] (2413) {} ; & &
		\node[label = {[label distance = 0.01cm]0:3142}] (3142) {} ; & &
		\node[label = {[label distance =0cm] right:2341}] (2341) {} ; & &
		\node[label = {[label distance =0cm] right:3214}] (3214) {} ; \\
		& \node[label = {below left:1423}] (1423) {} ; & &
		\node[label = {[label distance = 0.1cm]182:1342}] (1342) {} ; & &
		\node[label = {below:2143}] (2143) {} ; & &
		\node[label = {-30:3124}] (3124) {} ; & &
		\node[label = {below right:2314}] (2314) {} ; & \\
		& & & \node[label = {below left:1243}] (1243) {} ; & &
		\node[label = {[label distance =0cm] 180:1324}] (1324) {} ; & &
		\node[label = {below right:2134}] (2134) {} ; & & & \\
		& & & & & \node[label = {below:1234}] (1234) {} ; & & & & & \\
	};
	
	\draw[line width= 0.4ex, red,  nearly opaque ]
	(1234)--(2134)
	(1243)--(2143)
	(4123)--(4213)
	(3124)--(3214)
	(3412)--(3421)
	(4312)--(4321)
	
	(4132)--(4231)
	(3142)--(3241);
	
	\draw[line width= 0.4ex, orange,  nearly opaque ]
	(2134)--(2314)
	(2413)--(2431)
	(4213)--(4231)
	(4132)--(4312)
	(1342)--(3142)
	(1324)--(3124)
	
	(4123)  to [bend left, looseness=0.3]  (4321)
	(2143)--(2341);
	
	\draw[line width= 0.4ex,  green!50!black, nearly opaque ]
	(1432)--(4132)
	(1423)--(4123)
	(2143)--(2413)
	(2314)--(2341)
	(3214)--(3241)
	(3142)--(3412)
	
	(2134) to [bend left, looseness=0.3] (2431)
	(3124) to [bend right, looseness=0.3] (3421);
	
	\draw[line width= 0.4ex,  blue, nearly opaque ]
	(1234)--(1324)
	(1423)--(1432)
	(4123)--(4132)
	(4231)--(4321)
	(2341)--(3241)
	(2314)--(3214)
	
	(1243)--(1342)
	(4213)--(4312);

	\draw[line width=0.4ex, purple!50!black,nearly opaque ]
	(1243)--(1423)
	(2413)--(4213)
	(2431)--(4231)
	(3241)--(3421)
	(3142)--(3124)
	(1324)--(1342)	
	
	(1234) to [bend right, looseness=0.2] (1432)
	(3412)--(3214);
	
	\draw[line width=0.4ex, black,nearly opaque ]
	(1234)--(1243)
	(2134)--(2143)
	(2341)--(2431)
	(3421)--(4321)
	(3412)--(4312)
	(1342)--(1432)
	
	(1423)--(1324)
	(2413)--(2314);
	
	\matrix [matrix of math nodes,column sep={0.6cm,between origins},
	row sep={1cm,between origins},
	nodes={circle, draw, inner sep = 0pt , minimum size=1.2mm}]
	{
		& & & & & \node[label = {[label distance=-0.2cm]90:4321}] (4321) {} ; & & & & & \\
		& & & \node[label = {above left:4312}] (4312) {} ; & &
		\node[label = {[label distance = -0.2cm]160:4231}] (4231) {} ; & &
		\node[label = {above right:3421}] (3421) {} ; & & & \\
		& \node[label = {above left:4132}] (4132) {} ; & &
		\node[label = {[label distance =0cm] left:4213}] (4213) {} ; & &
		\node[label = {above:3412}] (3412) {} ; & &
		\node[label = {[label distance = 0.1cm]0:2431}] (2431) {} ; & &
		\node[label = {above right:3241}] (3241) {} ; & \\
		\node[label = {[label distance =0cm] left:1432}] (1432) {} ; & &
		\node[label = {[label distance =0cm] left:4123}] (4123) {} ; & &
		\node[label = {[label distance = 0.01cm]180:2413}] (2413) {} ; & &
		\node[label = {[label distance = 0.01cm]0:3142}] (3142) {} ; & &
		\node[label = {[label distance =0cm] right:2341}] (2341) {} ; & &
		\node[label = {[label distance =0cm] right:3214}] (3214) {} ; \\
		& \node[label = {below left:1423}] (1423) {} ; & &
		\node[label = {[label distance = 0.1cm]182:1342}] (1342) {} ; & &
		\node[label = {below:2143}] (2143) {} ; & &
		\node[label = {-30:3124}] (3124) {} ; & &
		\node[label = {below right:2314}] (2314) {} ; & \\
		& & & \node[label = {below left:1243}] (1243) {} ; & &
		\node[label = {[label distance =0cm] 180:1324}] (1324) {} ; & &
		\node[label = {below right:2134}] (2134) {} ; & & & \\
		& & & & & \node[label = {below:1234}] (1234) {} ; & & & & & \\
	};
\end{tikzpicture}%
&		\begin{tikzpicture}[scale=0.8]
\tikzstyle{every node}=[font=\normalsize, label distance=-1mm]
\matrix [matrix of math nodes,column sep={0.6cm,between origins},
row sep={1cm,between origins},
nodes={circle, draw, inner sep = 0pt , minimum size=1.2mm}]
{
	& & & & & \node[label = {[label distance=-0.2cm]90:4321}] (4321) {} ; & & & & & \\
	& & & \node[label = {above left:4312}] (4312) {} ; & &
	\node[label = {[label distance = -0.2cm]160:4231}] (4231) {} ; & &
	\node[label = {above right:3421}] (3421) {} ; & & & \\
	& \node[label = {above left:4132}] (4132) {} ; & &
	\node[label = {[label distance =0cm] left:4213}] (4213) {} ; & &
	\node[label = {above:3412}] (3412) {} ; & &
	\node[label = {[label distance = 0.1cm]0:2431}] (2431) {} ; & &
	\node[label = {above right:3241}] (3241) {} ; & \\
	\node[label = {[label distance =0cm] left:1432}] (1432) {} ; & &
	\node[label = {[label distance =0cm] left:4123}] (4123) {} ; & &
	\node[label = {[label distance = 0.01cm]180:2413}] (2413) {} ; & &
	\node[label = {[label distance = 0.01cm]0:3142}] (3142) {} ; & &
	\node[label = {[label distance =0cm] right:2341}] (2341) {} ; & &
	\node[label = {[label distance =0cm] right:3214}] (3214) {} ; \\
	& \node[label = {below left:1423}] (1423) {} ; & &
	\node[label = {[label distance = 0.1cm]182:1342}] (1342) {} ; & &
	\node[label = {below:2143}] (2143) {} ; & &
	\node[label = {-30:3124}] (3124) {} ; & &
	\node[label = {below right:2314}] (2314) {} ; & \\
	& & & \node[label = {below left:1243}] (1243) {} ; & &
	\node[label = {[label distance =0cm] 180:1324}] (1324) {} ; & &
	\node[label = {below right:2134}] (2134) {} ; & & & \\
	& & & & & \node[label = {below:1234}] (1234) {} ; & & & & & \\
};

\draw[line width= 0.4ex, red,  nearly opaque ]
(1234)--(2134)
(1243)--(2143)
(4123)--(4213)
(3124)--(3214)
(3412)--(3421)
(4312)--(4321);

\draw[line width= 0.4ex, orange,  nearly opaque ]
(2134)--(2314)
(2413)--(2431)
(4213)--(4231)
(4132)--(4312)
(1342)--(3142)
(1324)--(3124);

\draw[line width= 0.4ex,  green!50!black, nearly opaque ]
(1432)--(4132)
(1423)--(4123)
(2143)--(2413)
(2314)--(2341)
(3214)--(3241)
(3142)--(3412);

\draw[line width= 0.4ex,  blue, nearly opaque ]
(1234)--(1324)
(1423)--(1432)
(4123)--(4132)
(4231)--(4321)
(2341)--(3241)
(2314)--(3214);

\draw[line width=0.4ex, purple!50!black,nearly opaque ]
(1243)--(1423)
(2413)--(4213)
(2431)--(4231)
(3241)--(3421)
(3142)--(3124)
(1324)--(1342)	;

\draw[line width=0.4ex, black,nearly opaque ]
(1234)--(1243)
(2134)--(2143)
(2341)--(2431)
(3421)--(4321)
(3412)--(4312)
(1342)--(1432);

\matrix [matrix of math nodes,column sep={0.6cm,between origins},
row sep={1cm,between origins},
nodes={circle, draw, inner sep = 0pt , minimum size=1.2mm}]
{
	& & & & & \node[label = {[label distance=-0.2cm]90:4321}] (4321) {} ; & & & & & \\
	& & & \node[label = {above left:4312}] (4312) {} ; & &
	\node[label = {[label distance = -0.2cm]160:4231}] (4231) {} ; & &
	\node[label = {above right:3421}] (3421) {} ; & & & \\
	& \node[label = {above left:4132}] (4132) {} ; & &
	\node[label = {[label distance =0cm] left:4213}] (4213) {} ; & &
	\node[label = {above:3412}] (3412) {} ; & &
	\node[label = {[label distance = 0.1cm]0:2431}] (2431) {} ; & &
	\node[label = {above right:3241}] (3241) {} ; & \\
	\node[label = {[label distance =0cm] left:1432}] (1432) {} ; & &
	\node[label = {[label distance =0cm] left:4123}] (4123) {} ; & &
	\node[label = {[label distance = 0.01cm]180:2413}] (2413) {} ; & &
	\node[label = {[label distance = 0.01cm]0:3142}] (3142) {} ; & &
	\node[label = {[label distance =0cm] right:2341}] (2341) {} ; & &
	\node[label = {[label distance =0cm] right:3214}] (3214) {} ; \\
	& \node[label = {below left:1423}] (1423) {} ; & &
	\node[label = {[label distance = 0.1cm]182:1342}] (1342) {} ; & &
	\node[label = {below:2143}] (2143) {} ; & &
	\node[label = {-30:3124}] (3124) {} ; & &
	\node[label = {below right:2314}] (2314) {} ; & \\
	& & & \node[label = {below left:1243}] (1243) {} ; & &
	\node[label = {[label distance =0cm] 180:1324}] (1324) {} ; & &
	\node[label = {below right:2134}] (2134) {} ; & & & \\
	& & & & & \node[label = {below:1234}] (1234) {} ; & & & & & \\
};

\end{tikzpicture}\end{array}$ \\
\begin{tikzpicture}
\draw[line width= 0.4ex, red,  nearly opaque ]
(7,5)--(7.6,5) node[right, color = black, opaque] {$t_1 - t_2$};

\draw[line width= 0.4ex, orange,  nearly opaque ]

(9.5,5)--(10.1,5) node[right, color = black, opaque] {$t_1 - t_3$};

\draw[line width= 0.4ex,  green!50!black, nearly opaque ]

(12,5)--(12.6,5) node[right, color=black, opaque] {$t_1-t_4$};

\draw[line width= 0.4ex,  blue, nearly opaque ]

(14.5,5)--(15.1,5) node[right, color=black, opaque] {$t_2-t_3$};

\draw[line width=0.4ex, purple!50!black,nearly opaque ]
(17,5)--(17.6,5) node[right, color=black, opaque] {$t_2-t_4$};

\draw[line width=0.4ex, black,nearly opaque ]
(19.5,5)--(20.1,5) node[right, color=black, opaque] {$t_3-t_4$};	
\end{tikzpicture}
\caption{GKM graphs of $\Hess(S,h)$ for $h=(2,4,4,4)$ and $h = (2,3,4,4)$.}
\label{fig_GKM_2444}
\end{figure}

\begin{example} \label{example_si_sigma_wh}
We illustrate the $s_i$-action on $ \sigma_{w,h}$ when $w \rightarrow s_iw$. (See Figure~\ref{fig_GKM_2444} for GKM graphs for $h = (2,4,4,4)$ and $h = (2,3,4,4)$.)

\begin{enumerate}
\item Let $h = (2,4,4,4)$ and $w = 1243$ and $i = 3$. In this case, we obtain $s_i w = 1234$ and $\Owo{u} \cap \Owh{s_iw} = \Owho{u}$ for any $u \in \Owh{s_iw}^T= \mathfrak S_4$. Therefore,
the set $\mathcal A_{s_i,w}$ is given as follows.
\[
\begin{split}
\mathcal A_{s_3,w} &=
\{ u \in \Owh{e}^T \cap \Ow{w}^T \mid \dim(\Owo{u} \cap \Owh{e}) = \dim \Owh{w} \text{ and } u \dashrightarrow s_3 u \} \\
&= \{ u \in \frak{S}_4 \mid u \geq w, \dim \Owh{u} = \dim \Owh{w} \text{, and } u \dashrightarrow s_3 u\}\\
&= \{ u \in \mathfrak{S}_4 \mid \ell_h(u) = 1\} \cap \{ u \in \mathfrak{S}_4 \mid u \geq w \text{, and } u \dashrightarrow s_3 u\}\\
&= \{ 1243, 1324, 2134, 2314, 3124, 4123 \}
\cap \{ u \in \mathfrak{S}_4 \mid u \geq w, u \dashrightarrow s_3 u \} \\
&= \{1243, 4123\} \cap \{ u \in \mathfrak{S}_4 \mid u \dashrightarrow s_3 u \} \\
&= \{4123\}.
\end{split}
\]
Here, as mentioned in Example~\ref{example_reachability_full_flag}, we have $\Ow{w}^T = \{u \in \mathfrak{S}_n \mid u \geq w \text{~(in the Bruhat order)}\}$.
Moreover, since $\Owh{s_iw} = \Hess(S,h)$, we obtain $\mathcal{T}_{u} = \overline{\Owo{u} \cap \Owh{s_iw}} = \overline{\Owho{u}} = \Owh{u}$.
By Proposition~\ref{prop: edge remaining 2}(2), we get
\begin{equation} \label{equation_1243_s3}
(t_4 - t_3) \swh{1234} =
s_3 \cdot (\swh{1243} + \swh{4123}) - (\swh{1243} + \swh{4123}).
\end{equation}
For $u= 1234$ or $4123$,
we know the support of the class $\swh{u}$ in the GKM description by Theorem~\ref{thm_support_of_Owh}. Moreover, the subgraph obtained by restricting the GKM graph to this support is regular. Accordingly, using~\eqref{equation_1243_s3}, we obtain the values $\swh{1243}(v)$ for $v \in \Owh{1243}^T$ (cf. Example~\ref{example_2444_2143_2413_2341}).  See Table~\ref{table_1243}.

Here, we notice that the subgraph obtained by restricting the GKM graph to the support of $\swh{1243}$ is not regular, and we cannot apply the same method to find the values $\swh{1243}(v)$ for $v \in \Omega_{1243}^T$.

\item Let $h=(2,4,4,4)$ and  $w=2143$ and $i=1$.  Then the following holds.
\begin{equation}\label{equation_1243_s1}
(t_2 -t_1) \sigma_{1243,h}
= s_1 \cdot(\sigma_{2143,h} + \sigma_{2413,h}+ \sigma_{2341,h}) - (\sigma_{2143,h} + \sigma_{2413,h} + \sigma_{2341,h}).
\end{equation}
Indeed, we obtain the description of classes $\swh{2143}, \swh{2413}, \swh{2341}$ in Example~\ref{example_2444_2143_2413_2341} and of the class $\swh{1243}$ in (1).
By considering the classes $\swh{1243}$, $\swh{2143}+\swh{2413}+\swh{2341}$, and $s_1 \cdot(\swh{2143}+\swh{2413}+\swh{2341})$, we deduce that the equation~\eqref{equation_1243_s1} holds. See Table~\ref{table_1243}.
\begin{table}
\begin{small}
\begin{tabular}{c|cccccc}
\toprule
$v$ & $1234$ & $1243$ & $1324$ & $1342$ & $1423$ & $1432$ \\
\midrule
$\swh{1243}(v)$ & $0$ & $t_{3,4}$ & $0$ & $t_{2,4}$ & $t_{3,4}$ & $t_{2,4}$ \\
\midrule
$(\swh{2143}+\swh{2413} + \swh{2341})(v)$  & 0 & 0 & 0 & 0 & 0 & 0 \\
$(s_1 \cdot (\swh{2143}+\swh{2413} + \swh{2341}))(v)$
& 0 & $-t_{1,2}t_{3,4}$ & 0 & $-t_{2,4}t_{1,2}$ & $-t_{3,4}t_{1,2}$
& $-t_{1,2}t_{2,4}$\\
\midrule
\midrule
$v$ & $2134$ & $2143$ & $2314$ & $2341$ & $2413$ & $2431$ \\
\midrule
$\swh{1243}(v)$ & $0$ & $t_{3,4}$ & $0$ & $t_{1,4}$ & $t_{3,4}$ & $t_{1,4}$\\
\midrule
$(\swh{2143}+\swh{2413} + \swh{2341})(v)$
& 0 & $t_{1,2}t_{3,4}$ & 0 & $t_{1,4}t_{1,2}$ & $t_{3,4}t_{1,2}$ & $t_{1,2}t_{1,4}$ \\
$(s_1 \cdot (\swh{2143}+\swh{2413} + \swh{2341}))(v)$
& 0 & 0 & 0 & 0 & 0 & 0 \\
\midrule
\midrule
$v$ & $3124$ & $3142$ & $3214$ & $3241$ & $3412$ & $3421$ \\
\midrule
$\swh{1243}(v)$ & $0$ & $t_{2,4}$ & $0$ & $t_{1,4}$ & $t_{2,4}$ & $t_{1,4}$\\
\midrule
$(\swh{2143}+\swh{2413} + \swh{2341})(v)$
& 0 & 0 & 0 & $t_{1,2}t_{1,4}$ & 0 & $t_{1,2}t_{1,4}$\\
$(s_1 \cdot (\swh{2143}+\swh{2413} + \swh{2341}))(v)$
& 0 & $-t_{1,2}t_{2,4}$ & 0 & 0 & $-t_{1,2}t_{2,4}$ & 0 \\
\midrule
\midrule
$v$ & $4123$ & $4132$ & $4213$ & $4231$ & $4312$ & $4321$ \\
\midrule
$\swh{1243}(v)$ & $-t_{1,3}$ & $-t_{1,2}$ & $-t_{2,3}$ & $t_{1,2}$ & $t_{2,3}$ & $t_{1,3}$\\
\midrule
$(\swh{2143}+\swh{2413} + \swh{2341})(v)$
& 0 & 0 & $-t_{1,2}t_{2,3}$ & $t_{1,2} t_{1,2}$ & 0 & $t_{1,2}t_{1,3}$\\
$(s_1 \cdot (\swh{2143}+\swh{2413} + \swh{2341}))(v)$
& $t_{1,2}t_{1,3}$ & $t_{1,2}t_{1,2}$ & 0 & 0 & $-t_{1,2}t_{2,3}$ & 0 \\
\bottomrule
\end{tabular}
\end{small}
\caption{Description of classes $\swh{1243}$, $\swh{2143}+\swh{2413}+\swh{2341}$, $s_1\cdot(\swh{2143}+\swh{2413}+\swh{2341})$ for $h = (2,4,4,4)$.
Here, we set $t_{i,j} \colonequals t_i - t_j$.
}\label{table_1243}
\end{table}
\item Let $h=(2,4,4,4)$ and  $w=1423$ and  $i=3$. Then the following holds.
\begin{eqnarray*}
(t_4 -t_3) \sigma_{1324, h} = s_3\cdot(\sigma_{1423, h} + s_1 \cdot\sigma_{4213, h})- (\sigma_{1423, h} + s_1 \cdot\sigma_{4213, h})
\end{eqnarray*}
and $s_1 \cdot\sigma_{4213, h} = \sigma_{4213, h} + (t_2 -t_1)\sigma_{4123, h}$.

\item Let $h=(2,3,4,4)$ and  $w=2143$ and  $i=1$.  Then  the following holds.
\begin{eqnarray*}
(t_2 -t_1) \sigma_{1243, h} = s_1 \cdot(\sigma_{2143, h} +s_{1,3}\cdot \sigma_{2431, h}  ) -(\sigma_{2143, h} +s_{1,3}\cdot \sigma_{2431, h}   )
\end{eqnarray*}
and $s_{1,3}\cdot \sigma_{2431, h} = \sigma_{2431, h} + (t_3-t_1)\sigma_{2413, h}$.

\item Let $h=(2,3,4,4)$ and $w=1324$ and $i=2$. Then the following holds.
\begin{eqnarray*}
(t_3-t_2)\sigma_{1234, h} = s_2 \cdot (\sigma_{1324, h} + \sigma_{1342, h} + \sigma_{3124, h} + \sigma_{3412, h}) - (\sigma_{1324, h} + \sigma_{1342, h} + \sigma_{3124, h} + \sigma_{3412, h}).
\end{eqnarray*}
\end{enumerate}

\end{example}

\subsection{The associated fiber bundle $Z=SL_i \times _{B_i^-} \Omega_w$}
To describe the action of $s_i$ on $\sigma_{w,h}$ in the case when $s_iw \rightarrow w$ or $w \rightarrow s_iw$,   we  will apply Theorem~\ref{thm:generators and relations}  to a morphism    naturally associated with the pair $(s_i,w)$.

\subsubsection{ }
The minimal parabolic subgroup $P^-_{i}$ containing $B^-$  associated to the simple reflection $s_i$  is defined by $P^-_{i} =B^- \sqcup B^- \dot{s}_{i} B^-$. Then $P_{i}^-/B^-=SL_i/B_i^-$ is isomorphic to $\mathbb  C P ^1$, where $SL_i\simeq SL_2(\mathbb C)$ is the semisimple part of $P^-_{i}$ and $B_i^-$ is its Borel subgroup $SL_i \cap B^-$. Let  $U_i$ be the root group of root $\alpha_i$ and $U^{- }_i$ be the root group of root $-\alpha_i$. Then  $SL_i/B_i^-$  is decomposed as
$\{B_i^-/B_i^-\}  \sqcup U^{- }_i \dot{s}_{i}B_i^-/B_i^-  \simeq \{a \,\,  point \} \sqcup \mathbb C \simeq  \{\dot{s}_iB_i^-/B_i^-\} \sqcup U_i B_i^-/B_i^- $.

Put
$$\dot{s}_i =\left(
\begin{array}{cc}
0&-1 \\
1 &0
\end{array} \right),
z = \left(
\begin{array}{cc}
1&0 \\
{\bf z} &1
\end{array} \right) \in U_i^-, \text{ and }
y = \left(
\begin{array}{cc}
1&{\bf y}\\
0 &1
\end{array} \right) \in U_i,$$
where ${\bf z}, {\bf y} \in \mathbb C$.
Also, we set
\[
0 \colonequals \dot{s}_iB_i^-/B_i^-, \quad
\infty \colonequals B_i^-/B_i^-.
\]
Then the coordinate chart on $U_i^-\dot{s}_iB_i^-/B^-_i$ with center at $0$  (on $U_iB_i^-/B_i^-$ with center at $\infty$, respectively) is given by
\begin{eqnarray*}
{\bf z} \in \mathbb C \mapsto   z\dot{s}_iB_i^-  \qquad  ({\bf y} \in \mathbb C \mapsto  y B_i^-  , \text{ respectively}),
\end{eqnarray*}
and the transition map is given by
$${\bf z} \in \mathbb C^* \mapsto {\bf y}=\frac{1}{{\bf z}} \in \mathbb C^* . $$
 Here and after, for the sake of  simplicity,  we  only  write down the $2 \times 2$ submatrix $(g_{a,b})_{i \leq a,b \leq i+1}$ of~$g \in G$.

\subsubsection{ }
\label{sect all about Z}

Consider $Z\colonequals P^-_i \times_{B^-} \Omega_w =SL_{i}  \times_{B_i^-}\Omega_w $ and two maps $\varphi$ and $\pi$:
\begin{eqnarray*}
\xymatrix{
Z  =SL_{i}  \times_{B_i^-}\Omega_w \ar[d]^{\pi} \ar[r]^{\quad \quad   \varphi}  & SL_{i} \Omega_w \\
SL_{i} /B_i^- &  }
\end{eqnarray*}
where $\varphi([g,\xi])=g\xi$  and $\pi([g,\xi]) =gB_i^-/B_i^- \in SL_{i} /B_i^-$ for $g \in SL_{i} $ and $\xi \in \Omega_w$.
In this subsection, we list the properties of $Z$,  $\varphi$, and $\pi$,   that will be used in the proof of Proposition~\ref{prop: edge remaining 2}.

$\bullet$ $\pi:Z \rightarrow SL_i/B^-_i$ is a fiber bundle over $\mathbb CP^1$ with fibers isomorphic to $\Omega_w$. The restriction of $\pi$ to $U^{- }_i \dot{s}_{i}B_i^-/B_i^-  $  (to $U_i B_i^-/B_i^- $, respectively) is isomorphic to $\mathbb C \times \Omega_w$, whose trivialization is given by
$$[z\dot{s}_i,\xi] \mapsto ({\bf z},\xi) \qquad ([y,\zeta] \mapsto \rm{(}{\bf y},\zeta), \text{ respectively)},$$
where
$$\dot{s}_i =\left(
\begin{array}{cc}
0&-1 \\
1 &0
\end{array} \right),
z = \left(
\begin{array}{cc}
1&0 \\
{\bf z} &1
\end{array} \right) \in U_i^-,
y = \left(
\begin{array}{cc}
1&{\bf y}\\
0 &1
\end{array} \right) \in U_i, \text{ and  } \xi,\zeta \in \Omega_w,$$
and the transition map
  $({\bf z}, \xi) \mapsto ({\bf y}, \zeta)$
is given by  $${\bf y} = \frac{1}{\bf z} \text{ and }
\zeta = y^{-1}z \dot{s}_i \xi. $$
Note that
\begin{eqnarray*}
 y^{-1}z \dot{s}_i &=&
 \left(
\begin{array}{cc}
1 &-\frac{1}{\bf z}\\
0&1
\end{array}
\right)
\left(
\begin{array}{cc}
1 &0 \\
{\bf z} & 1
\end{array}
\right)
 \left(
\begin{array}{cc}
0 &-1\\
1&0
\end{array}
\right) \\
 &=&
\left(
\begin{array}{cc}
-\frac{1}{\bf z} &0 \\
1 &-{\bf z}
\end{array}
\right)
 \\
 &=& \left(
\begin{array}{cc}
1 &0 \\
-{\bf z}&1
\end{array}
\right)
\left(
\begin{array}{cc}
-\frac{1}{\bf z} &0 \\
0 & -{\bf z}
\end{array}
\right).
\end{eqnarray*}

\begin{lemma} [{\cite[Proposition~3.2.1]{BK}}] \label{Z is trivial and varphi is a projection} If $s_iw >w$, then $Z$ is isomorphic to the product $(SL_i/B_i^-) \times \Omega_w$ and $\varphi$ is the projection to the second factor.
\end{lemma}

$\bullet$    $T$  acts on $Z $  as follows: For $t \in T$, $g \in P_i^- $ and $\xi \in \Omega_w$,  $t \cdot [g,\xi] = [tg,\xi]=[tgt^{-1}, t\xi] $.
Thus the map $\varphi$ is $T$-equivariant.
The fixed point set  $Z^T$ is
$$\{[\dot{s}_{i},  u ] \mid   u  \in \Omega_w^T\} \cup \{[\dot{e},   u ] \mid  u  \in \Omega_w^T\} .$$
For $u \in \Omega_w^T$, put
\begin{eqnarray*}
\Omega_{[\dot{s}_i, u ]}^{\circ} &\coloneq& \{ [g,\xi] \in Z \mid \lim_{t \rightarrow \infty} \lambda(t) [g,\xi] = [\dot{s}_i, u ]\}  \text{ and }\\
  \Omega_{[\dot{e}, u ]}^{\circ}&\coloneq& \{ [g,\xi] \in Z \mid \lim_{t \rightarrow \infty} \lambda(t) [g,\xi] = [\dot{e} , u ]\}
  \end{eqnarray*}
  for the choice of a one-parameter subgroup $\lambda:\mathbb C^* \rightarrow \mathrm{diag}(t^{\lambda_1}, t^{\lambda_2}, \dots, t^{\lambda_n})$ given in~Subsection~\ref{sec_preliminaries}.
 Then the Bia{\l}ynicki-Birula decomposition of $Z$ is given by
 $$Z = \left(\bigsqcup_{u \in \Omega_w^T} \Omega^{\circ}_{[\dot{s}_i, u]} \right) \sqcup \left(\bigsqcup_{u \in \Omega_w^T} \Omega^{\circ}_{[ \dot{e} , u]} \right).$$

\medskip

 When $u >s_iu $ and $s_iu  \in \Omega_w^T$, let $\mathbb CP^1_{s_iu,u}$ denote the projective line in $\Omega_w$  connecting $s_iu$ and $u$, which is given by $U_i^-s_iu \sqcup \{u\} =\{s_i u\} \sqcup U_i u$.
Then $\varphi^{-1}(u)$ and $\varphi^{-1}(s_iu)$ are $T$-invariant curves  contained in $SL_i \times_{B_i^-}\mathbb CP^1_{s_iu,u}$.

Denote by $\pi_1$ the restriction of $\pi$ to $SL_i \times_{B_i^-}\mathbb CP^1_{s_iu,u}$.
Then $\varphi^{-1}(u)$    is given by
\begin{eqnarray*}
 \left\{[z\dot{s}_i, s_i u] \mid z \in U_i^-\right\}  & \text{ on } \pi_1^{-1}(U_i^-\dot{s}_iB^-_i/B^-_i),  \\
 \left\{[y, y^{-1} u] \mid y \in U_i \right\} &\text{ on } \pi_1^{-1}(U_i B^-_i/B^-_i),
\end{eqnarray*}
 and $\varphi^{-1}(s_iu)$  is given by
 \begin{eqnarray*}
  \left\{[z\dot{s}_i, (zs_i)^{-1} s_iu] \mid z \in U_i^-\right\} & \text{ on } \pi_1^{-1}(U_i^-\dot{s}_iB^-_i/B^-_i),  \\
  \left\{[y, s_i u] \mid y \in U_i \right\} & \text{ on } \pi_1^{-1}(U_i B^-_i/B^-_i).
\end{eqnarray*}
(See Figure~\ref{fig_exceptional curves}. The   curve colored by blue is $\varphi^{-1}(u)$ and the   curve colored by green is $\varphi^{-1}(s_iu)$.)

\medskip
\vskip 3.5 cm
 \begin{figure}  [!h]

\raggedright
 \setlength{\unitlength}{1cm}
 \begin{picture}  (1,1)

 \linethickness{0.15mm}
 \put(1,4){\line(1,0){5}}   \put(14,4){\line(-1,0){5}}
 \linethickness{0.15mm}
 \put(1,4){\line(0,-1){4}}  \put(14,4){\line(0,-1){4}}
 \linethickness{0.7mm}
 \textcolor{blue}{\put(1,0){\line(1,0){5}}} \textcolor{yellow!20!green}{  \put ( 13.9,0 ){\line(-1,0){5}}}

\textcolor{yellow!20!green}{ \qbezier(0.75,4)(3,2) (5.5,1)} \textcolor{blue}{ \qbezier (9,1)(12,2)(13.6,4)}



\linethickness{0.15mm}
\put(3.5, -0.8){\vector(0,-1){1 }}   \put(11.5, -0.8){\vector(0,-1){1 }}
\put(2.8, -1.3){$\pi$}                  \put(10.8, -1.3){$\pi$}

 \put(1,-2){\line(1,0){5}}  \put ( 13.8,-2 ){\line(-1,0){5}}
 \put(0.8,-2.1){$\bullet$} \put(6,-2.1){$\circ$} \put(8.6,-2.1){$\circ$} \put(13.8,-2.1){$\bullet$}

\put(-0.7 , 4){$[\dot{s}_i, u]$}   \put(14.1, 4){$[\dot{e},u]$}
\put(-0.7 ,0){$[\dot{s}_i,s_iu]$} \put(14.1,0){$[\dot{e},s_iu]$}

\put(0.7, -2.5){$\dot{s}_iB^-_i/B^-_i $} \put(6, -2.5){$B^-_i/B^-_i$} \put(8.6, -2.5){$ \dot{s}_iB^-_i/B^-_i$} \put (13.8, -2.5){$B^-_i/B^-_i$}

\put(3, -3 ){$U_i^-\dot{s}_iB_i^-/B_i^- $}                  \put(11, -3 ){$U_i B_i^-/B_i^-$}
\end{picture}

\vskip 3 cm

	 \caption{$\varphi^{-1}(u)$ and $\varphi^{-1}(s_iu)$ in $SL_i \times_{B_i^-} \mathbb CP^1_{s_iu, u}$}\label{fig_exceptional curves}
\end{figure}
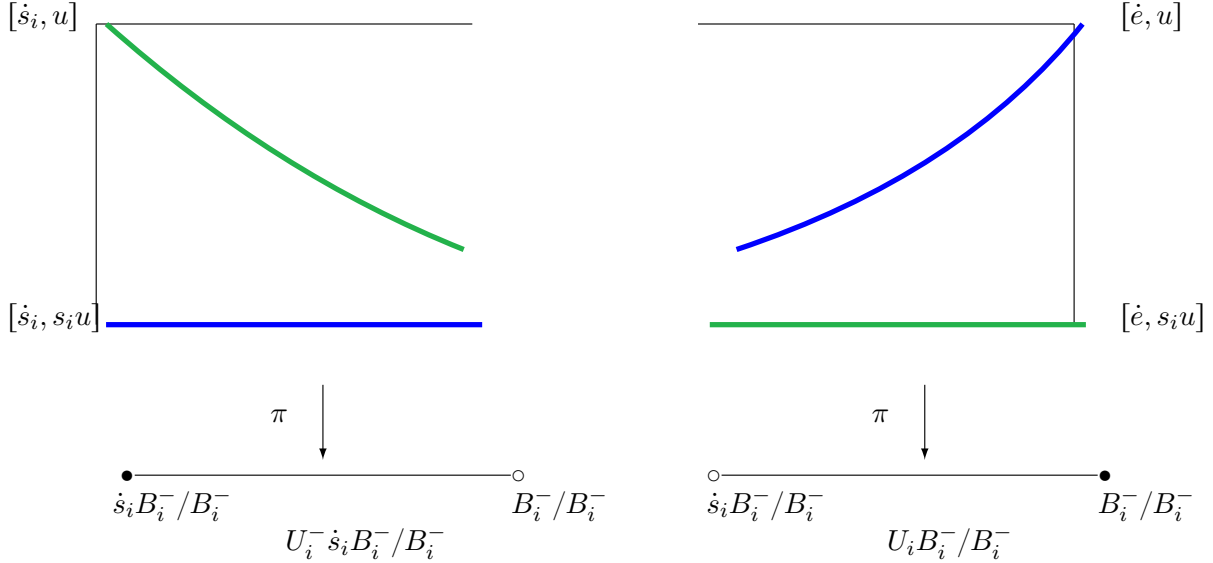

 \vskip 1 cm

Note that we have
$$\lambda(t)z \lambda(t)^{-1}=\left(
\begin{array}{cc}
1&0 \\
t^{-(\lambda_i - \lambda_{i+1})}{\bf z} &1
\end{array}\right), \quad
\lambda(t)y \lambda(t)^{-1} =\left(
\begin{array}{cc}
1&t^{ (\lambda_i - \lambda_{i+1})}{\bf y} \\
0 &1
\end{array}\right)
$$
and
$$
\dot{s}_i \lambda(t) \dot{s}_i =\left(
\begin{array}{cc}
t^{\lambda_{i+1}}&0 \\
0 &t^{\lambda_i}
\end{array}\right).
$$
From these computations, we get the following property.

\begin{lemma} \label{lem:T action}
 Assume $u >s_iu $ and $s_iu  \in \Omega_w^T$ and let $\mathbb CP^1_{s_iu,u}$ denote the projective line in $\Omega_w$  connecting $s_iu$ and $u$. Then the following statements hold.

\begin{enumerate}
\item For   a general point $[zs_i, yu]$ of $SL_i \times _{B_i^-} \mathbb CP^1_{s_iu,u}$,  we get
     $$\lim_{t \rightarrow \infty}\lambda(t)[z\dot{s}_i, yu]=[\dot{s}_i,u].$$
\item For a general point $[z \dot{s}_i, s_iu] $ of $\varphi^{-1}(u)$, we get  $$\lim_{t \rightarrow \infty}\lambda(t) [z \dot{s}_i, s_iu] =[\dot{s}_i, s_iu].$$
\end{enumerate}
\end{lemma}

\begin{proof}
 For $t \in \mathbb C^*$,  $\lambda(t)$ acts on $SL_i \times_{B_i^-} \mathbb CP^1_{s_iu,u}$ as follows. For   a general point $[zs_i, yu]$ of $SL_i \times _{B_i^-} \mathbb CP^1_{s_iu,u}$,
$$\lambda(t)[z\dot{s}_i, yu] = [(\lambda(t)z\lambda(t)^{-1})\dot{s}_i, (\dot{s}_1\lambda(t)\dot{s}_i)y  u].$$
Thus, as $t$ goes to the infinity, $\lambda(t)[z\dot{s}_i, yu]$ converges to $[\dot{s}_i,u]$. Similarly, for a general point $[z \dot{s}_i, s_iu] \in \varphi^{-1}(u)$, we get
\[
\lim_{t \rightarrow \infty}\lambda(t) [z \dot{s}_i, s_iu] =[\dot{s}_i, s_iu]. \qedhere
\]
\end{proof}

\subsubsection{ } \label{sect: Proof of Proposition 4.7}
Now let $h:[n] \rightarrow [n]$ be a Hessenberg function. In this subsection we assume that  $w \rightarrow s_iw$, that is, the edge $(w \rightarrow s_iw)$ is contained in the edge set of the  GKM graph of $\Hess(S,h)$.  Define  $Z'$ by  the irreducible component of $\varphi^{-1}(\Omega_{s_iw,h} )$ containing $\varphi^{-1}(\Omega_{s_iw,h}^{\circ})$.     Denote by $\varphi'$ and $\pi'$   the restrictions of $\varphi$  and $\pi$ to $Z'$:
\begin{eqnarray*}
\xymatrix{
Z'  \ar[d]^{\pi'} \ar[r]^{ \varphi'}  &  \Omega_{s_i w,h}  \\
SL_{i} /B_i^- &  }.
\end{eqnarray*}

Define $\mathcal A'_{0}$ and $\mathcal A'_{\infty}$ by
\begin{eqnarray*}
\mathcal A'_0 &\coloneq& \{u \in \Omega_{w}^T \mid [\dot{s}_i,u]  \in Z' \},  \\
\mathcal A'_{\infty} &\coloneq & \{u \in \Omega_{w}^T \mid [\dot{e},u]  \in Z' \}.
\end{eqnarray*}
    Then $\{[\dot{s}_i,u] \mid u \in \mathcal A_0'\} $ ($\{[\dot{e},u] \mid u \in \mathcal A_{\infty}'\}$, $\{[\dot{s}_i,u] \mid u \in \mathcal A_0'\}  \cup \{[\dot{e},u] \mid u \in \mathcal A_{\infty}'\}$,  respectively) is the $T$-fixed point set of ${\pi'}^{-1}(0)$ (${\pi'}^{-1}(\infty)$, $Z'$ respectively).

Define
\begin{eqnarray*}
{\Omega'}^{\circ}_{[\dot{s}_i, u ]} &\coloneq& \{ [g,\xi] \in Z' \mid \lim_{t \rightarrow \infty} \lambda(t) [g,\xi] = [\dot{s}_i, u ]\} \text { for } u \in \mathcal A_0'  \text{ and }\\
  {\Omega'}^{\circ}_{[\dot{e}, u ]}&\coloneq& \{ [g,\xi] \in Z' \mid \lim_{t \rightarrow \infty} \lambda(t) [g,\xi] = [\dot{e} , u ]\} \text{ for } u \in \mathcal A_{\infty}'.
  \end{eqnarray*}
  Then  the Bia{\l}ynicki-Birula decomposition of $Z'$ is given by
 $$Z' = \left(\bigsqcup_{u \in \mathcal A_0'} {\Omega'}^{\circ}_{[\dot{s}_i, u]} \right) \sqcup \left(\bigsqcup_{u \in \mathcal A_{\infty}'} {\Omega'}^{\circ}_{[\dot{e} , u]} \right).$$
 Since $Z'$ is irreducible,   there is a unique cell of  maximal dimension, $ {\Omega'}^{\circ}_{[\dot{s}_i, w]}$,   and other cells have dimension less than the dimension of  $ {\Omega'}^{\circ}_{[\dot{s}_i, w]}$.
  Denote by ${\Omega'} _{[\dot{s}_i, u ]}$ (${\Omega'} _{[\dot{e}, u ]}$, respectively) the closure of ${\Omega'}^{\circ}_{[\dot{s}_i, u ]}$ (${\Omega'}^{\circ}_{[\dot{e}, u ]}$, respectively) in $Z'$.

Let $D$ be an irreducible component of ${\pi'}^{-1}(0)$. Then $D$ is the closure of ${\Omega'}^{\circ}_{[\dot{s}_i,u]} \cap {\pi '}^{-1}(0)$ for some $ u \in \mathcal A'_{0}$. In this case, we say that the irreducible component $D$ is {\it centered at} $[\dot{s}_i,u]$.
Similarly, an irreducible component of ${\pi'}^{-1}(\infty)$ is said to be {\it centered at} $[\dot{e},u]$ if it is the closure of ${\Omega'}^{\circ}_{[\dot{e}, u]} \cap {\pi'}^{-1}(\infty)$ for some $u \in \mathcal A_{\infty}'$.

\begin{lemma} \label{lem:symmetric}
The sets $\mathcal A_0'$ and $\mathcal A_{\infty}'$ are identical.
Furthermore,  there is a bijection   between the set of irreducible components  of ${\pi'}^{-1}(0)$ and  the set of irreducible components of ${\pi'}^{-1}(\infty)$: If one is centered at $[\dot{s}_i,u]$ of dimension $d$, then the corresponding one is centered at $[\dot{e},u]$ of  dimension $d$, and vice  versa.
\end{lemma}

\begin{proof}
  Recall that the variables ${\bf z}$
   and ${\bf y}$ are defined by
 $z = \left(
 \begin{array}{cc}
 1&0 \\
 {\bf z} &1
 \end{array} \right) \in U_i^-
  \text{ and }
  y = \left(
  \begin{array}{cc}
  1&{\bf y}\\
  0 &1
 \end{array} \right) \in U_i.$
  The first gives a coordinate on $U_{i}^-\dot{s}_iB^-_i/B^-_i\simeq \mathbb C$ and the second gives a coordinate on $U_iB^-_i/B^-_i \simeq \mathbb C$. Identifying ${\bf z}$ with $1/{\bf y}$, we get $SL_i/B_i^- \simeq \mathbb CP^1$.

  From $w > s_iw$,  it follows that we have $j\colonequals w^{-1}(i+1) < k\colonequals w^{-1}(i)$.
Note that under the isomorphism $(B_i^-s_iB_i^-) \times _{B_i^-} \Omega^{\circ}_w \simeq \Omega_{s_iw}^{\circ}$, the projection $\pi:(B_i^-s_iB_i^-) \times _{B_i^-} \Omega^{\circ}_w  \rightarrow B_i^-s_iB_i^-/B_i^- \simeq \mathbb C$ induces a projection $$\Omega_{s_iw}^{\circ} \rightarrow \mathbb C ,$$ which we denote by the same notation $\pi$.
Then $\pi(z\dot{s}_i\xi) ={\bf z}$ for $z= \left(
\begin{array}{cc}
1&0 \\
{\bf z} &1
\end{array} \right)  \in U_i^-$ and $\xi \in \Omega_w^{\circ}$, and  $\pi(s_iw x) = x_{k,j}$ for $x=(x_{a,b}) \in (s_iw)^{-1}\Omega_{s_iw}^{\circ} \subset U^-$. The variable $x_{k,j}$ will be identified with the variable ${\bf z} \in \mathbb C$.

Let
 $${Z'}^{\circ}_{\infty} \coloneq \{[z \dot {s_i}, \xi] \in Z\mid z \in U_i^-,\quad  \xi \in \Omega_w^{\circ},\quad \text{ and } z\dot{s_i}\xi \in \Hess(S,h) \}$$   and
 $${Z'}^{\circ}_0 \coloneq \{[y, \zeta] \in Z \mid y \in U_i,\quad  \zeta \in \Omega_w^{\circ},\quad  \text{ and }  y\zeta \in \Hess(S,h) \}.  $$
 Then ${\pi'}^{-1}(\infty)$ is the intersection of the closure of ${Z'}^{\circ}_{\infty}$ in $Z$  with $\pi^{-1}(\infty)$, and
 ${\pi'}^{-1}(0)$ is the intersection of the closure of ${Z'}^{\circ}_{ 0}$ in $Z$  with $\pi^{-1}(0)$.
In particular, for each fixed point $[\dot{e} ,u]$ in ${\pi'}^{-1}(\infty)$, there is a curve $\gamma({\bf z})$ in ${Z'}^{\circ}_{\infty}$ with $\pi(\gamma({\bf z}))={\bf z}$, where ${\bf z} \in \mathbb C$, such that $\lim_{{\bf z} \rightarrow \infty} \gamma({\bf z}) =[\dot{e} ,u]$,  and for each fixed point $[\dot{s}_i ,u]$ in ${\pi'}^{-1}(0)$, there is a curve $\gamma({\bf y})$ in ${Z'}^{\circ}_{0}$ with $\pi(\gamma({\bf y}))={\bf y}$, where ${\bf y} \in \mathbb C$, such that $\lim_{{\bf y} \rightarrow \infty} \gamma({\bf y}) =[\dot{s}_i ,u]$.

Furthermore, ${Z'}^{\circ}_{ 0}$  is equal to
 $$   \{[\dot{s_i}z \dot {s_i}, \zeta] \in Z \mid z \in U_i^- \text { and } \zeta \in \Omega_w^{\circ} \text{ and }  z\dot{s_i}\zeta \in  s_i \Hess (S,h)\} $$
 because of the equality $\dot{s}_iU_i^-\dot{s}_i =U_i$.
 Thus $s_i {Z'}^{\circ}_{\infty}$ and ${Z'}^{\circ}_0 $ differ only by the last condition defining these sets, one is $z\dot{s_i}\xi \in \Hess(S,h)$ and the other is $z\dot{s_i}\zeta \in s_i\Hess(S,h)$.

By the same argument as in the proof of Proposition~\ref{prop: the induced action on chow ring},
 the $T$-fixed point set of  ${\pi'}^{-1}(0)$ is just the translate of the $T$-fixed point set of ${\pi'}^{-1}(\infty)$ by $s_i$, and there is a bijection between the set of irreducible components  of ${\pi'}^{-1}(0)$ and  the set of irreducible components of ${\pi'}^{-1}(\infty)$ that preserves the desired properties.
\end{proof}

Obviously, $D_0\colonequals \{[\dot{s}_i, \xi] \mid \xi \in \overline{\Omega_{w}^{\circ} \cap s_i^{-1}\Hess(S,h)} \}$ is an irreducible component of ${\pi'}^{-1}(0)$ of maximum  dimension, and $D_{\infty}\colonequals \{[\dot{e} , \xi] \mid \xi \in \overline{\Omega_{w}^{\circ} \cap  \Hess(S,h)} \}$ is an irreducible component of ${\pi'}^{-1}(\infty)$ of maximum dimension. We will describe other irreducible components ${\pi'}^{-1}(0)$ or   ${\pi'}^{-1}(\infty)$ of maximum  dimension in Lemma~\ref{lem: irreducible component of maximal dimension II}.

\begin{lemma} \label{lem: irreducible component of maximal dimension}
Any irreducible component of ${\pi'}^{-1}(0)$ of maximum dimension other than $D_0$ is $\Omega'_{[\dot{s}_i, u]}$ for some $u  \not = w \in \mathcal A_0'$.
\end{lemma}

\begin{proof}
Let $D$ be an irreducible component of ${\pi'}^{-1}(0)$ of maximum dimension and assume that it is  centered at $[\dot{s}_i,u]$. Then $D$ is the closure of ${\Omega'}^{\circ}_{[\dot{s}_i,u]} \cap {\pi'}^{-1}(0)$. If $D$ is not equal to $D_0$, then by the irreducibility of $Z'$, ${\Omega'}^{\circ}_{[\dot{s}_i,u]} \cap {\pi'}^{-1}(0)$ and ${\Omega'}^{\circ}_{[\dot{s}_i,u]} $ have the same dimension and thus they have the same closure because ${\Omega'}^{\circ}_{[\dot{s}_i,u]} $ is an affine cell.
\end{proof}

\begin{lemma} \label{lem: exceptional curve is contained in Z'}
Assume that $u >s_iu >w$.
\begin{enumerate}
\item
If $[\dot{s}_i, s_iu]   \in Z'$, then $\varphi^{-1}(u)$ is contained in $Z'$.
\item There is no irreducible component of ${\pi'}^{-1}(0)$ of maximum dimension centered at $[\dot{s}_i,s_i u]$.
    \end{enumerate}
\end{lemma}

\begin{proof} Assume that $u >s_iu >w$. The proof of  (2) will come first, followed by the proof of  (1).

(2) Let $D$ be an irreducible component of ${\pi'}^{-1}(0)$ of maximum dimension centered at $[\dot{s}_i,s_iu]$. Then by Lemma~\ref{lem: irreducible component of maximal dimension},     $D$ is equal to ${\Omega'} _{[\dot{s}_i,s_iu]} $,    %
and
by (1), $\varphi^{-1}(u)$ is contained in $Z'$.  By Lemma~\ref{lem:T action}(2), for any generic point $[g,\xi]$    of $\varphi^{-1}(u)$, $\lambda(t)[g,\xi]$ converges to $[\dot{s}_i,s_iu]$ as $t$ goes to $\infty$. By  the definition of ${\Omega'}^{\circ}_{[\dot{s}_i, s_iu]}$,  a generic point of $\varphi^{-1}(u)$ is contained in  ${\Omega'}^{\circ}_{[\dot{s}_i, s_iu]}$. Since $\varphi^{-1}(u)$ is not contained in ${\pi'}^{-1}(0)$,  $\Omega'_{[\dot{s}_i, s_iu]}$ is not contained in ${\pi'}^{-1}(0)$, contradicting to the fact that $D=\Omega'_{[\dot{s}_i, s_iu]}$ is  contained in ${\pi'}^{-1}(0)$. Therefore, there is no irreducible component of ${\pi'}^{-1}(0)$ of maximum dimension centered at $[\dot{s}_i,s_iu]$.

(1) If $[\dot{s}_i, s_iu]   \in Z'$, then there is a curve
$$ \beta: {\bf z} \in V  \mapsto  [z\dot{s}_i, \gamma({\bf z})]$$
in $\varphi^{-1}(\Omega_{s_iw,h}^{\circ})$ such that $\lim_{{\bf z} \rightarrow 0}\gamma({\bf z}) =s_iu$, where $V$ is an open neighborhood of $0 \in \mathbb C$. Here, we use the coordinate ${\bf z}$ for $z =\left( \begin{array}{cc}
1 & 0 \\
{\bf z} &1
\end{array}
\right)$ in $U_i^-$.  Then $\gamma({\bf z}) \in \Omega_{w}^{\circ}$ and $z\dot{s}_i\gamma({\bf z}) \in \Omega_{s_iw,h}^{\circ}$ for any $ {\bf z} \in V$, and $\lim_{{\bf z} \rightarrow 0} z\dot{s}_i\gamma({\bf z}) =u$.

For $z_1 \in U_i^-$,   the curve $\beta_1$   obtained by multiplying $z_1$ to $\beta$,
$$\beta_1:   {\bf z}  \in V \mapsto [z_1z\dot{s}_i, \gamma({\bf z} )],  $$
   is contained in $\varphi^{-1}(\Omega_{s_iw}^{\circ})$ but is not necessarily contained in $\varphi^{-1}(\Omega_{s_iw, h}^{\circ})$, in other words,
   $z\dot{s}_i\gamma({\bf z})$ is contained in $\Omega_{s_iw}^{\circ} \cap \Hess(S,h)$ but $z_1z\dot{s}_i \gamma({\bf z})$ is not necessarily contained in $\Omega_{s_iw}^{\circ} \cap \Hess(S,h)$. %
 We will modify $\beta_1$ to get a curve contained in $\varphi^{-1}(\Omega_{s_iw, h}^{\circ})$.

Fix $z_1 \in U_i^-$.
Since $z\dot{s}_i\gamma({\bf z})$ converges to $u$ as ${\bf z}$ goes to zero, so does $z_1z\dot{s}_i\gamma({\bf z})$ because $U^-_i$ fixes $u$.
Thus the distance ${\rm dist}\left(z\dot{s}_i\gamma({\bf z}),z_1z\dot{s}_i\gamma({\bf z}) \right)$  between $z\dot{s}_i\gamma({\bf z})$ and $z_1z\dot{s}_i\gamma({\bf z})$ converges to zero as ${\bf z}$ goes to zero.
 Here, ${\rm dist}(\eta_1, \eta_2)$ is the distance between two points $\eta_1$ and $\eta_2$ in $\Omega_{s_iw}^{\circ} \subset G/B$.

Recall that $(x_{{\alpha},{\beta}}) \in U^-$ where ${\alpha}>{\beta}$ and $s_iw({\alpha}) > s_iw({\beta})$ gives a coordinate system on $\Omega_{s_iw}^{\circ}$, and the intersection
$\Omega_{s_iw,h}^{\circ} = \Omega_{s_iw}^{\circ} \cap \Hess(S,h)$ can be defined as the graph
$$\{(x_{c,d}, x_{a,b})\mid x_{a,b}=g_{a,b}(x_{c,d}) \}$$
of some function $g_{a,b}(x_{c,d})$. Here, the index $(a,b)$ of the variable $x_{a,b}$ is an element of $\{(a,b) \mid a >h(b)\text{ and } s_iw(a) >s_w(b) \}$ and the index $(c,d)$ of the variable $x_{c,d}$ belongs to $\{(c,d) \mid h(d) \geq c >d \text{ and } s_iw(a) >s_iw(b) \}$. In this proof, we assume that the  indices $(a,b)$ and $(c,d)$  satisfy these conditions.

The $x_{c,d}$-coordinate plane in  $U^-$ having the constant  $x_{a,b}$-coordinate, say $x_{a,b}^{(0)}$, defines a submanifold of  $\Omega_{s_iw}^{\circ}$, which we denote by $\mathcal U_{  x_{a,b}^{(0)}}$ and    call the $x_{c,d}$-coordinate plane  in $\Omega_{s_iw}^{\circ}$ with the $x_{a,b}^{(0)}$-coordinate.
Then we may express $\Omega_{s_iw,h}^{\circ} $ as the image of
  a (unique) injective map $F_{  x_{a,b}^{(0)}}:\mathcal U_{  x_{a,b}^{(0)}} \rightarrow  \Omega_{s_iw }^{\circ}$ such that $\eta$ and $F_{x_{a,b}^{(0)}}(\eta)$ have the same $x_{c,d}$-coordinate for any $\eta \in \mathcal U_{  x_{a,b}^{(0)}}$. In particular, if $\eta$ is contained in $z\dot{s}_i\Omega_{w}^{\circ}$, then so is  $F_{x_{a,b}^{(0)}}(\eta)$.

Let $\eta ^{(0)}= z^{(0)}\dot{s}_i\xi^{(0)}$ be an element of $\Omega_{s_iw,h}$, where $z^{(0)} \in U_i^{- }$ and $\xi^{(0)} \in \Omega_{w}^{\circ}$.   Let $ x_{a,b}^{(0)} $ denote the $x_{a,b}$-coordinate of $\eta$.
 Then for any $\varepsilon >0$, there is $\delta^{(0)}>0$ such that for any $\eta $ in the  $x_{c,d}$-coordinate plane $\mathcal U_{x_{a,b}^{(0)}}$,  if ${\rm dist}\left(\eta, \eta^{(0)} \right)<\delta^{(0)}$, then ${\rm dist} \left( F^{(0)}(\eta) , \eta^{(0)}  \right)< \varepsilon  $, where $F^{(0)}$ is the map $F_{  x_{a,b}^{(0)}}:\mathcal U_{  x_{a,b}^{(0)} } \rightarrow  \Omega_{s_iw }^{\circ}$.

 Let $\overline{\beta}$ denote the image $\varphi \circ \beta$  of $\beta$ by $\varphi$.
By taking a smaller neighborhood $V_1$ of $0 \in \mathbb C$ whose closure is compact and is contained in $V$ if necessary, for any $\varepsilon >0$, we get $\delta>0$    such that for any ${\bf z} \in V_1$ the following holds: For any $\eta $ in the  same $x_{c,d}$-coordinate plane as $\overline{\beta}({\bf z})$, if ${\rm dist}\left(\eta, \overline{\beta}({\bf z}) \right)<\delta $, then ${\rm dist}\left( F^{\bf z}(\eta) , \overline{\beta}({\bf z}) \right)<\varepsilon  $, where $x_{a,b}^{\bf z}$ is the $x_{a,b}$-coordinate of $\overline{\beta}({\bf z})$ and  $F^{\bf z}$ is the map $F_{  x_{a,b}^{\bf z}  }:\mathcal U_{  x_{a,b}^{\bf z}  } \rightarrow  \Omega_{s_iw }^{\circ}$.

 Note that $z_1\overline{\beta}({\bf z}) = z_1z\dot{s}_i\gamma({\bf z})$ belongs to  the same  $x_{c,d}$-coordinate plane $\mathcal U_{x_{a,b}^{\bf z}}$  as $\overline{\beta}({\bf z}) =  z\dot{s}_i\gamma({\bf z})$.
 Applying  $F^{\bf z}$  to $z_1\overline{\beta}({\bf z})$ we get an element $F^{\bf z}(z_1\overline{\beta}({\bf z}))$ in $\Omega_{s_iw,h}^{\circ}$.
 By construction, $F^{\bf z}(z_1\overline{\beta}({\bf z}))$ is an element of $z_1z \dot{s}_i \Omega_{w}^{\circ}$, and thus there is $\gamma^{z_1}({\bf z}) \in \Omega_{w}^{\circ}$ such that $F^{\bf z}(z_1\overline{\beta}({\bf z}))=z_1z \dot{s}_i \gamma^{z_1}({\bf z})$.
  Take a neighborhood $V_2$ of $0 \in \mathbb C$ contained in $V_1$ such that
 for any ${\bf z} \in V_2$, ${\rm dist}\left( z_1 \overline{\beta}({\bf z}) ,\overline{\beta}({\bf z}) \right) <\delta$. Then we obtain
 ${\rm dist}\left( F^{\bf z}( z_1 \overline{\beta}({\bf z})) , \overline{\beta}({\bf z}) \right)<\varepsilon  $.

 Therefore,
$$\beta^{z_1}:  {\bf z}   \in V_1 \mapsto [z_1z\dot{s}_1, \gamma^{z_1}({\bf z})]$$ is a curve in $\varphi^{-1}(\Omega_{s_iw,h}^{\circ})$ with $\lim_{{\bf z}\rightarrow 0}\beta^{z_1}({\bf z}) =[z_1\dot{s}_i, s_i u]$, which implies that  $\varphi^{-1}(u) \backslash \{[\dot{e}, u]\}$ is contained in the closure $Z'$ of $\varphi^{-1}(\Omega_{s_iw,h}^{\circ})$. Since $\varphi^{-1}(u)$ is a $T$-invariant curve and $Z'$ is $T$-invariant and closed, it follows that $\varphi^{-1}(u)$ is contained in $Z'$.
\end{proof}

\begin{lemma} \label{lem: irreducible component of maximal dimension II}
The following statements hold.
\begin{enumerate}
\item Any irreducible component of maximum  dimension of ${\pi'}^{-1}(0)$ other than $D_0$ is of the form $ \{[\dot{s}_i,\xi] \mid \xi \in \overline{\Omega_{u}^{\circ} \cap s_i^{-1} \Omega_{s_iw,h}} \}$ for some $u \in \mathcal A_0'$ with $u \dasharrow s_iu$.
\item Any irreducible component of maximum  dimension of ${\pi'}^{-1}(\infty)$ other than $D_{\infty}$ is of the form $ \{[\dot{e} ,\xi] \mid \xi \in \overline{\Omega_{u}^{\circ} \cap   \Omega_{s_iw,h}} \}$ for some $u \in \mathcal A_{\infty}'$ with $u \dasharrow s_iu$.

\end{enumerate}

  \end{lemma}

  \begin{proof}
 (1)  Let $D$ be an irreducible component of ${\pi'}^{-1}(0)$  of maximum  dimension other than $D_0$. Then by   Lemma~\ref{lem: irreducible component of maximal dimension}, $D$  is     ${\Omega'}^{\circ}_{[\dot{s}_i, u]}$ for some $u \not=w \in \mathcal A_0'$.
 Furthermore, by Lemma~\ref{lem: exceptional curve is contained in Z'}(2),    we have $u > s_iu$. Indeed, if $s_iu >u = s_i(s_iu)$, then   Lemma~\ref{lem: exceptional curve is contained in Z'}(2) implies that there is no irreducible component of ${\pi'}^{-1}(0)$ of maximal dimension centered at $[\dot{s}_i, u]$, contradicting to the fact that $D$ is such one. If $u \rightarrow s_iu$, then ${\Omega'}^{\circ}_{[\dot{s}_i, u]}$ is not contained in ${\pi'}^{-1}(0)$, contradicting to the fact that $D$ is an irreducible component of ${\pi'}^{-1}(0)$.    Thus we have    $u \dasharrow s_iu$ and  $\Omega'_{[\dot{s}_i,u]} $ is of the form $\{[\dot{s}_i,\xi] \mid \xi \in \overline{\Omega_{u}^{\circ} \cap s_i^{-1} \Omega_{s_iw,h}} \}$.

\smallskip

(2)
Any irreducible component of ${\pi'}^{-1}(\infty)$ is of the form   ${\Omega'} _{[\dot{e},u]}$ for some $u \in \mathcal A'_{\infty}$. Let ${\Omega'} _{[\dot{e},u]}$ be an irreducible component of maximum  dimension of ${\pi'}^{-1}(\infty)$ other than $D_{\infty}$.  By Lemma~\ref{lem:symmetric},   there is an irreducible component of ${\pi'}^{-1}(0)$ of maximum dimension centered at $[\dot{s}_i,u]$   with $u \not=w$, and thus, we get $u \dasharrow s_iu$ by (1).
  \end{proof}

\subsection{Proof of Proposition~\ref{prop: edge remaining 2}} \label{sect: proof of si action}
In this subsection, we provide a proof of Proposition~\ref{prop: edge remaining 2}.
\begin{proof}[Proof of Proposition~\ref{prop: edge remaining 2}]
For the first statement, assume   that $s_{i}w \rightarrow  w$, that is, the edge $(s_{i}w \rightarrow w)$ is contained in the edge set of the GKM graph of $\Hess(S,h)$. Then
$\ell_h(w) = \ell_h(s_{i}w)-1$ by Lemma~\ref{lem:si and Gwh}(2), and thus $ \dim \Omega_{s_{i}w,h} = \dim \Omega_{w,h} -1$. We will use the same notations as in Subsection~\ref{sect all about Z}.
By Lemma~\ref{Z is trivial and varphi is a projection},
$Z$ is isomorphic to the product $(SL_i/B_i^-) \times \Omega_w$ and $\varphi$ is the projection to the second factor.

As in Subsection~\ref{sect: Proof of Proposition 4.7}, let $Z'$ be the irreducible component of $\varphi^{-1}( \Omega_{  w,h} )$ containing $\varphi^{-1}(\Omega_{w,h}^{\circ})$  and let $\varphi'$ and $\pi'$ denote the restriction of $\varphi$  and $\pi$ to $Z'$:
\begin{eqnarray*}
\xymatrix{
Z'  \ar[d]^{\pi'} \ar[r]^{ \varphi'}  & \Omega_{ w,h} \\
SL_{i} /B_i^- &  } 
\end{eqnarray*}
Then $Z'$ is isomorphic to the product $(SL_i/B_i^-) \times \Omega_{w,h}$ and $\varphi'$ is the projection to the second factor.
 Applying Theorem~\ref{thm:generators and relations} to $Z'$ and $\varphi'$, we have
$$ \varphi'_*[\pi'^{-1}(0)] -\varphi'_*[\pi'^{-1}(\infty)] = 0.$$
The right hand side is zero because of  $\dim Z' = \dim \Omega_{w,h}+1$.

Now $D_0=\{[\dot{s_i},x]\mid x \in \overline{\Omega_{w}^{\circ} \cap s_i^{-1}\Hess(S,h)} \}$ is a unique irreducible component of ${\pi'}^{-1}(0)$ of maximum  dimension, and    $D_{\infty}=\{[ \dot{e} ,x]\mid x \in \overline{\Omega_{w}^{\circ} \cap \Hess(S,h)} \}$ is a unique irreducible component of ${\pi'}^{-1}(\infty)$ of maximum dimension. Thus
 the  divisor class   $[\mathrm{div}_{Z'}(\pi')] =[\pi'^{-1}(0)]-[\pi'^{-1}(\infty)]$  is given by $[D_{0}] - [D_{\infty}]$.
By Proposition~\ref{prop: the induced action on chow ring},
    $ \varphi'_*[D_0]  = [\overline{(s _{i}\Omega_{w }^{\circ} )\cap \Hess(S,h)}]  =s_i \cdot [\overline{\Omega_w ^{\circ} \cap \Hess(S,h)}]$, and we have
$$s_{i}\cdot \sigma_{w,h} - \sigma_{w,h} =0. $$
 This completes the proof of Proposition~\ref{prop: edge remaining 2}(1).

\smallskip

For the second statement, assume   that $w \rightarrow s_{i}w$, that is, the edge $(w \rightarrow s_{i}w)$ is contained in the edge set of the GKM graph of $\Hess(S,h)$. Then
$\ell_h(w) = \ell_h(s_{i}w)+1$ by Lemma~\ref{lem:si and Gwh}(2), and thus we have $ \dim \Omega_{s_{i}w,h} = \dim \Omega_{w,h} +1$.
We will use the same notations as in Subsection~\ref{sect: Proof of Proposition 4.7}.

Recall that  $Z'$ is the irreducible component of $\varphi^{-1}(\Omega_{s_iw,h})$ containing $\varphi^{-1}(\Omega_{s_iw,h}^{\circ})$, %
and  $\varphi'$  ($\pi'$, respectively) is the restriction of $\varphi$    ($\pi$, respectively) to $Z'$:
\begin{eqnarray*}
\xymatrix{
Z'   \ar[d]^{\pi'} \ar[r]^{ \varphi'}  & \Omega_{s_{i}w,h} \\
SL_{i} /B_i^- &  }. 
\end{eqnarray*}
Note that $\dim Z' = \dim \Omega_{s_iw,h}$. Applying Theorem~\ref{thm:generators and relations} to $Z'$ and $\varphi'$, we have
$$ \varphi'_*[\pi'^{-1}(0)] -\varphi'_*[\pi'^{-1}(\infty)] = (t_{i+1} -t_i)\sigma_{s_{i} w, h}.$$ It remains to describe the  divisors $ \varphi'_*[\pi'^{-1}(0)] $ and $\varphi'_*[\pi'^{-1}(\infty)] $ in $\Omega_{s_iw,h}$, that is, irreducible components of $\pi'^{-1}(0)$ of maximum  dimension, and irreducible components of $\pi'^{-1}(\infty)$ of maximum dimension, and their images under $\varphi'$.

 The fixed point set of $Z'$ is the union $\{[\dot{s}_i,u]\mid  u\in \mathcal A_{0}'\} \cup \{[\dot{e}, u]\mid u \in \mathcal A_{\infty}'\}$,
where
$$\mathcal A'_0 =\{ u \in \mathcal A_0 \mid [\dot{s}_i, u] \in Z'\} \text{ and }\mathcal A'_{\infty} =\{ u \in \mathcal A_{\infty} \mid [\dot{e},u] \in Z'\} .$$
By Lemma~\ref{lem:symmetric}, $\mathcal A_0'$ is equal to $\mathcal A_{\infty}'$. %
By Lemma~\ref{lem: irreducible component of maximal dimension II},
any irreducible component of maximum dimension of ${\pi'}^{-1}(0)$ other than $D_0 $ is of the form $\{[\dot{s}_i,\xi] \mid \xi \in \overline{\Omega_{u}^{\circ} \cap s_i^{-1} \Omega_{s_iw,h}} \}$ for some $u \in \mathcal A_0'$ with $u \dasharrow s_iu$.
    The image of $\{[\dot{s}_i,\xi] \mid \xi \in \overline{\Omega_{u}^{\circ} \cap s_i^{-1} \Omega_{s_iw,h}} \}$ by $\varphi'$ is $\overline{(s_i\Omega_u^{\circ}) \cap \Omega_{s_iw,h}} = \overline{ \Omega_{s_iu}^{\circ} \cap \Omega_{s_iw,h}}$, and  $\Omega_{s_iu}^{\circ} \cap \Omega_{s_iw,h} $ has the same dimension as $\Omega_{ u}^{\circ} \cap \Omega_{s_iw,h} $.
 By Lemma~\ref{lem: irreducible component of maximal dimension II},
  any irreducible component of maximum dimension of ${\pi'}^{-1}(\infty)$ other than $D_{\infty} $ is of the form $\{[\dot{e} ,\xi] \mid \xi \in \overline{\Omega_{u}^{\circ} \cap  \Omega_{s_iw,h}} \}$ for some $u \in \mathcal A_{\infty}'$ with $u \dasharrow s_iu$.
Thus
 the  divisor class   $[\mathrm{div}_{Z'}(\pi')] =[\pi'^{-1}(0)]-[\pi'^{-1}(\infty)]$  is given by
 $$\left( [D_{0}] + \sum_u \left[\{[\dot{s}_i,\xi] \mid \xi \in \overline{\Omega_{u}^{\circ} \cap s_i^{-1} \Omega_{s_iw,h}} \} \right]\right)  - \left( [D_{\infty}] + \sum_u \left[ \{[\dot{e} ,\xi] \mid \xi \in \overline{\Omega_{u}^{\circ} \cap  \Omega_{s_iw,h}\}} \right] \right) $$
 where the sum is over $u \in \mathcal A_0' =\mathcal A_{\infty}'$ such that $u \dasharrow s_iu$ and $\dim (\Omega_u^{\circ} \cap \Omega_{s_iw,h}) = \dim \Omega_{w,h}$.

Put
$$ \mathcal A_{s_i,w} \coloneq \{u \in \Omega_{s_{i}w,h}^T \cap \Omega_w^T  \mid  u\dasharrow s_iu   \text{  and }  \dim (\Omega_{ u}^{\circ} \cap \Omega_{s_iw,h} )   =\dim \Omega_{w,h}   \}. $$
Then the set  $\{u \in \mathcal A'_{\infty} \mid  u \dasharrow s_iu, \dim (\Omega_u^{\circ} \cap \Omega_{s_iw,h}) = \dim \Omega_{w,h}\}$ is contained in $\mathcal A_{s_i,w}$. We claim that  $\mathcal A_{s_i,w}$ is contained in $\mathcal A'_{\infty}$.
By definition, $\Omega_{s_iw,h}^T $ is the union $\mathcal A_{\infty}' \cup s_i\mathcal A_{0}'$.
Thus  we have $$\Omega_{s_iw,h}^T \cap \Omega_w^T = \mathcal A_{\infty}' \cup (s_i\mathcal A_{0}' \cap \Omega_{w}^T) .$$
If  $u\dasharrow s_iu$,       $  \dim (\Omega_{ u}^{\circ} \cap \Omega_{s_iw,h} )   =\dim \Omega_{w,h} $, and   $u \in s_i\mathcal A_{0}' \cap \Omega_{w}^T$, then we get $u >s_iu  $ and $s_iu \in \mathcal A_0'$. By   Lemma~\ref{lem: exceptional curve is contained in Z'} (1), $\varphi^{-1}(u)$ is contained in $Z'$. In particular, $u$ is contained in $\mathcal A_{\infty}'$.
As a result,  $\mathcal A_{s_i,w}$ is contained in $\mathcal A'_{\infty}$.

Consequently, $\mathcal A_{s_i,w}$ equals $\{u \in \mathcal A_{\infty}'\mid u\dasharrow s_iu   \text{  and }  \dim (\Omega_{ u}^{\circ} \cap \Omega_{s_iw,h} )   =\dim \Omega_{w,h}   \}  $, which again equals   $\{u \in \mathcal A_{0}'\mid u\dasharrow s_iu   \text{  and }  \dim (\Omega_{ u}^{\circ} \cap \Omega_{s_iw,h} )   =\dim \Omega_{w,h}   \}$.

For $u \in  \mathcal A_{s_i,w}$, define $\mathcal T_u$  and $\mathcal T_{s_iu}$ by the closure of $ \Omega_u^{\circ} \cap \Omega_{s_{i}w,h}$  and $\Omega_{s_iu}^{\circ} \cap \Omega_{s_iw,h}$,  and let $\tau_u$ and $\tau_{s_iu}$ denote the equivariant class  in $H^*_T(\Hess(S,h))$ induced by them.
Then
 we get
$$\left( s_i \cdot \sigma_{w,h} + \sum_{u \in \mathcal A_{s_i,w}  } \tau_{s_iu} \right) - \left(\sigma_{w,h} + \sum_{u   \in \mathcal A_{s_i,w} } \tau_u \right)= (t_{i+1 }- t_{i})\sigma_{s_i w, h} \text{ in } H^{\ast}_T(\Hess(S,h)). $$
The intersection $\mathcal A_{s_i,w} \cap s_i \mathcal A_{s_i,w}$ is empty by the definition.
This completes the proof of Proposition~\ref{prop: edge remaining 2}(2).
\end{proof}

\begin{remark}
From the proof of Proposition~\ref{prop: edge remaining 2}(2), we see that when $h=(n,n,\dots, n)$, $\mathcal A_{s_i,w}$ is empty for any $w$ and $s_i$,  and we get
$$ s_{i} \cdot \sigma_{w,h}  -  \sigma_{w,h}  =(t_{i+1}-t_i)\sigma_{s_iw,h},$$
which is the result of Proposition~\ref{prop:si action in full flag}(2).
\end{remark}

\section{Permutohedral varieties}\label{sec:cho}

In this section, we consider the Hessenberg variety associated with the function $h$ defined as $h(i)=i+1$ for $i=1,\dots, n-1$ and $h(n)=n$. This variety, denoted by $\mathcal H_n$, is called the \emph{permutohedral variety} and is known to be the toric variety corresponding to the fans determined by the chambers of the root system of type $A_{n-1}$. We provide  explicit descriptions of the classes $\sigma_{w, h}$ and the symmetric group actions for each  class. It is known from  \cite{SW} that $H_T^{2k}(\mathcal H_n)$ is a direct sum of permutation modules of $\frak S_n$ for each $k$. The explicit descriptions of basis classes and the symmetric group action enable us to construct permutation modules that constitute the $\frak S_n$-module $H_T^{2k}(\mathcal H_n)$ for all $k$.

In this section, we only consider the permutohedral varieties, and we drop $h$ in $A_{w, h}$, $\sigma_{w,h}$, $J_{w,h,j}$, $\Ghw$ and others for convenience.

\subsection{Bases of $H_T^{2k}(\mathcal H_n)$}
For a permutation $w\in \frak S_n$, we say that $i\in[n-1]$ is a \emph{descent} of $w$ if $w(i)>w(i+1)$. Let $D(w)$ be the set of descents of $w$, and let $\mathrm{des}(w)=|D(w)|$ be the number of descents of $w$. Then the following lemma follows from the definition of the GKM graphs.

\begin{lemma}\label{lem:edges}
	Let $w\in \frak S_n$ be a vertex in the GKM graph of $\mathcal H_n$.
	\begin{enumerate}
		\item The number of all edges adjacent to $w$ is $n-1$.
		\item The number of outgoing(downward) edges from $w$ is the number of descent of $w$; that is, $\ell_h(w)=\mathrm{des}(w)$.
	\end{enumerate}
\end{lemma}

For a set partition $I_1, \dots, I_l$ of $[n]$, the \textit{Young subgroup} $\mathfrak{S}_{I_1} \times \cdots \times \mathfrak{S}_{I_l}$ is defined to be the stabilizer of $I_1,\dots,I_l$ in the symmetric group $\mathfrak{S}_n$. For a nonempty subset $I$ of $[n]$, we denote the Young subgroup $\mathfrak{S}_I \times \mathfrak{S}_{[n]\setminus I}$ by $\frak S_I \subseteq \frak S_n$ for notational simplicity.

For two integers $a < b$, we denote by $[a,b]$ the set $\{ m \in \Z \mid a \leq m \leq b \}$. Also, for notational simplicity, for $w \in \mathfrak{S}_n$ and $1 \leq a < b \leq n$, we set
\[
w[a,b] \colonequals \{ w(m) \mid a \leq m \leq b\}.
\]
We will provide explicit descriptions of the support $A_{w}$ of $\sigma_w$ and the value $\sigma_w(v)$ for each $v\in A_{w}$.

\begin{proposition}\label{prop:Perm_A} For a permutation $w\in \frak{S}_n$, let $D(w)=\{d_1<\cdots < d_k\}\subseteq [n-1]$ be the set of descents of $w$, and let $d_0=0$, $d_{k+1}=n$. If we let $ J_s=w[d_{s-1}+1,d_s]$ for $s=1, 2, \dots, k+1$, then
\[
A_{w}=\left(\frak S_{J_1}\times \cdots \times \frak S_{J_{k+1}} \right) w=\{uw\mid u\in\frak S_{J_1}\times \cdots \times \frak S_{J_{k+1}}  \}.
\]
\end{proposition}

\begin{proof}
	Because we are considering the Hessenberg function $h(i) = i+1$ for all $i$, the graph $G_w$ is a disjoint union of path graphs and it is enough to consider pairs $(j,j+1)$ when we find the graph $G_w$. In addition, an edge $j \to (j+1)$ is in the graph $G_w$ if and only if $w(j) < w(j+1)$. Therefore, we have that $(j \to j+1)$ is an edge in $G_w$ if and only if  $j \notin D(w) = \{d_1 < \dots < d_k\}$
	and we obtain	
	\[
	G_w = \bigsqcup_{s \in \{0,1,\dots,k\}} \mathcal{P}(d_{s}+1,\dots, d_{s+1})\,,
	\]
	where $\mathcal{P}(v_1,\dots,v_m)$ denote the path graph with vertices $v_1,\dots,v_m$ and edges $(v_p,v_{p+1})$ for $p \in [m-1]$.
	For an index $j$ satisfying $d_{s} < j \leq d_{s+1}$, the set of reachable vertices from $j$ is $[d_{s}+1,d_{s+1}\}$. For any index $j$ satisfying $d_s < j \leq d_{s+1}$, we obtain
	\[
 \{w(i) \mid i \in [d_s+1,d_{s+1}] \} \cup w(j) =
w[d_{s}+1,d_{s+1}] = J_{s+1}.
	\]
	By the definition of $J_{w,j}$ and $A_{w}$, we obtain
	$A_{w} = \left(\mathfrak{S}_{J_1} \times \cdots \times \mathfrak{S}_{J_{k+1}}\right) w$.
\end{proof}

\begin{remark}
	It is well known that the GKM graph of $\mathcal H_n$ is the one skeleton of the \emph{permutohedron} of type $A_{n-1}$, that is the convex hull of the permutation vectors $\{(w(1), \dots, w(n)) \mid w\in \frak{S}_n \}$ in $\mathbb R^n$. Proposition~\ref{prop:Perm_A} can be proved using the arguments on the face structure of the permutahedron; as shown in \cite{PabiniakSabatini}.
\end{remark}

We now provide an explicit description of each class $ \sigma_w$ for $w\in  \frak{S}_n$.

\begin{theorem}\label{thm:Perm_class} For a permutation $w\in \frak{S}_n$, let $D(w)=\{d_1<\cdots < d_k\}\subseteq [n-1]$ be the set of descents of $w$ and $v\in A_{w}$. Then
\[
\sigma_w (v)=\prod_{s=1}^k (t_{v(d_s+1)}-t_{v(d_s)}).
\]
\end{theorem}
\begin{proof} The induced subgraph $\Gamma_{w}$ of the GKM graph $\Gamma$ of $\mathcal H_n$, with the vertex set $A_{w}$ is $(n-1-k)$-regular. Hence, by Proposition~\ref{cor:regular class}, $\sigma_w$ is uniquely determined by the condition that the value $\sigma_w(v)$ is the product of $\alpha(v \rightarrow u)$ for $v \rightarrow u$ is an edge in $\Gamma$ but not in $\Gamma_{w}$.
\end{proof}

\begin{example} Let $w=25\,347\,168$ be a permutation in $\frak{S}_8$. Then $D(w)=\{2, 5\}$ and
\[
A_{w}=\{v\in \frak S_8 \mid \{v(1), v(2)\}={\{2, 5\}}, \{v(3), v(4), v(5)\}={\{3, 4, 7\}}, \{v(6), v(7), v(8)\}=\{1, 6, 8\}\}
\]
has $2\times 6 \times 6=72$ elements. For example, $v=52\,437\,681$ is an element of  $A_{w}$.
	Moreover, because of Theorem~\ref{thm:Perm_class} we obtain  $\sigma_w (w)=(t_3-t_5)(t_1-t_7)$ and $\sigma_w (v)=(t_4-t_2)(t_6-t_7)$.
\end{example}

From Lemma~\ref{lem:edges} and Theorem~\ref{thm:Perm_class} we obtain a nice basis for the $(2k)${th} cohomology space of the permutohedral variety $\mathcal H_n$.

\begin{proposition}\label{prop:basis} For each $k=0, 1, \dots, n-1$, $B_k\colonequals \{\sigma_w \mid w\in \mathfrak{S}_n, \, \mathrm{des}(w)=k \}$ forms a basis of the $(2k)$th cohomology space $H^{2k}(\mathcal H_n)$; hence the $(2k)$th Betti number of $\mathcal H_n$ is the number of permutations in $\mathfrak S_n$ with $k$ descents.
\end{proposition}

\subsection{Symmetric group action on $H^*(\mathcal H_n)$}
In this subsection, we will consider the symmetric group action on the equivariant cohomology ring $H^{\ast}_T(\mathcal H_n)$, and then study that on the singular cohomology ring $H^{\ast}(\mathcal H_n)$.

We denote by
\[
B_k\colonequals\{\sigma_w \mid w\in \mathfrak{S}_n, \, \mathrm{des}(w)=k \},
\]
a basis of the cohomology space $H^{2k}_T(\mathcal H_n)$.
An $l$-\emph{composition} of $n$ is a sequence of positive integers $\mathbf a=(a_1, \dots, a_{l})$  satisfying $\sum_{i} a_i = n$. For a composition $\mathbf a=(a_1, \dots, a_{l})$ of $n$, we let $S(\mathbf a)=\{ a_1, a_1+a_2, \dots, a_1+\cdots +a_{l-1}\}$ be a subset of $[n-1]$ and let
\[
\mathfrak{S}_n(\mathbf a) \colonequals \{ w \in \mathfrak S_n \mid D(w)=S(\mathbf a)\}, \quad
B_k(\mathbf a)\colonequals\{ \sigma_w\mid  w\in\mathfrak{S}_n(\mathbf a)\}.
\]
Then we obtain
\[
B_k= \bigsqcup_{\mathbf a}\,\, B_k(\mathbf a),
\]
where the union is over the $(k+1)$-compositions of $n$.

\begin{lemma}\label{lem:disconnectedness of S_n(a)} For a $(k+1)$-composition $\mathbf a=(a_1, \dots, a_{k+1})$ of $n$, any pair of two distinct permutations in $\mathfrak{S}_n(\mathbf a)$ is not connected in the GKM graph of $\mathcal H_n$.
\end{lemma}
\begin{proof} Let $w \rightarrow w s_{i,j}$, $i<j$, be an edge in the GKM graph of $\mathcal H_n$ for $w \in \mathfrak{S}_n(\mathbf a)$, then $j$ must be $i+1$. However, $w s_i$ is not in $\mathfrak{S}_n(\mathbf a)$ since exchanging the $i$th and the $(i+1)$th element of $w$ will change the descent set; $D(w s_i)= D(w) \cup \{ i\}$ if $i \not \in  D(w)$, and $D(w s_i)= D(w) \setminus \{ i\}$ if $i \in  D(w)$.
\end{proof}

Lemma~\ref{lem:disconnectedness of S_n(a)} combined with Proposition~\ref{prop: edge deleted} proves the following proposition.

\begin{proposition}\label{prop:s_i action on S_n(a)} For a permutation $w\in \mathfrak{S}_n$ and a simple transposition $s_i\in \mathfrak{S}_n$, if $D(s_i w)=D(w)$, then $s_i \cdot \sigma_w=\sigma_{s_i w}$ as elements in $H^{\ast}_T(\mathcal H_n)$.
\end{proposition}

\begin{example} When $n=4$ and $k=1$, there are  three $(k+1)$-compositions of $n$; $\mathbf a=(1, 3), \mathbf b= (2, 2)$ and $\mathbf c=(3, 1)$.  The corresponding subsets of $\mathfrak{S}_4$ are as follows:
	\begin{align*}
	\mathfrak{S}_4(\mathbf a)&=\{4\textcolor{red}{|}123, 3\textcolor{red}{|}124, 2\textcolor{red}{|}134\},\\
	\mathfrak{S}_4(\mathbf b)&=\{ 34\textcolor{red}{|}12,  24\textcolor{red}{|}13,  23\textcolor{red}{|}14,  14\textcolor{red}{|}23, 13\textcolor{red}{|}24\},\\
	\mathfrak{S}_4(\mathbf c)&=\{234\textcolor{red}{|}1, 134\textcolor{red}{|}2, 124\textcolor{red}{|}3\}.
	\end{align*}
	We consider the permutations in $\mathfrak{S}_4(\mathbf b)$ in more detail. We set
\[
w_{34}\colonequals34\textcolor{red}{|}12,\quad w_{24}\colonequals24\textcolor{red}{|}13,\quad w_{23}\colonequals23\textcolor{red}{|}14,\quad w_{14}\colonequals14\textcolor{red}{|}23,\quad w_{13}\colonequals13\textcolor{red}{|}24.
\]
	Since  $s_2 w_{34}=w_{24}$, $s_1 w_{24}=w_{14}$,  $s_3 w_{24}=w_{23}$,  $s_3 w_{14}=w_{13}$, and $s_1 w_{23}=w_{13}$, by Proposition~\ref{prop:s_i action on S_n(a)}, we can obtain all $\sigma_w$ for $w\in \mathfrak{S}_4(\mathbf b)$ by applying sequences of $s_i$'s to $\sigma_{w_{34}}$. However, we note that $s_2 w_{13}\not\in \mathfrak{S}_4(\mathbf b)$, that is  $w_{13}$ and  $s_2 w_{13}$ have different descent sets, and we cannot apply Proposition~\ref{prop:s_i action on S_n(a)} in this case.
\end{example}

\begin{proposition}\label{prop:stabilizer of sigma_w_in_HT} For a permutation $w\in \frak{S}_n$, let $D(w)=\{d_1<\cdots < d_k\}\subseteq [n-1]$ be the set of descents of $w$, and let $d_0=0$, $d_{k+1}=n$.  The stabilizer subgroup of $\frak{S}_n$ for the class $\sigma_w \in H^{\ast}_T(\mathcal H_n)$ is  $\frak S_{J_1}\times \cdots \times \frak S_{J_{k+1}}$, where $J_s=w[d_{s-1}+1,d_s]$ for $s=1, 2, \dots, k+1$.
\end{proposition}
\begin{proof}
	By Proposition~\ref{prop:Perm_A}, the support $A_w$ of $\sigma_w$ is given by
	\[
	A_w = (\frak S_{J_1}\times \cdots \times \frak S_{J_{k+1}}) w = \{ uw \mid u \in \frak S_{J_1}\times \cdots \times \frak S_{J_{k+1}}\}.
	\]
	This implies that for an element $v \notin \frak S_{J_1}\times \cdots \times \frak S_{J_{k+1}}$, the support is not invariant, that is, $v \cdot A_w \neq A_w$. Accordingly, the stabilizer subgroup for the class $\sigma_w$ is contained in $\frak S_{J_1}\times \cdots \times \frak S_{J_{k+1}}$. By contrast, for elements $u$ and $v$ in $\frak S_{J_1}\times \cdots \times \frak S_{J_{k+1}}$, we obtain the following
	\[
	\begin{split}
	(u \cdot \sigma_w)(vw)
	= u (\sigma_w (u^{-1}vw))
	&= u \left(\prod_{s=1}^k (t_{u^{-1}vw(d_s+1)} - t_{u^{-1}vw(d_s)})\right)\\
	&= \prod_{s=1}^k (t_{vw(d_s+1)} - t_{vw(d_s)}) = \sigma_w(vw).
	\end{split}
	\]
	Here, the second and the last equalities derive from Theorem~\ref{thm:Perm_class}. This proves the proposition.
\end{proof}

We now consider the $s_i$-action on the element~$\sigma_{w}$ when $w^{-1}(i+1)+1 = w^{-1}(i)$; that is, there is a descent in between $i+1$ and $i$ in the one-line notation of $w$.
Let $D(w) = \{ d_1,\dots, d_k\}$. We denote the location by $d_{\ell}$,  where $i+1$ appears in $w$; that is $w(d_{\ell}) = i+1$.
In addition, we set $d_0 = 0$ and $d_{k+1}=n$.
Consider numbers between descents $d_{\ell-1}$ and $d_{\ell+1}$ as follows:
\[
\widetilde{P} \colonequals \{ w(j) \mid d_{\ell-1} < j < d_{\ell}\}, \quad
\widetilde{Q} \colonequals \{ w(j) \mid d_{\ell}+1 < j \leq d_{\ell+1}\}.
\]
We choose subsets
$P = \{  p_1 < p_2 < \dots < p_x  \} \subset \widetilde{P}$, and
$Q = \{ q_1 < q_2 < \dots < q_y  \} \subset \widetilde{Q}$.
We define $\tvPQi \in \mathfrak{S}_n$ to be
\begin{equation*}
\tvPQi = w(1) \ \dots  \ w(d_{\ell-1}) \ p_1 \ p_2 \ \dots \ p_x \ i+1 \ q_1 \ q_2 \ \dots q_y \
((\widetilde{P} \cup \widetilde{Q} \cup \{i\}) \setminus (P \cup Q)) \!\!\uparrow
\ w(d_{\ell+1}+1) \ \dots \ w(n).
\end{equation*}

Now we compare the descents sets $D(w)$ and $D(\tvPQi)$. First of all, because we are taking the same numbers as $w$ before $d_{\ell-1}$ and after $d_{\ell+1}+1$, we have
\[
\{d_{a} \mid a \in [k] \setminus \{ \ell-1, \ell, \ell+1 \} \} \subset D(\tvPQi).
\]
Because $P \subset [i-1]$ and $Q \subset [i+2, n]$, we obtain
\[
p_1 < p_2 < \dots < p_x < i+1 < q_1 < q_2 < \dots < q_y.
\]
Moreover, $\min ((\widetilde{P} \cup \widetilde{Q} \cup \{i\}) \setminus (P \cup Q)) \leq i$. Therefore,
there is a descent $x+y+1+d_{\ell-1}$ in~$\tvPQi$.  Depending on whether  $d_{\ell-1}$, $d_{\ell+1}$ are descents in $\tvPQi$ or not, we modify $\tvPQi$ as follows:
\begin{equation*}
\vPQi \colonequals \begin{cases}
\tvPQi & \text{ if } d_{\ell-1}, d_{\ell+1} \in D(\tvPQi) \cup \{0,n\},\\
 \tvPQi s_{d_{\ell-1}} &
\begin{array}{l}
\text{if } d_{\ell-1} \notin D(\tvPQi) \cup \{0\} \text{ and } \\
\quad d_{\ell+1} \in D(\tvPQi) \cup \{n\},
\end{array}
\\
\tvPQi s_{d_{\ell+1}} &
\begin{array}{l}
\text{if }d_{\ell-1} \in D(\tvPQi) \cup\{0\} \text{ and } \\
\quad d_{\ell+1} \notin D(\tvPQi) \cup \{n\},
\end{array} \\
 \tvPQi s_{d_{\ell-1}}  s_{d_{\ell+1}}  & \text{ if } d_{\ell-1}, d_{\ell+1} \notin D(\tvPQi) \cup \{0,n\}.
\end{cases}
\end{equation*}
Define
\begin{equation}\label{eq_def_of_a}
u_{P,Q} = \begin{cases}
e & \text{ if } d_{\ell-1}, d_{\ell+1} \in D(\tvPQi) \cup \{0,n\},\\
s_{\tvPQi(d_{\ell-1}), \tvPQi(d_{\ell-1}+1)} &
\begin{array}{l}
\text{if } d_{\ell-1} \notin D(\tvPQi) \cup \{0\} \text{ and } \\
\quad d_{\ell+1} \in D(\tvPQi) \cup \{n\},
\end{array}
\\
s_{\tvPQi(d_{\ell+1}), \tvPQi(d_{\ell+1}+1)}  &
\begin{array}{l}
\text{if }d_{\ell-1} \in D(\tvPQi) \cup\{0\} \text{ and } \\
\quad d_{\ell+1} \notin D(\tvPQi) \cup \{n\},
\end{array} \\
s_{\tvPQi(d_{\ell+1}), \tvPQi(d_{\ell+1}+1)} s_{\tvPQi(d_{\ell-1}), \tvPQi(d_{\ell-1}+1)} & \text{ if } d_{\ell-1}, d_{\ell+1} \notin D(\tvPQi) \cup \{0,n\}.
\end{cases}
\end{equation}
Then, we have
\[
\vPQi = u_{P,Q} \tvPQi
\]
Because of the construction, we obtain
\begin{equation}\label{eq_descents_of_vPQi}
D(\vPQi) = \{ d_a \mid a \in [k] \setminus \{ \ell\}\} \cup \{ x+y+1+d_{\ell-1}\}.
\end{equation}

Moreover, we define the class $\tsvPQi$ by
\[
\tsvPQi \colonequals u_{P,Q} \sigma_{\vPQi}.
\]
We note that when $P = \widetilde{P}$ and $Q = \emptyset$, we obtain  $\tvPQi= \widetilde{w}^{(i)}_{\widetilde{P},\emptyset} = w$ and  $\widetilde{\sigma}_{\widetilde{P},\emptyset}^{(i)} = \sigma_w$.
With these terminologies, we describe the $s_i$-action as follows:
\begin{proposition}\label{prop_si_action_on_sigma_i}
Let $w$ be a permutation satisfying $w^{-1}(i+1)+1 = w^{-1}(i)$.
	Let $\sigma_w^{(i)}$ be an element defined by
	\[
	\sigma_w^{(i)} \colonequals \sum_{\substack{P \subset \widetilde{P}, \\ Q \subset \widetilde{Q}}} \tsvPQi \in H^{\ast}_T(\mathcal H_n).
	\]
	Then we have
	\[
	(t_{i+1} - t_i)\sigma_{s_i  w} = s_i \cdot \sigma_w^{(i)} - \sigma_w^{(i)}.
	\]
\end{proposition}
The proof of Proposition~\ref{prop_si_action_on_sigma_i} will be given at the end of this subsection.
\begin{example}\label{example_1324}
Let $n = 4$ and $w = 1324$. Then, for subsets $P \subset \{1\}$ and $Q \subset \{4\}$, permutations~$w_{P,Q}^{(2)}$ are given as follows.
	\begin{center}
		\begin{tabular}{|c|*{2}{c|}}
			\hline
			\diagbox{$P$}{$Q$}
			&\makebox[3em]{$\{4\}$}&\makebox[3em]{$\emptyset$} \\ \hline \hline
			$\{1\}$ &$134\textcolor{red}{|}2$ & $13\textcolor{red}{|}24$  \\\hline
			$\emptyset$ & $34\textcolor{red}{|}12$ & $3\textcolor{red}{|}124$ \\\hline
		\end{tabular}
	\end{center}
Here, we put $\textcolor{red}{|}$ on descents.
By Proposition~\ref{prop_si_action_on_sigma_i}, we have
\[
\sigma_{1324}^{(2)} = \sigma_{1342}+ \sigma_{1324} + \sigma_{3412} + \sigma_{3124};
\]
\[
s_2 \sigma_{1324}
= (t_3-t_2) \sigma_{1234}
+ \sigma_{1324}
+ (\sigma_{1342} + \sigma_{3412}+ \sigma_{3124})
-(\sigma_{1243} + \sigma_{2413} + \sigma_{2134}).
\]
This computation agrees with that in Example~\ref{example_si_sigma_wh}.
\end{example}

\begin{example}\label{example_13245}
	Let $n = 5$ and consider $s_2 = 13245$. Then, for $P \subset [1] = \{1\}$ and $Q  \subset [4,5] = \{4,5\}$, permutations~$w_{P,Q}^{(2)}$ are given as follows.
	\begin{center}
		\begin{tabular}{|c|*{4}{c|}}
			\hline
			\diagbox{$P$}{$Q$}
			&\makebox[3em]{$\{4,5\}$}&\makebox[3em]{$\{4\}$}&\makebox[3em]{$\{5\}$}
			&\makebox[3em]{$\emptyset$} \\ \hline \hline
			$\{1\}$ &$1345\textcolor{red}{|}2$ & $134\textcolor{red}{|}25$& $135\textcolor{red}{|}24$ &$13\textcolor{red}{|}245$ \\\hline
			$\emptyset$ & $345\textcolor{red}{|}12$ & $34\textcolor{red}{|}125$ & $35\textcolor{red}{|}124$ & $3\textcolor{red}{|}1245$ \\\hline
		\end{tabular}
	\end{center}
	We put $\textcolor{red}{|}$ on descents. By Proposition~\ref{prop_si_action_on_sigma_i}, we have
	\[
	\sigma_{13245}^{(2)} = \sigma_{13452} + \sigma_{13425} + \sigma_{13524} + \sigma_{13245} + \sigma_{34512} + \sigma_{34125} + \sigma_{35124} + \sigma_{31245}.
	\]
	Accordingly, we obtain
	\[
	\begin{split}
	s_2 \sigma_{13245}
	& = (t_3-t_2) \sigma_{12345} + \sigma_{13245}  \\
	& \qquad + (\sigma_{13452} + \sigma_{13425} + \sigma_{13524} + \sigma_{34512} + \sigma_{34125} + \sigma_{35124} + \sigma_{31245}) \\
	& \qquad \quad - (
	\sigma_{12453} + \sigma_{12435} + \sigma_{12534} + \sigma_{24513} + \sigma_{24135}
	+ \sigma_{25134} + \sigma_{21345}
	).
	\end{split}
	\]
\end{example}
\begin{remark}
In Examples~\ref{example_1324} and~\ref{example_13245}, the permutations $w$ are simple reflections. Accordingly, $D(w) = 1$ and we obtain  $d_{\ell-1} = 0, d_{\ell+1} = n$. Therefore, we obtain $\vPQi = \tvPQi$ and this holds in all such cases.
\end{remark}
Before providing a proof of Proposition~\ref{prop_si_action_on_sigma_i}, we prepare lemmas.
\begin{lemma}\label{lemma_union_of_supports}
The union $\bigcup_{\substack{P \subset \widetilde{P}, \\ Q \subset \widetilde{Q}}} \supp(\tsvPQi)$ of the supports of $\tsvPQi$ is given as follows.
	\[
	\begin{split}
	&u \in \bigcup_{\substack{P \subset \widetilde{P}, \\ Q \subset \widetilde{Q}}} \supp(\tsvPQi) \\
	&\quad \iff
	\begin{array}{l}
	u^{-1}(i+1) < u^{-1}(i),  \\
	u[d_{a-1}+1,d_a] = w[d_{a-1}+1,d_a]\quad \text{ for all } a \in [k] \setminus \{\ell,\ell+1\}.
	\end{array}
	\end{split}
	\]
\end{lemma}
\begin{proof}
Let $P = \{p_1 < p_2< \dots < p_x \} \subset \widetilde{P}$ and $Q = \{q_1 < q_2 < \dots < q_y \} \subset \widetilde{Q}$.
We first consider the support $A_{\vPQi}$. Using the description of $D(\vPQi)$ and Proposition~\ref{prop:Perm_A}, we have
\[
\begin{split}
A_{\vPQi} &= \{ v \vPQi \mid
v \in \mathfrak{S}_{J_1} \times \cdots \times \mathfrak{S}_{J_{k+1}}
\}\\
&= \{ \vPQi v' \mid
v' \in G_{P,Q}\}.
\end{split}
\]
Here,
\[
J_s = \begin{cases}
\vPQi[d_{s-1}+1,d_s]  & \text{ if }s \neq \ell, \ell+1, \\
\vPQi[d_{\ell-1}+1,x+y+1+d_{\ell-1}] & \text{ if } s = \ell, \\
\vPQi[x+y+2+d_{\ell-1},d_{\ell+1}] & \text{ if } s = \ell+1
\end{cases}
\]
and we set
\[
G_{P,Q} \colonequals
\mathfrak{S}_{[d_1]}
\times \mathfrak{S}_{[d_1+1,d_2]} \times
\cdots \times \mathfrak{S}_{[d_{\ell-2}+1,d_{\ell-1}]}
\times \mathfrak{S}_{[d_{\ell-1}+1,x+y+1+d_{\ell-1}]}
\times \mathfrak{S}_{[x+y+2+d_{\ell-1},d_{\ell+1}]}
\times \cdots \times \mathfrak{S}_{[d_k+1,n]}.
\]
Because we use the same permutation $a$ to obtain the class $\tsvPQi$ and the permutation~$\vPQi$, we get
\begin{equation}\label{eq_supp_vPQi}
\begin{split}
\supp(\tsvPQi)
&= u_{P,Q}  A_{\vPQi}
= \{ u_{P,Q} \vPQi v' \mid v' \in G_{P,Q} \} \\
&= \{ u_{P,Q} u_{P,Q} \tvPQi v' \mid v' \in G_{P,Q} \} \\
&= \{ \tvPQi v' \mid v' \in G_{P,Q} \},
\end{split}
\end{equation}
where $u_{P,Q}$ is the permutation defined in~\eqref{eq_def_of_a}.
Because of~\eqref{eq_supp_vPQi} and the definition of $\tvPQi$, for $u \in \supp(\tsvPQi)$, we have $u^{-1}(i+1)<u^{-1}(i)$ and
$\{ u(d_{a-1}+1),\dots,u(d_a)\} = \{ w(d_{a-1} +1),\dots,w(d_a)\}$  for all $a \in [k] \setminus \{\ell,\ell+1\}$.

Now we assume that $u \in \mathfrak{S}_n$ satisfies $u^{-1}(i+1)<u^{-1}(i)$ and $u[d_{a-1}+1,d_a] = w[d_{a-1}+1,d_a]$  for all $a \in [k] \setminus \{\ell,\ell+1\}$.
Then we have
\[
\begin{split}
u[d_{\ell-1}+1,d_{\ell}] \cup u[d_{\ell}+1,d_{\ell+1}]
&=
w[d_{\ell-1}+1,d_{\ell}] \cup w[d_{\ell}+1,d_{\ell+1}]\\
&= \widetilde{P} \cup \widetilde{Q} \cup \{i,i+1\}.
\end{split}
\]
We note that any element in $\widetilde{P}$ is less than $i$;
any element in $\widetilde{Q}$ is greater than $i+1$.
By setting $P = \{ u(z) \in \widetilde{P}\cup \widetilde{Q} \mid u(z) < i \}$
and $Q = \{u(z) \in \widetilde{P}\cup \widetilde{Q} \mid u(z) > i+1 \}$, we obtain $u \in \supp(\tsvPQi)$ by~\eqref{eq_supp_vPQi}. This proves the lemma.
\end{proof}
\begin{lemma}\label{lemma_eq_s(i)_u}
For $u \in \bigcup_{\substack{P \subset \widetilde{P}, \\ Q \subset \widetilde{Q}}} \supp(\tsvPQi)$, we have
\[
	\sigma_w^{(i)}(u) = (t_i - t_{i+1}) \prod_{a \in [k] \setminus  \{ \ell\}} (t_{u(d_{a}+1)} - t_{u(d_{a})}).
\]
\end{lemma}
\begin{proof}
	Let $P = \{  p_1 < p_2 < \dots < p_x \} \subset \widetilde{P}$ and $Q = \{  q_1 < q_2 < \dots < q_y \} \subset \widetilde{Q}$. Then at any point $u \in A_{\vPQi}$, we have
	\begin{equation}\label{eq_vPQi_u}
	\tsvPQi (u) = (t_{u(x+y+2+d_{\ell-1})} - t_{u(x+y+1+d_{\ell-1})})\prod_{a \in [k] \setminus  \{ \ell\}} (t_{u(d_{a}+1)} - t_{u(d_{a})})
	\end{equation}
	since  $D(\vPQi) = \{ d_a \mid a \in [k] \setminus \{ \ell\}\} \cup \{ x+y+d_{\ell-1}\}$ (see~\eqref{eq_descents_of_vPQi}) and by Theorem~\ref{thm:Perm_class}.

Suppose that $u \in \bigcup_{\substack{P \subset \widetilde{P}, \\ Q \subset \widetilde{Q}}} \supp(\tsvPQi)$.
By Lemma~\ref{lemma_union_of_supports}, we have $u^{-1}(i+1) < u^{-1}(i)$.
	We denote by $b = u^{-1}(i+1)$ and $b' = u^{-1}(i)$. By the assumption, we have $d_{\ell-1} < b < b'$.
	For given $P \subset \widetilde{P}$ and $Q \subset \widetilde{Q}$, the permutation $u$ is in the support $\supp(\tsvPQi)$ if and only if $\{u(d_{\ell-1}+1),\dots,u(\beta)\} = P \cup Q \cup \{i+1\}$ for some $b \leq \beta < b'$.
	Therefore, we obtain
	\begin{equation}\label{eq_s(i)_u_1}
	\sigma_w^{(i)}(u) =
	\sum_{\substack{P \subset \widetilde{P}, \\ Q \subset \widetilde{Q}}} \tsvPQi(u)
	= \sum_{b \leq \beta < b'} \widetilde{\sigma}_{P_{\beta},Q_{\beta}}^{(i)}(u)
	\end{equation}
	where $P_{\beta}$ and $Q_{\beta}$ are subsets satisfying $P_{\beta} \cup Q_{\beta} \cup \{i+1\} = \{u(d_{\ell-1}+1),\dots,u(\beta)\}$ for $b \leq \beta < b'$. Moreover, we have
	\[
	\bigcup_{b \leq \beta < b'} D(w_{P_{\beta},Q_{\beta}}^{(i)} )  = \{ \beta \mid b \leq \beta < b'\} \cup \{d_a \mid a \in [k] \setminus \{\ell\}\}.
	\]
	By applying~\eqref{eq_vPQi_u}, we obtain
	\begin{equation}\label{eq_s(i)_u_2}
	\begin{split}
	\sum_{b \leq \beta < b'} \widetilde{\sigma}_{P_{ \beta},Q_{\beta}}^{(i)}(u) &= \left( \sum_{b \leq \beta < b '} t_{u(\beta+1)} - t_{u(\beta)}\right) \prod_{a \in [k] \setminus  \{ \ell\}} (t_{u(d_{a}+1)} - t_{u(d_{a})}) \\
	& = (t_{u(b')} -t_{u(b)})\prod_{a \in [k] \setminus  \{ \ell\}} (t_{u(d_{a}+1)} - t_{u(d_{a})}) \\
	& = (t_{i} -t_{i+1})\prod_{a \in [k] \setminus  \{ \ell\}} (t_{u(d_{a}+1)} - t_{u(d_{a})}).
	\end{split}
	\end{equation}
	Combining equations~\eqref{eq_s(i)_u_1} and~\eqref{eq_s(i)_u_2}, we prove the lemma.
\end{proof}

\begin{proof}[Proof of Proposition~\ref{prop_si_action_on_sigma_i}]
For $u \in \mathfrak{S}_n$, suppose that $u \notin  \bigcup_{\substack{P \subset \widetilde{P}, \\ Q \subset \widetilde{Q}}} \supp(\tsvPQi)$ but in $\supp(\sigma_{s_i  w})$. By Lemma~\ref{lemma_union_of_supports} and  Proposition~\ref{prop:Perm_A}, we have $u^{-1}(i+1) > u^{-1}(i)$ and $u[d_{a-1}+1,d_a] = w[d_{a-1}+1,d_a]$  for all $a \in [k] \setminus \{\ell, \ell+1\}$.
	For this case, $(s_i  u)^{-1}(i+1) < (s_i  u)^{-1}(i)$. Accordingly, $s_i  u \in \bigcup_{\substack{P \subset \widetilde{P}, \\ Q \subset \widetilde{Q}}} \supp(\tsvPQi)$ and we have
	\[
	\begin{split}
	(s_i \cdot \sigma_w^{(i)} - \sigma_w^{(i)})(u)
	&= s_i (\sigma_w^{(i)}(s_i  u)) - \sigma_w^{(i)}(u) \\
	&= s_i \left[ (t_{i} -t_{i+1})\prod_{a \in [k] \setminus  \{ \ell\}} (t_{s_i  u(d_{a}+1)} - t_{s_i  u(d_{a})})\right] - 0 \\
	&= (t_{i+1} -t_i) s_i\left[\prod_{a \in [k] \setminus  \{ \ell\}} (t_{s_i  u(d_{a}+1)} - t_{s_i  u(d_{a})})\right] \\
	&= (t_{i+1} -t_i)  \prod_{a \in [k] \setminus  \{ \ell\}} (t_{u(d_{a}+1)} - t_{u(d_{a})}).
	\end{split}
	\]
	Here, the second equality derives from Lemma~\ref{lemma_eq_s(i)_u}.
	On the other hand, we have $D(s_i w) = \{ d_a \mid a\in [k] \setminus \{\ell\}\}$ and
	\[
	\sigma_{s_i w}(u) =  \prod_{a \in [k] \setminus  \{ \ell\}} (t_{u(d_{a}+1)} - t_{u(d_{a})})
	\]
	by Theorem~\ref{thm:Perm_class}. This proves the proposition for $u \notin \bigcup_{\substack{P \subset \widetilde{P}, \\ Q \subset \widetilde{Q}}} \supp(\tsvPQi)$ but in $\supp(\sigma_{s_i  w})$.
Suppose that $u \notin \bigcup_{\substack{P \subset \widetilde{P}, \\ Q \subset \widetilde{Q}}} \supp(\tsvPQi)$ and $u \notin \supp(\sigma_{s_i  w})$. Then we obtain
\[
(s_i \cdot \sigma_w^{(i)} - \sigma_w^{(i)})(u)
	= s_i (\sigma_w^{(i)}(s_i  u)) - \sigma_w^{(i)}(u) = 0.
\]
	Therefore, this proves the proposition for $u \notin \bigcup_{\substack{P \subset \widetilde{P}, \\ Q \subset \widetilde{Q}}} \supp(\tsvPQi)$.

	Suppose that $u \in \bigcup_{\substack{P \subset \widetilde{P}, \\ Q \subset \widetilde{Q}}} \supp(\tsvPQi)$. Then we have  $s_i u  \notin \bigcup_{\substack{P \subset \widetilde{P}, \\ Q \subset \widetilde{Q}}} \supp(\tsvPQi)$.
	Moreover, we obtain
	\[
	\begin{split}
	(s_i \cdot \sigma_w^{(i)} - \sigma_w^{(i)})(u)
	&= s_i \cdot\sigma_w^{(i)}(s_i u) -  \sigma_w^{(i)}(u) \\
	&= 0 - (t_i - t_{i+1}) \prod_{a \in [k] \setminus  \{ \ell\}} (t_{u(d_{a}+1)} - t_{u(d_{a})})\\
	&= (t_{i+1} -t_i)\prod_{a \in [k] \setminus  \{ \ell\}} (t_{u(d_{a}+1)} - t_{u(d_{a})}).
	\end{split}
	\]
	This completes the proof.
\end{proof}


We provide a corollary of Proposition~\ref{prop_si_action_on_sigma_i} which will be subsequently used.
\begin{corollary}\label{cor_sigma_invariant_si_singular}
Suppose $w \in \mathfrak{S}_n$ with descent set $D(w)$ and such that $ w^{-1}(i+1) = w^{-1}(i)-1$. As above, let $d_{\ell} \in D(w)$ such that $w(d_{\ell}) = i+1$. If $d_{\ell}- 1, d_{\ell}, d_{\ell}+1 \in D(w) \cup \{0,n\}$, then
\[
s_i \cdot \sigma_w = \sigma_w + (t_{i+1} - t_i) \sigma_{s_i w}.
\]
In particular, $\sigma_w \in H^{\ast}(\mathcal H_n)$ is invariant under $s_i$.
\end{corollary}
\begin{proof}
	The conditions $d_{\ell}- 1, d_{\ell}, d_{\ell}+1 \in D(w) \cup \{0,n\}$ and $d_{\ell}=w^{-1}(i+1) = w^{-1}(i) - 1$ implies that we have the following one-line notation of $w$:
	\[
	w = \dots \textcolor{red}{|} i+1 \textcolor{red}{|} i \textcolor{red}{|} \dots.
	\]
	In this case, the sets $\widetilde{P}$ and $\widetilde{Q}$ are both empty. Therefore, the class $\sigma_w^{(i)}$ is given by
	$\sigma_w^{(i)} = \sigma_w$,	so  we have
	\[
	(t_{i+1} - t_i) \sigma_{s_i w} = s_i \cdot \sigma_w - \sigma_w.
	\]
	Accordingly, as an element in $H^{\ast}(\mathcal H_n)$, the class $\sigma_w$ is invariant under the action of $s_i$.
\end{proof}

\begin{example}
	Let $n = 5$ and consider $w = 21435$. Then $\widetilde{P} = \{ 1\}	$ and $\widetilde{Q} = \{ 5\}$. For $P \subset \widetilde{P}$ and $Q \subset \widetilde{Q}$, permutations $\tvPQi$ and $\vPQi$ are given as follows.
	\begin{center}
		$\tvPQi$:	\begin{tabular}{|c|*{2}{c|}}
			\hline
			\diagbox{$P$}{$Q$}
			&\makebox[3em]{$\{5\}$}
			&\makebox[3em]{$\emptyset$} \\ \hline \hline
			$\{1\}$ & $2 \textcolor{red}{|} 145 \textcolor{red}{|} 3$ & $2\textcolor{red}{|} 14 \textcolor{red}{|} 35$  \\\hline
			$\emptyset$ & $245\textcolor{red}{|}13$  & $24\textcolor{red}{|}135$ \\\hline
		\end{tabular}
		$\qquad$
		$\vPQi$:	\begin{tabular}{|c|*{2}{c|}}
			\hline
			\diagbox{$P$}{$Q$}
			&\makebox[3em]{$\{5\}$}
			&\makebox[3em]{$\emptyset$} \\ \hline \hline
			$\{1\}$ & $2 \textcolor{red}{|} 145 \textcolor{red}{|} 3$ & $2\textcolor{red}{|} 14 \textcolor{red}{|} 35$  \\\hline
			$\emptyset$ & $4\textcolor{red}{|}2 5\textcolor{red}{|}13$  & $4\textcolor{red}{|}2 \textcolor{red}{|}135$ \\\hline
		\end{tabular}
	\end{center}
	Therefore, we have
	\[
	\sigma_{21435}^{(3)} = \sigma_{21453} + \sigma_{21435} + s_{2,4} \sigma_{42513} + s_{2,4} \sigma_{42135}.
	\]
Here, $s_{2,4}$ is a transposition exchanging $2$ and $4$.
	Because $s_{2,4} = s_2s_3s_2$, we have
	\[
	\begin{split}
	s_{2,4} \cdot \sigma_{42513} &= (s_2s_3s_2)\cdot \sigma_{42513} = (s_2s_3)\cdot \sigma_{43512} \\
	&= s_2 \cdot ((t_4-t_3) \sigma_{34512} + \sigma_{43512} + \sigma_{45312} - \sigma_{35412}) \quad \text{(by Proposition~\ref{prop_si_action_on_sigma_i})}\\
	&= (t_4-t_2) \sigma_{24513} + \sigma_{42513} + \sigma_{45213} - \sigma_{25413}, \\
	s_{2,4} \cdot \sigma_{42135} &= (s_2s_3s_2) \cdot \sigma_{42135}
	= (s_2s_3) \cdot \sigma_{43125} \\
	&= s_2 \cdot ((t_4-t_3) \sigma_{34125} + \sigma_{43125}) \quad \text{(by Corollary~\ref{cor_sigma_invariant_si_singular})}\\
	&= (t_4-t_2) \sigma_{24135} + \sigma_{42135}.
	\end{split}
	\]
	Accordingly, we obtain
	\[
	\begin{split}
	s_3 \cdot \sigma_{21435}
	& = (t_4-t_3) \sigma_{21345} + \sigma_{21435} \\
	&\qquad + \sigma_{21453}+ (t_4-t_2) \sigma_{24513} + \sigma_{42513} + \sigma_{45213} - \sigma_{25413}
	+ (t_4-t_2) \sigma_{24135} + \sigma_{42135} \\
	& \quad \qquad - (\sigma_{21354} + (t_3-t_2) \sigma_{23514} + \sigma_{32514} + \sigma_{35214} - \sigma_{25314}
	+ (t_3-t_2) \sigma_{23145} + \sigma_{32145}).
	\end{split}
	\]
\end{example}

\subsection{Permutation module decompositions}\label{sec:hong_permutohedral}
The $e$-positivity conjecture on the chromatic quasisymmetric functions of Shareshian and Wachs \cite{SW} is shown independently by Brosnan--Chow and Guay-Paquet, to be  equivalent to the conjecture that $H^{2k}(\Hess(X, h))$ is isomorphic to a direct sum of permutation modules of $\frak S_n$ for each $k$.
\emph{Permutation modules} of the symmetric group $\mathfrak{S}_n$ are
$M^{\la}=1 \uparrow_{ \mathfrak{S}_{J_1}\times\cdots \times \mathfrak{S}_{J_{k+1}}}^{\mathfrak{S}_n}= \mathbb{C}[\mathfrak{S}_n(\mathbf t^\la)]$ for partitions $\la=(\la_1, \dots, \la_{k+1})$ of $n$, where $J_s\colonequals
[\la_1+\cdots+\la_{s-1}+1,\la_1+\cdots+\la_s]$ for $s=1, \dots, k+1$ and $\mathbf t^\la=(J_1, \dots, J_{k+1})$. A natural basis of $M^\la$ is the set $\{(J_1, \dots, J_{k+1}) \mid |J_s|=\la_s,\,\, \bigcup_s J_s=[n]\}$. See \cite{Sag} for the representation theory of the symmetric groups.
We remark that $M^{\bf a}$, for a composition ${\bf a}$ of $n$ can be defined in the same way as $M^\la$ is defined, whereas $M^{\bf a}$ is isomorphic to $M^{\la({\bf a})}$ for the partition $\la({\bf a})$ obtained from ${\bf a}$ by arranging the parts of ${\bf a}$ in nonincreasing order.

\begin{proposition} [\cite{BrosnanChow18},  \cite{G-P2}]
	For a Hessenberg function $h$,
	\[ \sum_k \mathrm{ch}  H^{2k}(\Hess(S,h))t^k=\omega X_{G(h)}(\mathbf x, t)\,,\]
	where  $\mathrm {ch}$ is the Frobenius characteristic map and $\omega$ is the involution on the ring $\Lambda$ of symmetric functions   sending elementary symmetric functions to complete  homogeneous symmetric functions.
\end{proposition}

The $e$-positivity conjecture is proved to be true for permutohedral varieties by Shareshian and Wachs in their seminal paper \cite{SW}, where they provided a closed form formula for the expansion of the corresponding chromatic quasisymmetric functions as sums of elementary symmetric functions (Theorem C.4 or Table 1 of \cite{SW}).
Transforming this formula in $\Lambda[t]$ into a formula in $R[t]$ via the isomorphism of the ring $\Lambda$  of symmetric functions with the ring $R$ of $\mathfrak S_n$-modules, we obtain the following formula in $R[t]$.

\begin{proposition} [Theorem~C.4 or Table 1~of \cite{SW}] \label{prop:expansion}
	\begin{equation}\label{eq:permutation module decomposition} \sum_{k=0}^{n-1}H^{2k}(\mathcal H_n)t^k = \sum_{m=1}^{\lfloor\frac{n+1}{2}\rfloor }
	\sum_{\substack{ k_1, \dots, k_m \geq 2,\\ \sum k_i = n+1 }}
	M^{(k_1-1, k_2, \dots, k_m)} t^{m-1}\prod_{i=1}^m [k_i -1]_t.
\end{equation}
	Here, $[k]_t$ is the polynomial
\[
1+t + \dots + t^{k-1} = \frac{1-t^k}{1-t}
\]
	in $t$.
\end{proposition}
Let $\mathcal G_k$ be the set of $w \in \mathfrak S_n$ with $\mathrm{des}(w)=k$  such that  $A_w$ contains $w_0$, the longest element in~$\mathfrak S_n$, and let $\mathcal G$ be the union  $\bigcup_{k=0}^{n-1}\mathcal G_k$.

We will construct a basis $\bigcup_{k=0}^{n-1}\{ \widehat{\sigma}_w \mid w \in \La \}$ of $\sum_{k=0}^{n-1}H^{2k}(\mathcal H_n)$ generating permutation modules in the right hand side of~\eqref{eq:permutation module decomposition}.

Recall that for a permutation $w \in \mathfrak S_n$ with $D(w) =\{d_1 < \cdots < d_k\}$, the support $A_w$ of $\sigma_w$ is $(\mathfrak S_{J_1} \times \cdots \times \mathfrak S_{J_{k+1}})w$, where $J_s=w[d_{s-1}+1,d_s]$ for $s=1,2,\dots, k+1$. Here, we set $d_0=0$ and $d_{k+1}=n$ (Proposition~\ref{prop:Perm_A}).

\begin{lemma}\label{lemma_bijection_between_Lk_and_compositions}
	There is a bijective correspondence between the set of compositions and $\mathcal G\colonequals\bigcup_{k = 0}^{n-1} \La$. Indeed, a composition $\mathbf a = (a_1,\dots,a_{k+1})$ corresponds to an element $w(\mathbf a)$ in $\La$ defined by
\[
w(\mathbf a) \colonequals n-d_1+1 \cdots n\textcolor{red}{|}n-d_2+1 \cdots n-d_1 \textcolor{red}{|} \cdots \textcolor{red}{|} 1 \cdots n-d_k,
\]
where $d_s=\sum_{i=1}^{s} a_i$ for $s=1, \dots, k$.
\end{lemma}

\begin{proof} Let ${\bf a}=(a_1, \dots, a_{k+1})$ be a composition of $n$.
	Then, clearly, $A_{w({\bf a})}$ contains $w_0$, which shows that $w({\bf a})\in\La$.
	
	Conversely, suppose that $A_w$ contains $w_0$. Let ${\bf a}$ be the composition with $S({\bf a} )=D(w)=\{d_1<\cdots<d_k \}$. Then $w$ is of the form
	\[
n-d_1+1 \cdots n\textcolor{red}{|} n-d_2+1 \cdots n-d_1 \textcolor{red}{|} \cdots \textcolor{red}{|}1 \cdots n-d_k\,,
\]
	and thus we have $w({\bf a})=w$.
\end{proof}
We note that $w(\bf a)$ is the Bruhat-maximal element in the set of minimal coset representatives for $\mathfrak{S}_n/\mathfrak{S}_{\bf a}$, where $\mathfrak{S}_{\bf a}$ denotes the Young subgroup generated by $s_i$ for $i \in [n-1] \setminus S(\bf a)$.

\begin{remark} In general, when $h\colon[n] \rightarrow [n]$ is a Hessenberg function,
	the intersection of $X_w^{\circ}$ with the nilpotent Hessenberg variety $\Hess(N,h)$ is nonempty if and only if
\[
w^{-1}(w(j)-1) \leq h(j) \text{ for all } j \in [n]
\]
(see Lemma 2.3 of \cite{AHHM19}). When $h=(2,3,\dots,n,n)$, our $\mathcal G = \bigcup_k \mathcal G_k$ is the image of the set of these $w$'s  by the involution $\iota\colon\mathfrak S_n \rightarrow \mathfrak S_n$ given by $(\iota(w))(i) = n-w(i) +1$ for $i \in [n]$. We expect that  for general $h$ the latter set  will play the same role as that of $\mathcal G$.
\end{remark}

\begin{definition}[cf. Chow~\cite{Chow_erasing}]\label{def:erasure}
	\begin{enumerate} \item For a  set $D=\{d_1 <\cdots <d_k\}\subset [n]$ of \emph{marks}, define the {\it erasure}  $e(D)$ of $D$ by
	\[
e(D)=\{ d \in D \mid d \not=1 \text{ and } d-1 \not \in D \}\,.
\]
	
		\item For a composition ${\bf a}=(a_1, \dots, a_{k+1})$ of $n$,
		we define the {\it erasure} $\widehat{\bf a}$ of ${\bf a}$  by the composition with $S(\widehat{\bf a}) = e(S({\bf a})) $.
	\end{enumerate}
\end{definition}

\begin{definition}\label{def:generator}
	For $w \in \mathcal G$ with $D(w)= D=\{d_1 <\cdots <d_k\}$, let $e(D)=\{\epsilon_1<\cdots < \epsilon_l\}\subset D$ be the erasure of $D$. For $s=1, 2, \dots, k+1$, we
	let $J_s=w[d_{s-1}+1,d_s]$ and for $t=1, 2, \dots, l+1$, we let $\widehat{J}_t=w[\epsilon_{t-1}+1, \epsilon_t]$, where $d_{0}=\epsilon_{0}=0$ and $d_{k+1}=\epsilon_{l+1}=n$.
	\begin{enumerate}
		\item Define a subgroup $\mathfrak S_w$ of $\mathfrak S_n $ as
\[
\mathfrak S_w\colonequals\mathfrak S_{\widehat{J}_1}\times \cdots \times \mathfrak S_{\widehat{J}_{l+1}}\,,
\]
		that contains $\mathfrak S_{J_1}\times \cdots \times \mathfrak S_{J_{k+1}}$ as a subgroup.
		\item Define
\[
\mathfrak S_w^{\circ}\colonequals\mathfrak S_w / {(\mathfrak S_{J_1}\times \cdots \times \mathfrak S_{J_{k+1}})}\,.
\]
		\item\label{sigmahat} Define an element in $H^{2k}_T(\mathcal H_n)$
\[
\widehat{\sigma}_{w}\colonequals\sum_{\bar{v} \in \mathfrak S_w^{\circ}} v \cdot \sigma_{w},
\]
where $\bar{v}$ is the coset that contains $v \in \frak{S}_w$.
		
		\item Define ${\bf a}(w)$ by the composition such that  $S({\bf a}(w)) =D(w)$. Put $\widehat{\bf a}(w)\colonequals\widehat{{\bf a}(w) }$.
		
	\end{enumerate}
\end{definition}

We note that each $\mathfrak S_{J_s}$, $s=1, \dots, k+1$, stabilizes $\sigma_w$ in Definition~\ref{def:generator}, thus the element $\widehat{\sigma}_{w}$ in~\eqref{sigmahat} is well defined.

Before we continue on the construction of permutation module decomposition of $H^*(\mathcal H_n)$, we state a conjecture on bases of the equivariant cohomology space of $\mathcal H_n$ due to Chow, whose proof follows from our main theorem in the current section; see Corollary~\ref{cor:Chow_conjecture}. The construction of the suggested basis elements by Chow is different from that obtained in the following lemma; however, they can be shown to coincide using the explicit formula for the classes $\sigma_w$:
\begin{lemma}\label{lem:generator} For $w \in \mathcal G$ with $D(w)=\{d_1 <\cdots <d_k\}$, let $\widehat{\mathbf a}=\widehat{\mathbf a}(w)$.
	Then
	\begin{enumerate}
		\item $\supp (\hats_w) =  A_{w(\widehat{\mathbf a})}$;
		\item for $u\in \supp (\hats_w)$, $\hats_w(u)=\prod_{s=1}^{k}(t_{u(d_s+1)}-t_{u(d_s)})$;
\item the stabilizer subgroup of $\hats_w \in
H^{\ast}_T(\mathcal H_n)$ is the same as $\mathfrak{S}_w$.
	\end{enumerate}
\end{lemma}
\begin{proof}
By Proposition~\ref{prop:Perm_A}, the support of $\sigma_w = A_w$ is given by
	\[
	A_w = (\mathfrak{S}_{J_1} \times \cdots \times \mathfrak{S}_{J_{k+1}})w.
	\]
Because of the definition of $\mathfrak{S}_w^{\circ}$, the support of the class $\widehat{\sigma}_{w}\colonequals\sum_{\bar{v} \in \mathfrak S_w^{\circ}} v \cdot \sigma_{w}$ is
\begin{equation}\label{eq_supp_hats_w}
\supp(\widehat{\sigma}_w) = (\mathfrak{S}_w/(\mathfrak{S}_{J_1} \times \cdots \times \mathfrak{S}_{J_{k+1}})) \times A_w   = \mathfrak{S}_w \cdot w.
\end{equation}
Moreover, we have $D(w(\widehat{\mathbf a})) = e(D(w))$.
This proves the first statement. The second statement follows from Theorem~\ref{thm:Perm_class} and the definition of the action (see Subsection~\ref{subsec_4.1}). For the third statement, because of~\eqref{eq_supp_hats_w}, the stabilizer subgroup of the class $\hats_w$ is contained in~$\mathfrak{S}_w$. Moreover, for any $v \in \mathfrak{S}_w$ and $u \in \supp \hats_w$, we have $v^{-1}u \in \supp(\hats_w)$ and
\[
(v \cdot \hats_w)(u) = v(\hats_w(v^{-1}u))
=  v \left[\prod_{s=1}^k(t_{(v^{-1}u)(d_s+1)} - t_{(v^{-1} u) (d_s)})\right]
= \prod_{s=1}^k (t_{u(d_s+1)} - t_{u(d_s)}) = \hats_w(u),
\]
and hence $v$ stabilizes the class $\hats_w$. This proves the third statement.
\end{proof}

\begin{conjecture}[Erasing marks conjecture (Chow~\cite{Chow_erasing})]
	\[
\bigcup_{k=0}^{n-1}\bigcup_{w \in \mathcal G_k}\{ v \cdot \widehat{\sigma}_w \mid \bar{v} \in \mathfrak S_n/\mathfrak S_w\}
\]
	forms a basis of the equivariant cohomology space $H^*_T(\mathcal H_n)$.
Here, $\bar{v}$ is the coset that contains $v \in \mathfrak{S}_n$.
\end{conjecture}


We state the main theorem of this section, whose proof will be provided in Section~\ref{sec:lee_permutohedral}. We will use the same notation for the image of $\widehat{\sigma}_w$ in the singular cohomology space $H^*(\mathcal H_n)$.

\begin{theorem} \label{thm:erasing conjecture}   For  $w \in \Gk$, let $M(w)\colonequals\mathbb C \mathfrak S_n (\widehat{\sigma}_{w})$ be the $\mathfrak S_n$-module  generated by $\widehat{\sigma}_{w}$.
	\begin{enumerate}
		\item For $w \in \Gk$, the $\mathfrak S_n$-module $M(w)$  is isomorphic to the permutation module  $M^{\widehat{\bf a}(w)}$.
		\item $H^{2k}(\mathcal{H}_n) = \bigoplus _{w \in \Gk}M(w)$.
	\end{enumerate}
	
\end{theorem}

Theorem~\ref{thm:erasing conjecture} proves the erasing marks conjecture of Chow.

\begin{corollary}[Erasing marks conjecture is true]\label{cor:Chow_conjecture}
	For each $k=0, \dots, n-1$,
\[
\bigcup_{w \in \mathcal G_k}\{ v\cdot \widehat{\sigma}_w\in H_T^{2k}(\mathcal H_n) \mid \bar{v} \in \mathfrak S_n/\mathfrak S_w\}
\]
	forms a basis of the $(2k)$th equivariant cohomology  module $H^{2k}_T(\mathcal H_n)$ over $\mathbb C[t_1, \dots, t_n]$.
\end{corollary}

\begin{proof} For a fixed $k$ and $w\in\mathcal G_k$, we have $\left|\mathfrak S_n/\mathfrak S_w\right|=\dim (M^{\widehat{\bf a}(w)})$. Hence, the second part of Theorem~\ref{thm:erasing conjecture} proves that $\bigcup_{w \in \mathcal G_k}\{ v\cdot \widehat{\sigma}_w\in H^{2k}(\mathcal H_n) \mid \bar{v} \in \mathfrak S_n/\mathfrak S_w\}$ is a $\mathbb C$-basis of $H^{2k}(\mathcal H_n)$ for each~$k$.
	
	According to~\cite[Lemma~2.1]{MP06}, for a manifold $M$ with a smooth action of a compact torus $T=(S_1)^n$ such that the fixed point set is finite and non-empty, then $H^{\ast}_T(M)$ is free as an $\C[t_1,\dots,t_n]$-module if and only if $H^{\text{odd}}(M) = 0$. In this case,
	\[
	H^{\ast}_T(M) \cong   H^{\ast}(M) \otimes \C[t_1,\dots,t_n]
	\]
	as $\C[t_1,\dots,t_n]$-modules. Because the odd degree cohomology  of the permutohedral variety vanishes, we obtain
	\begin{equation}\label{eq_equiv_and_singular_relation}
	H^{\ast}_T(\mathcal H_n) \cong H^{\ast}(\mathcal H_n) \otimes \C[t_1,\dots,t_n]
	\end{equation}
	as $\C[t_1,\dots,t_n]$-modules. This shows that
	$\bigcup_{w \in \Gk} \{ v \cdot \widehat{\sigma}_w \in H^{2k}_T(\mathcal H_n) \mid \bar{v} \in \mathfrak{S}_n / \mathfrak{S}_w\}$ is  a $\C[t_1,\dots,t_n]$-module basis of $H^{2k}_T(\mathcal H_n)$.
\end{proof}

\begin{example}\label{example:n5a2 cont cont} If $n=5$ and $k=2$, we obtain $\mathcal G_k =\{54123, 53412,52341,45312,45231,34521\}$.
\begin{table}
\begin{tabular}{c|ccccc}
\toprule
			$w$ & ${\bf a}(w)$  & $\widehat{\bf a}(w)$ & & $\mathfrak S_w$ & $\mathfrak S_w^{\circ}$ \\
\midrule
			$5  \textcolor{red}{|} 4 \textcolor{red}{|}123$  & $(1,1,3)$ & $(5)$  & $5\vdots 4 \vdots 123$ & $\mathfrak S_{\{5,4,1,2,3\}}$ & $\mathfrak S_{\{5,4,1,2,3\}}/ \mathfrak{S}_{\{1,2,3\}}$ \\
			$5\textcolor{red}{|}34\textcolor{red}{|}12$   & $(1,2,2)$   & $(3,2)$  & $5\vdots 34\textcolor{red}{|}12$ & $\mathfrak S_{\{5,3,4\}} \times \mathfrak S_{\{1,2\}}$ & $\mathfrak S_{\{5,3,4\}}/ \mathfrak{S}_{\{3,4\}}$ \\
			$5\textcolor{red}{|}234\textcolor{red}{|}1$  & $(1,3,1)$    & $(4,1)$  & $5 \vdots 234\textcolor{red}{|}1$ & $\mathfrak S_{\{5,2,3,4\}}$ & $\mathfrak S_{\{5,2,3,4\}}/ \mathfrak S_{\{2,3,4\}}$ \\
			$45\textcolor{red}{|}3\textcolor{red}{|}12$   & $(2,1,2)$   & $(2,3)$ & $45\textcolor{red}{|}3\vdots 12$ & $\mathfrak S_{\{4,5\}} \times \mathfrak  S_{\{3,1,2\}}$ & $\mathfrak  S_{\{3,1,2\}}/ \mathfrak S_{\{1,2\}}$ \\
			$45\textcolor{red}{|}23\textcolor{red}{|}1$   & $(2,2,1)$  & $(2,2,1)$ & $45\textcolor{red}{|}23\textcolor{red}{|}1$ & $\mathfrak S_{\{4,5\}} \times \mathfrak S_{\{2,3\}}$ &  $\{ e\}$\\
			$345\textcolor{red}{|}2\textcolor{red}{|}1$  & $(3,1,1)$   & $(3,2)$ & $345\textcolor{red}{|}2\vdots 1$ & $\mathfrak S_{\{3,4,5\}} \times \mathfrak S_{\{1,2\}}$ & $\mathfrak S_{\{2,1 \}}$ \\
\bottomrule
\end{tabular}
\caption{Let $n = 5$ and $k = 2$. For $w \in \Gk$, ${\bf a}(w)$, $\widehat{\bf a}(w)$, $\mathfrak S_w$, $\mathfrak S_w^{\circ}$.}\label{table_example_34555}
\end{table}
Using the computation of $\widehat{\bf a}(w)$ given in Table~\ref{table_example_34555} and Theorem~\ref{thm:erasing conjecture}, we have
\[
H^{2\cdot 2}(\mathcal{H}_5) = M^{(5)} \oplus (M^{(3,2)})^{\oplus 3} \oplus M^{(4,1)} \oplus M^{(2,2,1)}.
\]	
\end{example}

First, we show that the permutation modules $M(w)$ are exactly those appear in the known decomposition of $H^*(\mathcal H_n)$, as described  in Proposition~\ref{prop:expansion}.

\begin{proposition} \label{prop:young subgroup type} We have
\[
\sum_{k=0}^{n-1}\sum_{w \in \mathcal G_k}M^{\widehat{\bf a}(w)} \, t^k = \sum_{m=1}^{\lfloor\frac{n+1}{2}\rfloor}
	\sum_{\substack{k_1, \dots, k_m \geq 2,\\ \sum k_i = n+1 }}
	M^{(k_1-1, k_2, \dots, k_m)} \, t^{m-1}\prod_{i=1}^m [k_i -1]_t.
\]
	In other words, for each $(k_1-1, k_2, \dots, k_m)$ with $k_1, \dots, k_m \geq 2$ and $ \sum k_i = n+1$,
\[
\vert\{w \in \mathcal G_k \mid \widehat{\bf a}(w) =(k_1-1, k_2, \dots, k_m)  \}\vert
\]
 is equal to the coefficient of $t^k$ in the polynomial $t^{m-1}\prod_{i=1}^m [k_i -1]_t.$

	In particular, $\sum_{k=0}^{n-1} \sum_{w \in \mathcal G_k} \dim M^{\widehat{\bf a}(w)}=\dim H^*(\mathcal H_n) = n!$.
\end{proposition}

\begin{proof} We will show that  for each $(k_1, k_2, \dots, k_m-1)$ with $k_1, \dots, k_m \geq 2$ and $ \sum k_i = n+1$,
\[
\vert\{w \in \mathcal G_k \mid \widehat{\bf a}(w) =(k_1 , k_2, \dots, k_m-1)  \}\vert
\]
is equal to the coefficient of $t^k$ in the polynomial $t^{m-1}\prod_{i=1}^m [k_i -1]_t $.
	By Lemma~\ref{lemma_bijection_between_Lk_and_compositions}, there is a bijective map from the set of compositions ${\bf a}=(a_1, \dots, a_{k+1})$ to $\mathcal G_k$. We consider the set $\mathcal C _{(k_1, \dots, k_{m}-1)}$ of compositions ${\bf a}=(a_1, \dots, a_{k+1})$  whose erasure $\widehat{\bf a}$ is $(k_1, \dots, k_m-1)$.  Regard $t^{m-1}\prod_{i=1}^m [k_i -1]_t$ as
	\[ (1+ t + \cdots + t^{k_1-2})t(1+t+ \cdots + t^{k_2-2})t\cdots t(1+t+ \cdots + t^{k_m-2}).\]
	A monomial $t^{l_1}t t^{l_2}t \cdots t t^{l_{m}}$ in this polynomial corresponds to a composition
	\[(\underbrace{1,1,1}_{l_1}, l'_1,  \underbrace{1,1,1}_{l_2},l'_2, \cdots   , \underbrace{1,1,1}_{l_m}, l'_m) \]
	where $l_j +l'_j = k_j -1$ for $j=1, \dots, m  $, whose erasure is $(k_1, k_1, \dots, k_m-1)$.
\end{proof}

To complete the proof of Theorem~\ref{thm:erasing conjecture}, we need to show that $H^{2k}(\mathcal H_n)$ is linearly generated by $M(w)$'s, where $w \in \mathcal G_k$. In other words,
$\sigma_w$ with $\ell_h(w)=\mathrm{des}(w)=k$ can be expressed as a linear combination of elements in $ \bigcup_{w \in \mathcal G_k}\{v \widehat{\sigma}_w \mid  v \in \mathfrak S_n/\mathfrak S_w \}$.
The proof will be provided in the next section.

\section{Proof of Theorem~\ref{thm:erasing conjecture}}\label{sec:lee_permutohedral}
In this section, we prove Theorem~\ref{thm:erasing conjecture}.
For a permutation $w \in \mathfrak{S}_n$, we say that $\mathbf{a}$ is the \textit{composition of~$w$} if $D(w) = S(\mathbf a)$
and we denote it by $\mathbf a (w)$.
We first consider the case when the erasing occurs consecutively, that is, the composition  $\mathbf a$ of $w$ is given by
\begin{equation}\label{eq_composition_a_erasing_conse}
(\underbrace{1,\dots,1}_{m}, a_{m+1},\dots,a_{k+1})
\end{equation}
where $m \geq 0$ and $a_j>1$ for $m+1 \leq j \leq k+1$.
For this case, by Lemma~\ref{lemma_bijection_between_Lk_and_compositions},  the corresponding element $w(\mathbf a)$ in~$\La$ is given by
\begin{equation}\label{eq_wa_in_sec6}
w(\mathbf a) =   n
 \textcolor{red}{|} n-1
 \textcolor{red}{|}  \dots
 \textcolor{red}{|}n-m+1 \textcolor{red}{|} y ~ y+1 ~ \dots~  y+{a_{m+1}}-1
\textcolor{red}{|} \dots .
\end{equation}
Here, $y = a_{m+2} + \cdots + a_{k+1} +1$ and we decorate the places where the descents appear.

Now we describe the element $\widehat{\sigma}_{w(\mathbf a)}$ using a certain (edge) labeled graph $G(\mathbf a) = (V,E)$ embedded in the Euclidean space $\R^m$ (see Proposition~\ref{prop_hat_w_description_using_G}).
\begin{itemize}
	\item The set $V = V(G(\mathbf a))$ of vertices is given by
	\[
	V = \{ (z_1,\dots,z_m) \in \Z^m \mid a_{m+1} -1 \ge z_1 \ge z_2 \ge \dots \ge z_m \ge 0\}.
	\]
If $a_{m+1}$ does not exist, that is, the composition $\mathbf a$ ends with $1$, then we set $V = \{(0,\dots,0)\} \subset \Z^m$.
If $m = 0$, then $V = \emptyset$.
	\item Two vertices $(z_1,\dots,z_m)$ and $(z_1',\dots,z_m')$ are connected by an edge if and only if
	$|z_j-z_j'| = 1$ for some $j \in [m]$ and all the other coordinates are the same. An edge with $z_j < z_j'$ is labeled by the value $n-m+(j-1) - z_j$.
\end{itemize}
For the case when the erasing does not occur consecutively, we may break the composition $\mathbf a$ into smaller pieces $\mathbf a_1,\dots, \mathbf a_{\beta}$ so that the erasing appears consecutively in each piece, and moreover, the erasure $\widehat{\mathbf a}$ is the concatenation of $\widehat{\mathbf a}_1,\dots,\widehat{\mathbf a}_{\beta}$.
\begin{definition}
Let $\mathbf a$ be a composition of $n$. Then $({\mathbf a}_1,\dots,{\mathbf a}_{\beta})$ is an \emph{admissible decomposition} of~$\mathbf a$ if
\begin{enumerate}
\item each $\mathbf a_i$ is of the form in~\eqref{eq_composition_a_erasing_conse};
\item the erasure $\widehat{\mathbf a}$ is the concatenation of $\widehat{\mathbf a}_1,\dots,\widehat{\mathbf a}_{\beta}$;
\item the number $\beta$ of compositions in the decomposition is the minimum among the decompositions satisfying (1) and (2).
\end{enumerate}
\end{definition}

   For instance, for a composition $\mathbf a = (2,1,1,3,4,1,1,1,5,1,2)$, if we break it into
\[
\mathbf a_1 = (2), \mathbf{a}_2= (1,1,3,4), \mathbf{a}_3 = (1,1,1,5), \mathbf{a}_4 = (1,2),
\]
then $(\mathbf a_1,\mathbf a_2, \mathbf a_3, \mathbf a_4)$ is an admissible decomposition of $\mathbf a$.
In this case, the erasure $\widehat{\mathbf a}$ is $(2,5,4,8,3)$, and we have $\widehat{\mathbf a}_1 = (2), \widehat{\mathbf a}_2 = (5,4)$, $\widehat{\mathbf a}_3 = (8)$, $\widehat{\mathbf a}_4 = (3)$.
See the following diagram:
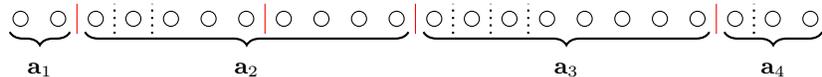
\[
\begin{tikzpicture}[scale = 0.5]
\foreach \x in {1,...,22}{
	\draw (\x,0) circle (0.2cm) ;
}
\draw [thick,decorate,decoration={brace,amplitude=5pt,mirror},xshift=0.4pt,yshift=-10pt]
(0.7,0) -- (2.3,0)
node[black,midway,yshift=-0.5cm] {\footnotesize $\mathbf a_1$};
\draw [thick,decorate,decoration={brace,amplitude=5pt,mirror},xshift=0.4pt,yshift=-10pt]
(2.7,0) -- (11.3,0)
node[black,midway,yshift=-0.5cm] {\footnotesize $\mathbf a_2$};
\draw [thick,decorate,decoration={brace,amplitude=5pt,mirror},xshift=0.4pt,yshift=-10pt]
(11.7,0) -- (19.3,0)
node[black,midway,yshift=-0.5cm] {\footnotesize $\mathbf a_3$};
\draw [thick,decorate,decoration={brace,amplitude=5pt,mirror},xshift=0.4pt,yshift=-10pt]
(19.7,0) -- (22.3,0)
node[black,midway,yshift=-0.5cm] {\footnotesize $\mathbf a_4$};

\draw[red] (2.5,-0.4)--(2.5,0.4);
\draw[red] (7.5,-0.4)--(7.5,0.4);
\draw[red] (11.5,-0.4)--(11.5,0.4);
\draw[red] (19.5,-0.4)--(19.5,0.4);

\foreach \y in {3.5, 4.5, 12.5, 13.5, 14.5, 20.5}
{
\draw[thick, dotted] (\y,-0.4)--(\y,0.4);
}
\end{tikzpicture}
\]
Here, we draw dotted vertical line for the places in the erasure.

\label{page_explanationa_for_v}
We define the graph $G(\mathbf a)$ as the product $G(\mathbf a_1) \times \cdots \times G(\mathbf a_{\beta})$ of the graphs $G(\mathbf a_i)$ for $i=1,\dots,\beta$. We denote by $(v_1,\dots,v_{\beta})$ a vertex in the graph $G(\mathbf a)$. Here, $v_i$ is a vertex of the graph $G(\mathbf a_i)$ for $i= 1,\dots,\beta$.
For an edge connecting $(v_1,\dots,v_{i-1},v_{i},v_{i+1},\dots,v_{\beta})$ and $(v_1,\dots,v_{i-1},v_i',v_{i+1},\dots,v_{\beta})$ in the graph $G(\mathbf a)$, we label the value as that on the edge connecting $v_i$ and $v_{i}'$ plus the value $n_{i+1}+\dots+n_{\beta}$, where $\mathbf{a}_i$ is the composition of $n_i$ (see Example~\ref{example_Hn_product}).

\begin{example}\label{example_graph_G}
	\begin{enumerate}
		\item Suppose that $\mathbf a = (1,1,4)$. The corresponding element $w(\mathbf a)$ is $6 \textcolor{red}{|} 5 \textcolor{red}{|} 1234$. Then we need to erase two consecutive descents, so the graph $G(1,1,4)$ is drawn in $\R^2$. The set $V$ of vertices of the corresponding graph $G(\mathbf a)$ is given by
		\[
		V = \{ (z_1,z_2) \in \Z^2 \mid 3 \ge z_1 \ge z_2 \ge 0\}
		\]
		and the graph $G(1,1,4)$ is presented in Figure~\ref{figure_G114}.
		For example, the edge connecting $(1,1)$ and $(2,1)$ is decorated by $6 - 2 + (1-1) - 1 = 3$. Here, we have $m = 2$, $j = 1$, and $z_j = 1$.
		\item Suppose that $\mathbf a = (1,1,1,3)$. The corresponding element $w(\mathbf a)$ is $6 \textcolor{red}{|}  5 \textcolor{red}{|}  4 \textcolor{red}{|}  123$. Then we need to erase three consecutive descents, and the vertex set of the graph $G(1,1,1,3)$ is given by
		\[
		V = \{(z_1,z_2,z_3) \in \Z^3 \mid 2 \ge z_1 \ge z_2 \ge z_3 \ge 0\}.
		\]
		The graph $G(1,1,1,3)$ is given in Figure~\ref{figure_G1113}.
	\end{enumerate}
\end{example}
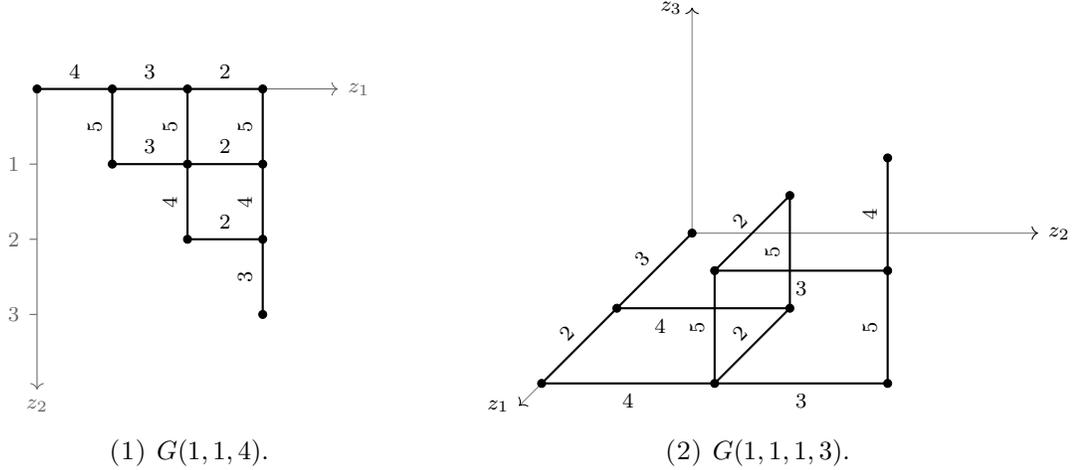
\begin{figure}
	\begin{subfigure}[b]{0.45\textwidth}
		\centering
		\begin{tikzpicture}[every node/.style={font=\scriptsize}]
		\draw[->, color=black!60!white] (0,0)--(4,0) node [right] {$z_1$};
		\draw[->, color=black!60!white] (0,0)--(0,-4) node [below] {$z_2$};
		\foreach \x in {1,2,3}{
			\draw[color = black!60!white] (0,-\x)--(-0.1,-\x) node[left] {$\x$};
		}
		
		\draw[fill=black] (0,0) circle (1.5pt)
		(1,0) circle (1.5pt)
		(2,0) circle (1.5pt)
		(3,0) circle (1.5pt)
		(1,-1) circle (1.5pt)
		(2,-1) circle (1.5pt)
		(3,-1) circle (1.5pt)
		(2,-2) circle (1.5pt)
		(3,-2) circle (1.5pt)
		(3,-3) circle (1.5pt);
		
		\draw[thick] (0,0)--(1,0) node[above, midway] {$4$};
		\draw[thick] (1,0)--(2,0) node[above, midway] {$3$};
		\draw[thick] (1,-1)--(2,-1) node[above, midway] {$3$};
		\draw[thick] (2,0)--(3,0) node[above, midway] {$2$};
		\draw[thick] (2,-1)--(3,-1) node[above, midway] {$2$};
		\draw[thick] (2,-2)--(3,-2) node[above, midway] {$2$};
		
		\draw[thick] (1,-1)--(1,0) node[midway,above, sloped] {$5$};
		\draw[thick] (2,-1)--(2,0) node[midway,above, sloped] {$5$};
		\draw[thick] (3,-1)--(3,0) node[midway,above, sloped] {$5$};
		\draw[thick] (2,-2)--(2,-1) node[midway,above, sloped] {$4$};
		\draw[thick] (3,-2)--(3,-1) node[midway,above, sloped] {$4$};
		\draw[thick] (3,-3)--(3,-2) node[midway,above, sloped] {$3$};
		
		\end{tikzpicture}
		\caption{$G(1,1,4)$.}
		\label{figure_G114}
	\end{subfigure}
	\begin{subfigure}[b]{0.45\textwidth}
		\centering
		\begin{tikzpicture}[x=2.3cm, y=1.5cm, z=-1cm,every node/.style={font=\scriptsize}]
		
		\draw [->, draw=gray] (0,0,0) -- (2,0,0) node [right] {$z_2$};
		\draw [->, draw=gray] (0,0,0) -- (0,2,0) node [left] {$z_3$};
		\draw [->, draw=gray] (0,0,0) -- (0,0,2.3) node [left] {$z_1$};	 	
		
		\draw[fill=black] (0,0,0) circle (1.5pt)
		(0,0,1) circle (1.5pt)
		(0,0,2) circle (1.5pt)
		(1,0,1) circle (1.5pt)
		(1,0,2) circle (1.5pt)
		(2,0,2) circle (1.5pt)
		(1,1,1)  circle (1.5pt)
		(1,1,2)  circle (1.5pt)
		(2,1,2) circle (1.5pt)
		(2,2,2)  circle (1.5pt);
		
		\draw[thick] (0,0,0)--(0,0,1) node[midway, above, sloped] {$3$};
		\draw[thick] (0,0,1)--(0,0,2) node[midway, above, sloped] {$2$};
		\draw[thick] (1,0,1)--(1,0,2) node[midway, above, sloped] {$2$};
		\draw[thick] (1,1,1)--(1,1,2) node[midway, above, sloped] {$2$};
		
		\draw[thick] (0,0,1)--(1,0,1) node[near start, below, sloped] {$4$};
		\draw[thick] (0,0,2)--(1,0,2) node[midway, below, sloped] {$4$};
		\draw[thick] (1,0,2)--(2,0,2) node[midway, below, sloped] {$3$};
		\draw[thick] (1,1,2)--(2,1,2) node[midway, below, sloped] {$3$};
		
		\draw[thick] (1,0,1)--(1,1,1) node[midway, above, sloped] {$5$};
		\draw[thick] (1,0,2)--(1,1,2) node[midway, above, sloped] {$5$};
		\draw[thick] (2,0,2)--(2,1,2) node[midway, above, sloped] {$5$};
		\draw[thick] (2,1,2)--(2,2,2) node[midway, above, sloped] {$4$};
		
		\end{tikzpicture}
		\caption{$G(1,1,1,3)$.}
		\label{figure_G1113}
	\end{subfigure}
	\caption{Graphs $G(\mathbf a)$ in Example~\ref{example_graph_G}.}
\end{figure}

For a decomposition $\mathbf a$ of $n$, the graph $G(\mathbf a)$ will be used to describe permutations $w$ having the decomposition $\mathbf a$, that is, $\mathbf a(w) = \mathbf a$ (see Proposition~\ref{lemma_wz}). Notice that the graph $G(\mathbf a)$ is defined by the product of the graphs $G(\mathbf a_{i})$ for $i=1,\dots,\beta$, where $(\mathbf a_1,\dots,\mathbf a_{\beta})$ is its admissible decomposition. Since each $\mathbf a_i$ is of the form in~~\eqref{eq_composition_a_erasing_conse}, we first consider the case when $\mathbf a$ is of the form in~\eqref{eq_composition_a_erasing_conse}.

For each vertex $z$ in the graph $G(\mathbf a)$, we associate an element $w_z \in \mathfrak{S}_n$ such that $\mathbf a(w_z) = \mathbf a$. (They will be used to describe $\widehat{\sigma}_w$ in Proposition~\ref{prop_hat_w_description_using_G}.)
For a vertex $z = (z_1,\dots,z_m)$, consider a shortest path from the origin to the vertex ~$z$. If the path is given by a sequence of edges labeled by $(i_1,\dots,i_{p})$, then we associate the permutation
\begin{equation}\label{eq_def_w_z}
w_z \colonequals s_{i_p} \cdots s_{i_1}  w(\mathbf a)
\end{equation}
to the vertex $z$. Here, we have
\begin{equation}\label{eq_sigma_wz}
\sigma_{w_z} = s_{i_p}\cdots s_{i_1} \sigma_{w(\mathbf a)}
\end{equation}
due to Proposition~\ref{prop:s_i action on S_n(a)} since $i_j$ and $i_{j+1}$ are not adjacent in $s_{i_{j-1}} \cdots s_{i_1} w(\mathbf a)$ for $j=1,\dots,p$.

We note that if there are two different paths connecting vertices
\[
(z_1,\dots,z_{j-1},z_j,z_{j+1},z_{j+2},\dots,z_m) \text{ and } (z_1,\dots,z_{j-1},z_j+1,z_{j+1}+1,z_{j+2},\dots,z_m),
\]
then permutations obtained from these two paths are the same. More precisely, suppose that there are four vertices
\[
\begin{split}
v_1 &= (z_1,\dots,z_{j-1}, z_{j}, z_{j+1},z_{j+2},\dots,z_m), \\
v_2 &= (z_1,\dots,z_{j-1},z_{j}+1,z_{j+1},z_{j+2},\dots,z_m),\\
v_3 &= (z_1,\dots,z_{j-1}, z_j, z_{j+1}+1,z_{j+2},\dots,z_m),\\
v_4 &= (z_1,\dots,z_{j-1}, z_{j}+1,z_{j+1}+1,z_{j+2},\dots,z_m)
\end{split}
\]
in the graph $G(\mathbf a)$. By the definition of the graph $G(\mathbf a)$, we have $z_j \gneq z_{j+1}$. Otherwise, one cannot have the vertex $v_3$. There are four edges whose labels are given as follows.
\[
\begin{split}
&(v_1,v_2) \text{ and } (v_3,v_4): n - m + (j-1) - z_j, \\
&(v_1,v_3) \text{ and } (v_2,v_4): n - m + j - z_{j+1}.
\end{split}
\]
Since $z_j > z_{j+1}$, we have
\[
|(n-m+j-z_{j+1}) -  ( n-m+(j-1) - z_j) | = | 1 + z_j - z_{j+1}|  > 1.
\]
This implies $s_{n-m+(j-1)-z_j} s_{n-m+j-z_{j+1}} = s_{n-m+j-z_{j+1}} s_{n-m+(j-1)-z_j}$ from the commutativity of two simple reflections $s_i$ and $s_j$ satisfying $|i-j|>1$. Therefore, the permutation~$w_z$ in~\eqref{eq_def_w_z} is well-defined.

For a given $z = (z_1,\dots,z_n)$, there are several paths connecting $z$ and the origin $(0,\dots,0)$. To reach the vertex $z$, we consider  a path passing vertices $(0,0,\dots,0), (1,0,\dots,0),\dots,(z_1,0,\dots,0)$.
Then we have
\[
\begin{split}
w_{(z_1,0,\dots,0)} &= s_{n-m-z_1+1} \cdots  s_{n-m-1} s_{n-m} w(\mathbf a) \\
&= n \textcolor{red}{|} n-1
\textcolor{red}{|} n-2
\textcolor{red}{|} \dots
\textcolor{red}{|} n-m+2
\textcolor{red}{|} n-m+1-z_1
\textcolor{red}{|} y ~ y+1 ~ \dots~  n-m ~ n-m+1
\textcolor{red}{|} \dots.
\end{split}
\]
Subsequently, we consider a path passing vertices $(z_1,0,\dots,0),(z_1,1,0,\dots,0),\dots,(z_1,z_2,0,\dots,0)$.
This path provides
\[
\begin{split}
w_{(z_1,z_2,0,\dots,0)}
&= s_{n-m-z_2+2} \cdots s_{n-m} s_{n-m+1} w_{(z_1,0,\dots,0)} \\
&= n \textcolor{red}{|} n-1
\textcolor{red}{|} n-2
\textcolor{red}{|} \dots
\textcolor{red}{|} n-m+2-z_2
\textcolor{red}{|} n-m+1-z_1
\textcolor{red}{|} y ~ y+1 ~ \dots~  n-m+1 ~ n-m+2
\textcolor{red}{|} \dots.
\end{split}
\]
Similarly, one can see that the values $w_z(1),\dots,w_z(m)$ are given as follows:
\[
w_z(1) \dots w_{z}(m) = n -z_m \textcolor{red}{|} n-1 -z_{m-1}
\textcolor{red}{|} \dots
\textcolor{red}{|} n-m+2-z_2
\textcolor{red}{|} n-m+1-z_1.
\]
We note that $w_z^{-1}(n-m+j-z_j) = d_{m-j+1}$ for $j=1,\dots,m$.
Now, by the construction of $w_z$, we have the following lemma:
\begin{lemma}\label{lemma_compare_wz_and_wz'}
	Let $\mathbf a $ be a composition of $n$ having $k+1$ parts of the form in~\eqref{eq_composition_a_erasing_conse}. We denote by $d_1<d_2<\dots<d_k$ the elements of $S(\mathbf a)$ and we set $d_{k+1}=n$. Then, for vertices $z = (z_1,\dots,z_m)$ and $z' = (z_1',\dots,z_m')$ connected by an edge in the graph $G(\mathbf a)$ satisfying $z_j = z_j'+1$, we have
\[
w_{z} = s_{i}w_{z'}, \quad w_{z'}^{-1}(i) = d_{m-j+1} ~~\text{ and }~~
w_{z'}^{-1}(i+1)
\in \{d_{m}+1,\dots,d_{m+1}\},
\]
where $i = n-m+(j-1)-z_j$. Indeed, $w_z(q) = w_{z'}(q)$ for $q \notin \{d_{m-j+1}\} \cup  \{d_{m}+1,\dots,d_{m+1}\}$.
\end{lemma}

\begin{example}\label{example_graph_G2}
	\begin{enumerate}
		\item Continuing with Example~\ref{example_graph_G}(1), we have the elements $w_z$ as follows:
		
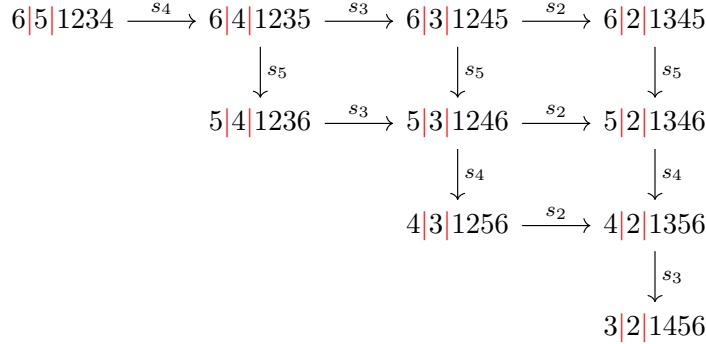
\[
		\begin{tikzcd}
		6 \textcolor{red}{|}5\textcolor{red}{|}1234 \arrow[r, "s_4"]
		& 6 \textcolor{red}{|} 4 \textcolor{red}{|} 1235 \arrow[r, "s_3"] \arrow[d, "s_5"]
		& 6 \textcolor{red}{|} 3 \textcolor{red}{|} 1245 \arrow[r, "s_2"]\arrow[d, "s_5"]
		& 6 \textcolor{red}{|} 2 \textcolor{red}{|} 1345\arrow[d, "s_5"] \\
		& 5 \textcolor{red}{|} 4 \textcolor{red}{|} 1236 \arrow[r, "s_3"]
		& 5 \textcolor{red}{|} 3 \textcolor{red}{|} 1246\arrow[r, "s_2"] \arrow[d, "s_4"]
		& 5 \textcolor{red}{|} 2 \textcolor{red}{|} 1346 \arrow[d, "s_4"]\\
		&& 4 \textcolor{red}{|} 3 \textcolor{red}{|} 1256\arrow[r, "s_2"]
		& 4 \textcolor{red}{|} 2 \textcolor{red}{|} 1356 \arrow[d, "s_3"]\\
		&&& 3 \textcolor{red}{|} 2 \textcolor{red}{|} 1456
		\end{tikzcd}
		\]
		 In this example, $m=2$, and $d_1=1,d_2=2,d_3=6$. One may see that along the edge on the $z_1$-direction, the first value of the permutation does not change. Indeed, when $j=1$, $d_{m-j+1} = d_{2-1+1} = d_{2} = 2$, and $d_{m+1} = d_3 = 6$.
For example, for $z = (z_1,0)$ (which are on the first row in the diagram), we have $w_z(1) = 6$. On the other hand, for $z = (3,z_2)$ (which are on the fourth column in the diagram), we have $w_z(2) = 2$.  Indeed, when $j=2$, we obtain $d_{m-j+1} = d_{2-2+1} = d_1 = 1$.
		\item Continuing with Example~\ref{example_graph_G}(2), we have the elements $w_z$ as follows:
		
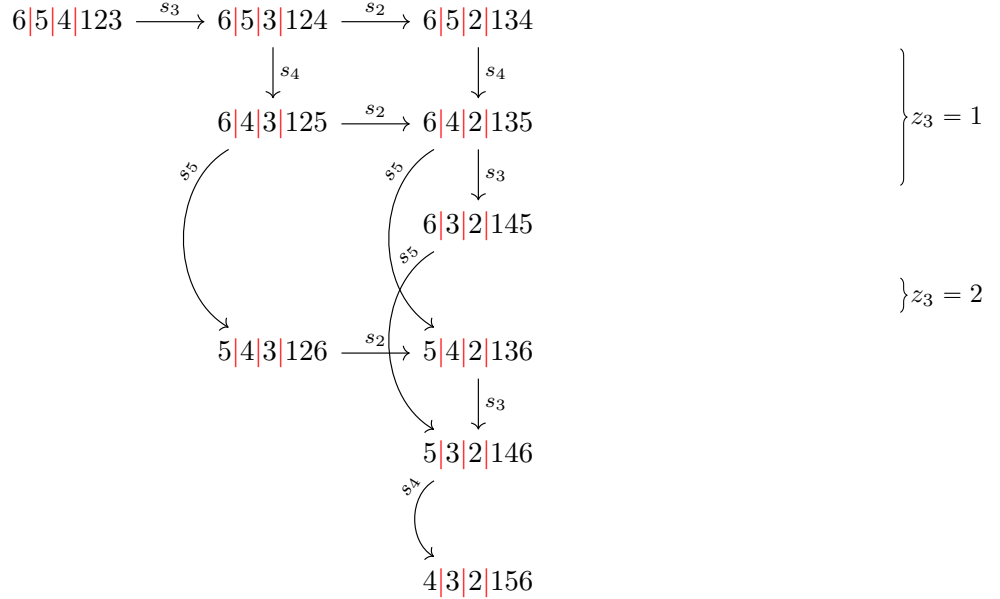
\[
		\begin{tikzcd}
		6 \textcolor{red}{|} 5 \textcolor{red}{|} 4 \textcolor{red}{|} 123 \arrow[r, "s_3"]
		& 6 \textcolor{red}{|} 5 \textcolor{red}{|} 3 \textcolor{red}{|} 124 \arrow [r, "s_2"] \arrow[d, "s_4"]
		& 6 \textcolor{red}{|} 5 \textcolor{red}{|} 2 \textcolor{red}{|} 134 \arrow[d, "s_4"] \tikzmark{bracebegin1} \\
		& 6 \textcolor{red}{|} 4 \textcolor{red}{|} 3 \textcolor{red}{|} 125  \arrow [r, "s_2"] \arrow[dd, bend right = 60, "s_5", pos = 0.2, sloped]
		& 6 \textcolor{red}{|} 4 \textcolor{red}{|} 2 \textcolor{red}{|} 135  \arrow[d, "s_3"] \arrow[dd, bend right = 60, "s_5", pos = 0.2, sloped]\\
		& & 6 \textcolor{red}{|} 3 \textcolor{red}{|} 2 \textcolor{red}{|} 145 \arrow[dd, bend right = 60, "s_5", pos = 0.1, sloped] \tikzmark{braceend1} \\ [2ex]
		& 5 \textcolor{red}{|} 4 \textcolor{red}{|} 3 \textcolor{red}{|} 126 \arrow [r, "s_2"]
		&  5 \textcolor{red}{|} 4 \textcolor{red}{|} 2 \textcolor{red}{|} 136  \arrow[d, "s_3"]\tikzmark{bracebegin2} \\
		&  & 5 \textcolor{red}{|} 3 \textcolor{red}{|} 2 \textcolor{red}{|} 146 \arrow[d, bend right = 60, "s_4", pos = 0.2, sloped] \tikzmark{braceend2}\\ [2ex]
		&  & 4 \textcolor{red}{|} 3 \textcolor{red}{|} 2 \textcolor{red}{|} 156 \tikzmark{bracebegin3}
		\end{tikzcd}
		\]
		\begin{tikzpicture}[overlay,remember picture]
		\draw[decorate,decoration={brace}] ( $ (pic cs:bracebegin1) +(0.5, 9pt)  $ ) -- ( $ (pic cs:braceend1) -(-0.5, 4pt) $ ) node[midway, right] {\small $z_3 = 0$};
		\draw[decorate,decoration={brace}] ( $ (pic cs:bracebegin2) +(0.5, 9pt)  $ ) -- ( $ (pic cs:braceend2) -(-0.5, 4pt) $ ) node[midway, right] {\small $z_3 = 1$};
		\draw[decorate,decoration={brace}] ( $ (pic cs:bracebegin3) +(0.5, 9pt)  $ ) -- ( $ (pic cs:bracebegin3) -(-0.5, 4pt) $ ) node[midway, right] {\small $z_3 = 2$};
		\end{tikzpicture}
	\end{enumerate}
\end{example}

\begin{example}\label{example_GVa_n42}
Suppose that $n =4$ and $k=2$, that is, we consider the compositions with $3$ parts of $n = 4$; $(1,1,2)$, $(1,2,1)$, and $(2,1,1)$.
For each composition $\mathbf{a}$, we provide the erasure $\widehat{\mathbf{a}}$ and $w(\mathbf{a})$.
\[
		\begin{array} {cccc}
			{\bf a}  & \widehat{\bf a} & w(\mathbf{a}) & \mathfrak S_{w(\mathbf{a})}^{\circ} \\ \hline
			(1,1,2) & (4) & 4\textcolor{red}{|}3\textcolor{red}{|}12 &\mathfrak S_{\{4,3,1,2\}}/\mathfrak{S}_{\{1,2\}} \\
		(1,2,1) & (3,1) & 4\textcolor{red}{|}23\textcolor{red}{|}1 &\mathfrak S_{\{4,2,3\}}/\mathfrak S_{\{2,3\}} \\
			(2,1,1) & (2,2)  & 34\textcolor{red}{|}2\textcolor{red}{|}1 &
\mathfrak S_{\{2,1\}}
		\end{array}
\]

	The lexicographic order defines a total order:
	\[
	(1,1,2) < (1,2,1) < (2,1,1).
	\]
	The corresponding graphs $G(\mathbf a)$ are depicted in Figure~\ref{figure_G_n4}.
	For each composition $\mathbf a$, we compute the elements $w_z$ for $z \in V(G(\mathbf a))$. Moreover, we present $\hats_{w(\mathbf a)}$ for $w \in \La$.
	\begin{itemize}
		\item $\mathbf a = (1,1,2)$.
		\[
		\begin{tikzcd}
		4 \textcolor{red}{|}  3 \textcolor{red}{|}  12 \arrow[r, "s_2"] & 4 \textcolor{red}{|}  2 \textcolor{red}{|}  13 \arrow[d, "s_3"]\\
		& 3 \textcolor{red}{|}  2 \textcolor{red}{|}  14
		\end{tikzcd}
		\]
		To compute the class $\widehat{\sigma}_{4312}$, we consider the following diagram:
		\[
		\begin{tikzcd}
		\sigma_{4312} \arrow[r, "s_2"] \arrow[d, "s_3"] & \sigma_{4213}  \arrow[r, "s_1"] \arrow[d, "s_3"] & \sigma_{4213} + \sigma_{4231} - \sigma_{4132} \arrow[d, "s_3"]\\
		\sigma_{4312} \arrow[d,"s_2"] & \sigma_{3214} \arrow[r, "s_1"] \arrow[d, "s_2"] & \sigma_{3214} + \sigma_{3241} - \sigma_{3142} \arrow[d, "s_2"] \\
		\sigma_{4213} \arrow[r, "s_3"] \arrow[d, "s_1"] & \sigma_{3214} \arrow[d, "s_1"] & \sigma_{3214} + \sigma_{3241} + \sigma_{3421} - \sigma_{2431} - \sigma_{2143} \arrow[d, "s_1"] \\
		\sigma_{4213} + \sigma_{4231} - \sigma_{4132} \arrow[r, "s_3"] & \sigma_{3214} + \sigma_{3241} - \sigma_{3142} \arrow[r,"s_2"]  & \sigma_{3214} + \sigma_{3241} + \sigma_{3421} - \sigma_{2143} - \sigma_{2431}
		\end{tikzcd}
		\]
		Accordingly, we obtain  		
		\[
		\begin{split}
		\hats_{4312} &= 2\sigma_{4312} + 4\sigma_{4213} + 6\sigma_{3214}  \\
		&\quad  + \underbrace{2(\sigma_{4231} - \sigma_{4132}) + 4\sigma_{3241} - 2 \sigma_{3142} - 2\sigma_{2143}}_{\mathbf a(v) = (1,2,1) > (1,1,2)} + \underbrace{2(\sigma_{3421} - \sigma_{2431})}_{\mathbf a(v) = (2,1,1) > (1,1,2)}.
		\end{split}
		\]
		\item $\mathbf a = (1,2,1)$.
		\[
		\begin{tikzcd}
		4 \textcolor{red}{|} 23 \textcolor{red}{|} 1 \arrow[r, "s_3"] & 3 \textcolor{red}{|} 24 \textcolor{red}{|} 1
		\end{tikzcd}
		\]
		\[
		\hats_{4231} = \sigma_{4231} + 2 \sigma_{3241} + \underbrace{(\sigma_{3421} - \sigma_{2431})}_{\mathbf a(v) = (2,1,1) > (1,2,1)}
		\]
		\item $\mathbf a = (2,1,1)$.
		\[
		\hats_{3421} = 2 \sigma_{3421}.
		\]
	\end{itemize}
	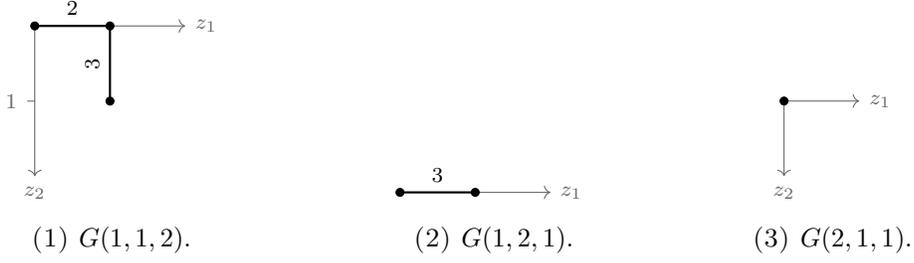
\begin{figure}\centering
		\begin{subfigure}[b]{0.3\textwidth}
			\centering
			\begin{tikzpicture}[every node/.style={font=\scriptsize}]
			\draw[->, color=black!60!white] (0,0)--(2,0) node [right] {$z_1$};
			\draw[->, color=black!60!white] (0,0)--(0,-2) node [below] {$z_2$};
			\foreach \x in {1}{
				\draw[color = black!60!white] (0,-\x)--(-0.1,-\x) node[left] {$\x$};
			}
			
			\draw[fill=black] (0,0) circle (1.5pt)
			(1,0) circle (1.5pt)
			(1,-1) circle (1.5pt);
			
			\draw[thick] (0,0)--(1,0) node[above, midway] {$2$};
			
			\draw[thick] (1,-1)--(1,0) node[midway,above, sloped] {$3$};
			
			\end{tikzpicture}
			\caption{$G(1,1,2)$.}
		\end{subfigure}
		\begin{subfigure}[b]{0.3\textwidth}
			\centering
			\begin{tikzpicture}[every node/.style={font=\scriptsize}]
			\draw[->, color=black!60!white] (0,0)--(2,0) node [right] {$z_1$};
			\draw[fill=black](0,0) circle (1.5pt)
			(1,0) circle (1.5pt);
			\draw[thick] (0,0)--(1,0) node[midway, above] {$3$};
			\end{tikzpicture}		
			\caption{$G(1,2,1)$.}
		\end{subfigure}
		\begin{subfigure}[b]{0.23\textwidth}
			\centering
			\begin{tikzpicture}[every node/.style={font=\scriptsize}]
			\draw[->, color=black!60!white] (0,0)--(1,0) node [right] {$z_1$};
			\draw[->, color=black!60!white] (0,0)--(0,-1) node [below] {$z_2$};
			\draw[fill=black](0,0) circle (1.5pt);
			\end{tikzpicture}			
			\caption{$G(2,1,1)$.}
		\end{subfigure}
		\caption{Graphs $G(\mathbf a)$ for $n = 4$ and $k = 2$.}
		\label{figure_G_n4}
	\end{figure}
\end{example}

As one may see in Example~\ref{example_GVa_n42}, the permutations $w_z$ coming from the graph $G(\mathbf a)$ capture the terms in $\hats_w$ whose compositions are $\mathbf a$, and the other terms in the sum $\hats_w$ have the composition greater than~$\mathbf a$ (with respect to the lexicographic order).
We show that this observation also holds in general in Proposition~\ref{prop_hat_w_description_using_G}.

Let $\mathscr M_w$ be the complete set of minimal length coset representatives for the quotient $\mathfrak{S}_w^{\circ} = \mathfrak{S}_w/(\mathfrak{S}_{J_1} \times \cdots \times \mathfrak{S}_{J_{k+1}})$, where $J_s = \{ w(d_{s-1}+1),\dots,w(d_s)\}$,  $D(w) = \{d_1< \dots < d_k\}$, and $d_{k+1} = n$ as before.
Let $(\mathbf a_1,\dots, \mathbf a_{\beta})$ be an admissible decomposition of $\mathbf a$, where $\mathbf a_i$ is a composition of $n_i$ and $\sum_{i=1}^{\beta} n_i = n$.
Then the set $\mathscr M_w$ is given by
\begin{equation}\label{eq_Mw_product}
\mathscr M_w = \{w_1 \cdots w_{\beta-1} w_{\beta} \mid w_i \in M_{w(\mathbf a_i)} \subset \mathfrak{S}_{[\sum_{j=1}^{i-1} n_j +1, \sum_{j=1}^i n_j]} \}.
\end{equation}
Applying the result in~\cite{Stu} appropriately, we obtain the following lemma.
\begin{lemma}[{cf.~\cite{Stu}}]\label{lemma_Mw}
Let $w = w(\mathbf a)$. For $\mathbf a = (1,\dots,1,a_{m+1})$, we have
\[
\mathscr M_w 
= \{ v_m v_{m-1} \cdots v_1 
\mid
v_j = s_{i_j} s_{i_j+1} \cdots s_{n-m+j-1} \quad \text{ for }1 \leq i_j \leq n-m+j,~~1 \leq j \leq m\}.
\]
Here, we set $v_j = e$ if $i_j = n-m+j$.
\end{lemma}

\begin{proposition}\label{lemma_wz}
Let $\mathbf a$ be a composition and $w = w(\mathbf a)$. Then
the graph $G(\mathbf a)$  captures the permutations $v \in \mathscr M_w$ which preserve the composition, and moreover, this is a complete set of the permutations having the same composition, that is,
\[
\{ w_z \mid z \in V(G(\mathbf a)) \}
= \{ vw \mid v \in \mathscr M_w \text{ and } \mathbf{a}(vw) = \mathbf{a}(w) \}
= \{ u \in \mathfrak{S}_n \mid \mathbf{a}(u) = \mathbf{a}(w)\}.
\]
\end{proposition}
\begin{proof}
Let $(\mathbf a_1,\dots, \mathbf a_{\beta})$ be an admissible decomposition of $\mathbf a$.
In this case, the graph $G(\mathbf a)$ is defined by the product $G(\mathbf a_1) \times \cdots \times G(\mathbf a_{\beta})$ of graphs.
By~\eqref{eq_Mw_product}, without loss of generality, we may assume that  $\mathbf{a} = (\underbrace{1,\dots,1}_{m},a_{m+1}) = (\underbrace{1,\dots,1}_{m},n-m)$. Then we have $\widehat{\bf a} = (n)$ and
\[
w = w(\mathbf a) = n \textcolor{red}{|} n-1 \textcolor{red}{|} \cdots \textcolor{red}{|} n-m+1 \textcolor{red}{|} 1 \ 2 \ \cdots \ n-m.
\]

We first consider the first equality in the statement.
According to \cite[Exercise~1.12]{BB05Combinatorics}, we have
\begin{equation}\label{eq_DR_vw}
D_R(vw) = D_R(w) \triangle ((w^{-1} T_R(v) w) \cap \mathscr S).
\end{equation}
Here, $\mathscr S = \{ s_i \mid i \in [n-1]\}$, $D_R(w) = \{ s_i \in \mathscr S \mid w(i) > w(i+1)\}$, and $T_R(w) = \{ s_{i,j} \mid 1 \leq i < j \leq n,  w(i) > w(j)\}$. Moreover, $A \triangle B \colonequals (A \cup B) \setminus (A \cap B)$.
Since $D_R(w) = \{ s_i \mid i \in D(w)\}$, for $v \in \mathfrak{S}_n$, we have
\[
\begin{split}
\mathbf{a}(vw) = \mathbf{a}(w) &\iff D(vw) = D(w) \\
&\iff D_R(vw) = D_R(w) \\
&\iff D_R(w) \triangle ((w^{-1} T_R(v) w) \cap \mathscr S) = D_R(w) \\
&\iff  (w^{-1} T_R(v) w) \cap \mathscr S = \emptyset.
\end{split}
\]
We consider elements $v \in \mathfrak{S}_n$ satisfying $ (w^{-1} T_R(v) w) \cap \mathscr S = \emptyset$.

We note that
\begin{equation}\label{eq_w_inverse}
w^{-1} = m+1 \ m+2 \ \cdots \ n \ m \ m-1 \ \cdots \ 1.
\end{equation}
Since $w^{-1} s_{i,j} w = s_{w^{-1}(i), w^{-1}(j)}$, using the description~\eqref{eq_w_inverse}, we obtain
\[
\begin{split}
&w^{-1} s_{n-1,n} w = s_1, \quad w^{-1} s_{n-2,n-1} w = s_2, \quad \ldots, \quad
w^{-1} s_{n-m+1,n-m+2} w = s_{m-1}, \\
&w^{-1} s_{1,n-m+1} w = s_{m}, \quad
w^{-1} s_{1,2} w = s_{m+1}, \quad \ldots, \quad
w^{-1} s_{n-m-1,n-m} w = s_{n-1}.
\end{split}
\]
Accordingly, $v \in \mathfrak{S}_n$ satisfies $ (w^{-1} T_R(v) w) \cap \mathscr S = \emptyset$ if and only if
\begin{equation}\label{eq_v_satisfies}
\begin{split}
&v(n-m+1) < v(n-m+2) < \dots < v(n-1) < v(n); \\
&v(1) < v(2) < \dots < v(n-m); \\
&v(1) < v(n-m+1).
\end{split}
\end{equation}

On the other hand, by Lemma~\ref{lemma_Mw}, we have
\[
\mathscr M_w
= \{ v_m v_{m-1} \cdots v_1 
\mid
v_j = s_{i_j} s_{i_j+1} \cdots s_{n-m+j-1} \quad \text{ for }1 \leq i_j \leq n-m+j,~~1 \leq j \leq m\}.
\]
We claim that $v = v_m v_{m-1} \cdots v_1$ satisfies~\eqref{eq_v_satisfies} if and only if
\begin{equation}\label{eq_index_i_satisfies}
1<i_1 < i_2< \dots < i_m.
\end{equation}
Indeed, if $1 < i_1 < i_2< \dots < i_m$, then $v = ([n] \setminus \{i_1,i_2,\dots,i_m\})\!\!\uparrow \ i_1 \ i_2 \ \cdots \ i_m$, which satisfies~\eqref{eq_v_satisfies}.
Moreover,  by varying the value $i_m$ from $n$ to $1$, the number of indices satisfying~\eqref{eq_index_i_satisfies} is
\begin{equation}\label{eq_counting}
{n-2 \choose m-1} + {n-3 \choose m-1} + \dots + {m-1 \choose m-1}.
\end{equation}
One can see that the number~\eqref{eq_counting} is the same as the number of permutations $v$ satisfying~\eqref{eq_v_satisfies} by varying the value $v(n-m+1)$ from $2$ to $m-1$.
This proves that an element $v  = v_m v_{m-1} \cdots v_1$ in $\mathscr M_w$ that is determined by the indices $i_1,\dots,i_m$ satisfies the condition~\eqref{eq_index_i_satisfies} if and only if $\mathbf{a}(vw) = \mathbf{a}(w)$.

Now we consider the relation between the graph $G(\mathbf a)$ and minimal length coset representatives preserving the composition.
For each $v = v_m v_{m-1} \cdots v_1 \in \mathscr M_w$ with $v_j = s_{i_j} s_{i_j +1} \cdots s_{n-m+j-1}$, we associate $z = (z_1,\dots,z_m)$ by setting $z_j = n-m+j - i_j$.
Then,  we have $z_j - z_{j+1} = (n-m+j - i_j) - (n-m+j+1 - i_{j+1})
= i_{j+1} - i_j -1$, thus, we obtain
\begin{equation}\label{eq_z_j_and_i_j}
z_j \geq z_{j+1} \iff i_j < i_{j+1}.
\end{equation}
Moreover, we have
\begin{equation}\label{eq_z_j_and_i_j2}
n-m -1 \geq z_1 = n-m +1 - i_1 \iff 1 < i_1.
\end{equation}
Combining~\eqref{eq_z_j_and_i_j} and~\eqref{eq_z_j_and_i_j2}, the condition~\eqref{eq_index_i_satisfies} is equivalent to
\[
a_{m+1}-1 = n-m-1 \geq z_1 \geq z_2 \geq \cdots \geq z_m \geq 0,
\]
which defines the set $V(G(\mathbf a))$ of vertices of the graph $G(\mathbf a)$. Indeed, this association provides $w_z = v_m v_{m-1} \cdots v_1 w$, which proves the the first equality in the statement.

To obtain the second equality, we first notice that
\[
\{ vw \mid v \in \mathscr M_w \text{ and } \mathbf{a}(vw) = \mathbf{a}(w)\} \subset \{ u \in \mathfrak{S}_n \mid \mathbf{a}(u) = \mathbf{a}(w)\}.
\]
For a permutation $u \in \mathfrak{S}_n$ satisfying $\mathbf{a}(u) = \mathbf{a}(w)$, we know that $D(u) = \{1,2,\dots,m\}$. Accordingly, by varying $u(m)$ from $2$ to $m-1$, we obtain that the number of elements in the set on the right hand side is~\eqref{eq_counting}. Hence the result follows.
\end{proof}

In Example~\ref{example_GVa_n42}, the element $\hats_w$ is written as a linear combination of classes $\sigma_v$ satisfying $\mathbf{a}(v) \geq \mathbf {a}(w)$. This holds in general as follows:
\begin{proposition}\label{prop_hat_w_description_using_G}
	Let $w$ be an element in $\La$ such that $\mathbf a(w) = \mathbf a$.
	Then the element $\hats_{w}$ is written as
	\[
	\hats_{w} =
	\sum_{z \in V(G(\mathbf a))} c_z \sigma_{w_z} + \sum_{\mathbf{a}(v) > \mathbf a} k_v \sigma_{v} \quad \text{ for } c_z, k_v \in \R \text{ and }c_z > 0.
	\]
	Here, we consider the lexicographic order on the set of compositions.
\end{proposition}
\begin{proof}
Let $\mathbf a$ be a composition and $w = w(\mathbf a)$.
According to Definition~\ref{def:generator}, $\hats_w = \sum_{\bar{v}  \in \mathfrak{S}_w^{\circ}} v \cdot \sigma_w$.
Without loss of generality, we may assume $\mathbf a= (\underbrace{1,\dots,1}_{m},a_{m+1}) = (\underbrace{1,\dots,1}_m,n-m)$.
We note that $\mathbf a$ is the minimum among the compositions of $n$ having $m+1$ parts with respect to the lexicographic order. Accordingly, the element $\hats_w$ is written by a linear combination of classes $\sigma_v$ satisfying $\mathbf{a}(v) \geq \mathbf {a}$. Since we proved in Proposition~\ref{lemma_wz} that whenever $u \in \mathfrak{S}_n$ satisfies $\mathbf{a}(u)= \mathbf a$ then $u = w_z$ for some $z \in V(G(\mathbf a))$, we obtain
	\[
	\hats_{w} =
	\sum_{z \in V(G(\mathbf a))} c_z \sigma_{w_z} + \sum_{\mathbf{a}(v) > \mathbf a} k_v \sigma_{v}.
	\]
For each $z \in V(G(\mathbf a))$, there exists $\bar{v} \in \mathfrak{S}_w^{\circ}$ for $v \in \mathscr M_w$ such that $D(vw) = D(w)$ and $vw = w_z$ by Proposition~\ref{lemma_wz}. Moreover, by~\eqref{eq_sigma_wz}, we obtain $v \sigma_w = \sigma_{vw} = \sigma_{w_z}$. On the other hand, whenever we need to consider $s_i w_z$ with $w_z \to s_i w_z$, we have
\[
s_i \cdot \sigma_{w_z} = (t_{i+1}-t_i) \sigma_{s_iw_z} + \sigma_{w_z} +
(\sigma_{w_z}^{(i)} - \sigma_{w_z}) - s_i \cdot (\sigma_{w_z}^{(i)} - \sigma_{w_z})
\]
by Proposition~\ref{prop_si_action_on_sigma_i}.
Here,  recall that $\sigma_{w_z}^{(i)}$ is a sum of certain classes and the coefficient of the term~$\sigma_{w_z}$ is $1$.
Indeed, the coefficient of the term $\sigma_{w_z}$ in $\sigma_{w_z}^{(i)} - \sigma_{w_z}$ is zero. Moreover, the other terms $\sigma_v$ in $\sigma_{w_z}^{(i)}$ are given by permutations $v$ satisfying $\mathbf{a}(v) > \mathbf a$ because $\mathbf a = (1,\dots,1,n-m)$. Since such permutations~$v$ satisfy $v \dasharrow s_iv$, there does not exist a class $\sigma_u$ with $\mathbf{a}(u) = \mathbf a$ in $s_i \cdot (\sigma_{w_z}^{(i)} - \sigma_{w_z})$.
This proves the positivity of~$c_z$.
\end{proof}

Before providing a proof of Theorem~\ref{thm:erasing conjecture}, we prepare one lemma.
\begin{lemma}\label{lemma_main_thm}
Let $w \in \La$ and $\mathbf a = \mathbf a(w)$ be a composition of $n$ consisting of $k+1$ parts. Then, for each $z \in V(G(\mathbf a))$, the element $\sigma_{w_z}$ is contained in $\sum_{u \in \La} M(u)$.
\end{lemma}

\begin{proof}
	Recall that we define the total order on the set of compositions lexicographically.
	We prove the statement using the induction with respect to this total order.
	For simplicity, we denote by $M$ the  sum $\sum_{u \in \La} M(u)$.
	For a given $k$, the maximal element $\mathbf a$ among the compositions having $k+1$ parts is
	\[
	\mathbf a = (n-k,1,\dots,1).
	\]
	In this case, the graph $G(\mathbf a)$ has one vertex $z_0 = (0,\dots,0)$. Therefore, by Proposition~\ref{prop_hat_w_description_using_G}, we obtain
	\[
	\hats_{w}  = c \sigma_{w} = c \sigma_{w_{z_0}}
	\]
	for some positive integer $c$. This proves the claim for this case.
	
	Now we suppose that the claim is true for any composition which is greater than $\mathbf a$.
	Recall the expression of $\hats_w \in M(w) \subset M$ for $w = w(\mathbf a)$ in Proposition~\ref{prop_hat_w_description_using_G}:
	\[
	\hats_{w} =
	\sum_{z \in V(G(\mathbf a))} c_z \sigma_{w_z} + \sum_{\mathbf{a}(v) > \mathbf a} k_v \sigma_{v}, \quad c_z, k_v \in \R \text{ and }c_z > 0.
	\]
	By the induction argument, we have $\sigma_v$ satisfying $\mathbf a(v) > \mathbf a$ is in the sum $M$, and hence
\[
\sum_{z \in V(G(\mathbf a))} c_z \sigma_{w_z} \in M.
\]
	To prove the lemma, it is enough to show that  for all $z \in V(G(\mathbf a)) \setminus \{\mathbf 0\}$, there exists a linear combination
	\begin{equation}\label{eq_Lem6.2_linear_combination}
	\sigma_{w_{z'}} - \sigma_{w_z} \in M
	\end{equation}
	for some $z' \in V(G(\mathbf a))$ such that the distance between $z'$ and the origin is shorter than that of~$z$.
	Indeed, since the coefficients $c_z$ are all positive the existence of linear combinations of the form~\eqref{eq_Lem6.2_linear_combination} implies
	\[
	\text{span}_{\C} \left(\left\{ \sum_{z \in V(G(\mathbf a))} c_z \sigma_{w_z} \right\} \cup \left\{\sigma_{w_{z'}} - \sigma_{w_z} \mid z \in V(G(\mathbf a))\setminus \{\mathbf 0\} \right\}  \right)
	= \text{span}_{\C}\left\{ \sigma_{w_z} \mid z \in V(G(\mathbf a)) \right\} \subset M.
	\]

	We first provide the linear combination~\eqref{eq_Lem6.2_linear_combination} when the erasing occurs consecutively, i.e., $\mathbf a$ has the form in~\eqref{eq_composition_a_erasing_conse}.
	We consider elements $z \in V(G(\mathbf a))$ such that $z' \colonequals z - (1,0,\dots,0)$ is also contained in $V(G(\mathbf a))$, that is, $z \in \{ (z_1,\dots,z_m) \in V(G(\mathbf a)) \mid z_1 \gneq z_2\}$. Then they are connected by an edge, and we denote the label by $i$  on the edge. Because of the construction of $w_z$, we have
	\[
	w_{z} = s_{i} w_{z'}
	\]
	and moreover, $w_{z'}^{-1}(i+1) < w_{z'}^{-1}(i)$. Note that such permutations $w_z$ and $w_z'$ have different values only between descents $d_{m-1}$ and $d_{m+1}$ by Lemma~\ref{lemma_compare_wz_and_wz'}. Furthermore, $w_z(d_{m}) = i+1$. We denote $j = w_z^{-1}(i)$. Since the descents sets of $w_{z'}$ and $w_z$ are the same, $d_{m} + 1 < j$.   Now we define a permutation $u_z$ by
	\[
	u_z = w_{z'} s_{d_{m}} s_{d_{m}+1} \dots s_{j-2}.
	\]
	In other words, we move the descent $d_{m}$ to the right until its location is just before $i$, i.e., two numbers $i+1$ and $i$ are consecutive in the one-line notation of $u_z$. Then, by applying Proposition~\ref{prop_si_action_on_sigma_i}, we obtain
	\[
	s_i \sigma_{u_z} = \sigma_{w_z} - \sigma_{w_{z'}} + (t_{i+1} - t_i) \sigma_{s_i u_z} + \sum_{\mathbf a(v) > \mathbf a} q_v \sigma_v.
	\]
	We notice that  $\mathbf{a}(u_z) > \mathbf a$. Accordingly, by the induction argument, both $\sigma_{u_z}$ and $\sum_{\mathbf a(v) > \mathbf a} q_v \sigma_v$ are in the  sum $M$. Because $M$ is invariant under the action of $\mathfrak{S}_n$, we prove that the linear combination $\sigma_{w_{z'}} - \sigma_{w_z}$ is contained in $M$.
	
	Now we consider elements $z = (z_1,\dots,z_m) \in V(G(\mathbf a))$ such that $z_1 = z_2 \gneq z_3$.
	For this case, we note that two elements $z' \colonequals z - (1,1,0,\dots,0)$ and $\bar{z} \colonequals z - (0,1,0,\dots,0)$ are in $V(G(\mathbf a))$. By the construction of the graph $G(\mathbf a)$, elements $z$ and $\bar{z}$ are connected by an edge, and we denote by~$i$ the label on the edge, that is,
	$n-m+(2-1)-z_2 = i$.
	Since $z_1 = z_2$ by the assumption, we have $n - m + (1-1)-z_1 = i-1$, that is, the label on the edge connecting $z'$ and $\bar{z}$ is $i-1$. See the left hand side of Figure~\ref{figure_wz_proof_of_Lemma6.6}. On the other hand, $z_1 = z_2$ implies that $z' - (0,1,\dots,0)$ is not in $G(\mathbf a)$, i.e., $D(s_i z') \neq D(z')$. Moreover, we have $D(z') = D(s_{i-1}z') = D(s_i s_{i-1}z')$. Therefore, the permutations $w_z, w_{z'}, w_{\bar{z}}$ have the form in Figure~\ref{figure_wz_proof_of_Lemma6.6}.
	\begin{figure}
		\[
		\begin{array}{cc}
		\raisebox{-1cm}{
			\begin{tikzpicture}[every node/.style={font=\scriptsize}]
			\draw[fill=black] (0,0) circle (1.5pt) node[above left] {$z'$};
			\draw[fill=black] (1,0) circle (1.5pt) node[above right] {$\bar{z}$};
			\draw[fill=black] (1,-1) circle (1.5pt) node[below] {$z$};
			\draw[thick] (0,0)--(1,0) node[midway, above] {$i-1$};
			\draw[thick] (1,0)--(1,-1) node[midway, above, sloped] {$i$};	
			\end{tikzpicture}}
		& \hspace{1cm}
		\begin{tikzcd}
		w_{z'}=\dots \textcolor{red}{|}  i+1 \textcolor{red}{|}  i \textcolor{red}{|}  \dots i-1 \dots \arrow[r, "s_{i-1}"]&
		w_{\bar{z}} = \dots \textcolor{red}{|}  i+1 \textcolor{red}{|}  i-1 \textcolor{red}{|}  \dots i \dots \arrow[d, "s_i"] \\
		& w_{z} = \dots \textcolor{red}{|}  i \textcolor{red}{|}  i-1 \textcolor{red}{|}  \dots i+1 \dots
		\end{tikzcd}
		\end{array}
		\]
		\caption{Some permutations $w_z$ when $z_1 = z_2 \gneq z_3$.}
		\label{figure_wz_proof_of_Lemma6.6}
	\end{figure}
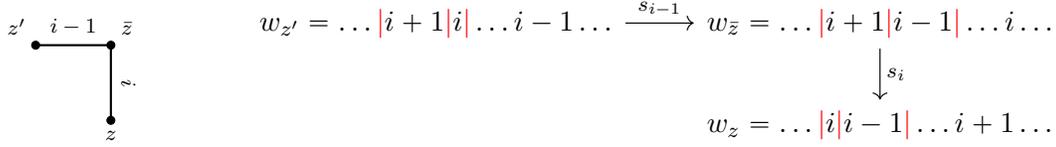

	By the previous observation, we have a linear combination
	\[
	\sigma_{w_{z'}} - \sigma_{w_{\bar z}}  \in M.
	\]
	Then, by applying $s_i$ on the relation, we obtain the desired linear combination
	\[
	\sigma_{w_{z'}} - \sigma_{w_z} \in M.
	\]
	Here, $s_i \cdot \sigma_{w_{z'}} = \sigma_{w_{z'}}$ since $\{w_{z'}^{-1}(i+1)-1, w_{z'}^{-1}(i+1), w_{z'}^{-1}(i)\} \subset D(w_{z'}) \cup \{0\}$ and by Corollary~\ref{cor_sigma_invariant_si_singular}.
	Continuing this process, we obtain the desired linear combinations when the erasing occurs consecutively.
	
	Now we suppose that the graph $G(\mathbf a)$ is the product $G(\mathbf a_1) \times \cdots \times G(\mathbf a_{\beta})$ of the graphs $G(\mathbf a_i)$, where $(\mathbf a_1,\dots,\mathbf a_{\beta})$ is an admissible decomposition of $\mathbf a$.
	We need to show that for any vertex $v = (v_1,\dots,v_{\beta}) \in V(G(\mathbf a))\setminus \{\mathbf 0\}$, there exists $v' = (v_1',\dots,v_{\beta}') \in V(G(\mathbf a))$ such that the distance between $v'$ and the origin is shorter than that of $v$ and
	\[
	\sigma_{w_{v'}} - \sigma_{w_{v}}   \in M.
	\]
	Among indices $1,\dots,\beta$, choose the minimal index $x$ and the maximal index $y$ such that $v_{x} \neq 0$ and $v_y \neq 0$.  Consider the vertex ${v_x}$ in the graph $G(\mathbf a_x)$. Then by the previous argument, we have a vertex $v_{x}'$ in the graph $G(\mathbf a_x)$ such that
	\[
	\sigma_{w_{(0,\dots,0,v_x',0,\dots,0)}}  - \sigma_{w_{(0,\dots,0,v_x,0,\dots,0)}} \in M.
	\]
	We apply the permutation $u$ on this relation obtained from the path connecting the vertex $(0,\dots,0,v_x,0,\dots,0)$ and $v = (0,\dots,0,v_x,v_{x+1},\dots,v_y,0,\dots,0)$:
	\[
	\begin{split}
	u(\sigma_{w_{(0,\dots,0,v_x',0,\dots,0)}} - \sigma_{w_{(0,\dots,0,v_x,0,\dots,0)}})
	&= \sigma_{u w_{(0,\dots,0,v_x',0,\dots,0)}} - \sigma_{u  w_{(0,\dots,0,v_x,0,\dots,0)}}\\
	&= \sigma_{w_{(0,\dots,0,v_x',v_{x+1},\dots,v_y,0,\dots,0)}} - \sigma_{w_{v}} \in M.
	\end{split}
	\]
	Since $(0,\dots,0,v_x',v_{x+1},\dots,v_y,0,\dots,0) \in G(\mathbf a)$, we obtain a desired linear relation. This completes the proof.
\end{proof}

\begin{example}\label{example_GVa_n42_2}
	We consider compositions $(1,1,2), (1,2,1)$, and $(2,1,1)$ in Example~\ref{example_GVa_n42}.
	We have observed that

	\begin{itemize}
		\item $\mathbf a = (1,1,2)$:
		\[
		\hats_{4312} =  2\sigma_{4312} + 4\sigma_{4213} + 6\sigma_{3214} + \underbrace{2(\sigma_{4231} - \sigma_{4132}) + 4\sigma_{3241} - 2 \sigma_{3142} - 2\sigma_{2143}}_{\mathbf a(v) = (1,2,1) > (1,1,2)} + \underbrace{2(\sigma_{3421} - \sigma_{2431})}_{\mathbf a(v) = (2,1,1) > (1,1,2)}.
		\]
		\item $\mathbf a = (1,2,1)$:  $\hats_{4231} = \sigma_{4231} + 2 \sigma_{3241} + \underbrace{(\sigma_{3421} - \sigma_{2431})}_{\mathbf a(v) = (2,1,1) > (1,2,1)}$.
		\item $\mathbf a = (2,1,1)$: $\hats_{3421} = 2 \sigma_{3421}$.
	\end{itemize}
	The maximal element is $\mathbf a = (2,1,1)$, and we obtain the element $\sigma_{3421}$ in $M \colonequals \sum_{u \in \La} M(u)$. For $\mathbf a = (1,2,1)$, consider
	\[
	s_3 \sigma_{2431} = \sigma_{2431} + \sigma_{4231} - \sigma_{3241} \in M.
	\]
	Here, since  $\sigma_{2431} = s_2 \cdot \sigma_{3421}$, both $\sigma_{2431} $ and  $s_3 \sigma_{2431}$ are contained in $M$.
	Therefore, we have $ \sigma_{4231} - \sigma_{3241} \in M$. Since the matrix $\begin{bmatrix} 1 & 2 \\ 1 & -1 \end{bmatrix}$ is nonsingular, both elements $\sigma_{4231}$ and $\sigma_{3241}$ are contained in $M$.
	Finally, consider $\mathbf a = (1,1,2)$. Since $s_1 \sigma_{4231} = \sigma_{4132}$ and $\sigma_{4231} \in M$, we have $\sigma_{4132} \in M$. Accordingly,
	\[
	s_2 \sigma_{4132} = \sigma_{4132} + \sigma_{4312} - \sigma_{4213} \in M
	\]
	and this implies $\sigma_{4312} - \sigma_{4213} \in M$. Moreover,
	\[
	s_3 (\sigma_{4312} - \sigma_{4213} ) = \sigma_{4312} - \sigma_{3214} \in M.
	\]
	Considering the following matrix
	\[
	\begin{array}{c}
	\begin{array}{ccc}
	\sigma_{4312} & \sigma_{4213} & \sigma_{3214} \\
	\end{array}\\
	\begin{bmatrix}
	1 & 2 & 3 \\
	1 & -1 & 0 \\
	1 & 0 & -1
	\end{bmatrix}
	\end{array}
	\]
	which is nonsingular, we obtain   $\sigma_{4312}, \sigma_{4213}, \sigma_{3214} \in M$.
\end{example}

\begin{example}\label{example_Hn_product}
	Suppose that $\mathbf a = (1,2,1,3)$. To draw the graph $G(\mathbf a)$, we need to consider $\mathbf a_1 = (1,2)$ and $\mathbf a_2 = (1,3)$. Accordingly, the graph $G(\mathbf a)$ and the elements $w_z$ are given as follows.
	\[
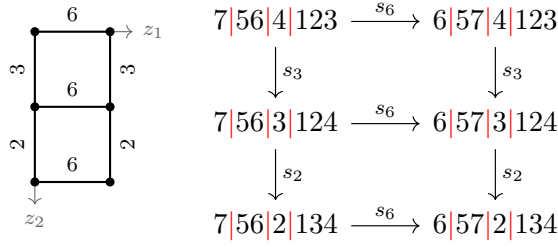
	\begin{array}{cc}
	\raisebox{-1.5cm}{
		\begin{tikzpicture}[every node/.style={font=\scriptsize}]
		\draw[->, color=black!60!white] (0,0)--(1.3,0) node [right] {$z_1$};
		\draw[->, color=black!60!white] (0,0)--(0,-2.3) node [below] {$z_2$};
		
		\draw[fill=black] (0,0) circle (1.5pt)
		(1,0) circle (1.5pt)
		(0,-1) circle (1.5pt)
		(0,-2) circle (1.5pt)
		(1,-2) circle (1.5pt)
		(1,-1) circle (1.5pt);
		
		\foreach \x in {0,-1,-2}{
			\draw[thick] (0,\x)--(1,\x) node[above, midway] {$6$};
		}
		\draw[thick] (0,-1)--(0,0) node[midway, above, sloped] {$3$};
		\draw[thick] (1,-1)--(1,0) node[midway,below, sloped] {$3$};
		
		\draw[thick] (0,-2)--(0,-1) node[midway, above, sloped] {$2$};
		\draw[thick] (1,-2)--(1,-1) node[midway, below, sloped] {$2$};
		
		\end{tikzpicture}}
	&
	\begin{tikzcd}
	7 \textcolor{red}{|} 56  \textcolor{red}{|} 4  \textcolor{red}{|} 123 \arrow[r, "s_6"] \arrow[d, "s_3"] &
	6  \textcolor{red}{|} 57  \textcolor{red}{|} 4  \textcolor{red}{|} 123 \arrow[d, "s_3"] \\
	7  \textcolor{red}{|} 56  \textcolor{red}{|} 3  \textcolor{red}{|} 124  \arrow[r, "s_6"] \arrow[d, "s_2"]& 6  \textcolor{red}{|} 57  \textcolor{red}{|} 3  \textcolor{red}{|} 124 \arrow[d, "s_2"] \\
	7  \textcolor{red}{|} 56  \textcolor{red}{|} 2  \textcolor{red}{|} 134  \arrow[r, "s_6"] &
	6  \textcolor{red}{|} 57  \textcolor{red}{|} 2  \textcolor{red}{|} 134
	\end{tikzcd}
	\end{array}
	\]
	We first find the following linear relations obtained from $G(\mathbf a_1) \times \{\mathbf 0\}$ and $\{\mathbf 0\} \times G(\mathbf a_2)$.
	\[
	\begin{split}
	z = (1,0): & \quad s_6 \sigma_{576 4123} = \sigma_{576 4123} + \sigma_{7564123} - \sigma_{6574123}, \\
	z = (0,1): & \quad s_3 \sigma_{756 1243} = \sigma_{7561243} + \sigma_{756 4123} + \sigma_{756 1423} + \sigma_{756 2413}  - (\sigma_{7563124} + \sigma_{ 756 1324} + \sigma_{756 2314}).
	\end{split}
	\]
	For the vertex $(1,1)$, we need to consider the $s_3$-action on the linear relation related to the vertex $(1,0)$:
	\[
	s_3 (\sigma_{7564123} - \sigma_{6574123})
	= \sigma_{7563124} - \sigma_{6573124}
	\]
	which is a desired relation related to the vertex $w_z$ with $z = (1,1)$.
\end{example}

\begin{proof}[Proof of Theorem~\ref{thm:erasing conjecture}]
We denote the sum $\sum_{u \in \La} M(u)$ of the modules by $M$. Then $M$ is
contained in the whole vector space $H^{2k}(\mathcal H_n)$.
We first claim that
\begin{equation}\label{eq_M_is_H2k}
M =  H^{2k}(\mathcal H_n).
\end{equation}
To prove the claim, it is enough to show that $M$ contains $H^{2k}(\mathcal
H_n)$. We will show that for any composition $\mathbf a$ of $n$, the class
$\sigma_v$
satisfying $\mathbf a(v) = \mathbf a$ are all contained in $\sum_{u \in
\La}M(u)$, which proves the claim.
Let $\mathbf a $ be a composition and $w = w(\mathbf a)$. For any permutation $v$ satisfying $\mathbf a(v) = \mathbf a$, there exist $z \in V(G(\mathbf a))$  such that $v = w_z$ by Proposition~\ref{lemma_wz}.
On the other hand,  we have already shown in Lemma~\ref{lemma_main_thm} that $\sigma_{w_z}$ is also in $\sum_{u \in \Gk} M(u)$ for any $z \in V(G(\mathbf a))$.
Therefore, any element $\sigma_v$ satisfying $\mathbf a(v) = \mathbf a$ is
contained in $\sum_{u \in \La} M(u)$. This proves the
claim~\eqref{eq_M_is_H2k}.

For $u \in \La$,  the stabilizer subgroup of $\hats_u \in
H^{\ast}_T(\mathcal H_n)$ is the same as $\mathfrak{S}_u$   by
Lemma~\ref{lem:generator}(3). Therefore the stabilizer subgroup $\Stab(\hats_u)$
of $\hats_u \in H^{\ast}(\mathcal H_n)$ contains $\mathfrak{S}_u$.
Because of the construction of $M(u)$ and the orbit-stabilizer theorem, we obtain
\begin{equation}\label{eq_dim_Mu}
\dim M(u) \leq \frac{|\mathfrak{S}_n|}{|\Stab(\hats_u)|}
\leq \frac{|\mathfrak{S}_n|}{|\mathfrak{S}_u|}.
\end{equation}
Accordingly, we obtain
\begin{equation}\label{eq_H2k_and_M_ineq}
\dim H^{2k}(\mathcal H_n) \leq \sum_{u \in \La} \dim M(u)
\leq \sum_{u \in \La} \frac{|\mathfrak{S}_n|}{|\mathfrak{S}_u|}
= \dim H^{2k}(\mathcal H_n).
\end{equation}
Here, the first inequality derives from~\eqref{eq_M_is_H2k}, and the second
inequality derives from~\eqref{eq_dim_Mu}. The last equality is provided by
Proposition~\ref{prop:young subgroup type}.
From the inequalities~\eqref{eq_H2k_and_M_ineq}, we obtain
\[
\Stab(\hats_u) = \mathfrak{S}_u \quad \text{ and } \quad
\bigoplus_{u \in \La} M(u) = H^{2k}(\mathcal H_n),
\]
which proves the theorem.
\end{proof}

\subsection*{Acknowledgements}

We would like to thank Michel Brion for his kind explanation of the equivariant multiplicity, Megumi Harada and Martha Precup for their interest and comments, especially on the incomplete description of the closure of  Białynicki-Birula cells in the first version of this paper, and Patrick Brosnan and Timothy Chow for valuable discussions and suggestions leading us to the erasing conjecture.
We are grateful to the IBS Center for Geometry and Physics for their hospitality and support during our visit.
We   sincerely appreciate the referee's  careful reading and helpful suggestions, significantly improving the paper.


 \providecommand{\bysame}{\leavevmode\hbox to3em{\hrulefill}\thinspace}
 \providecommand{\MR}{\relax\ifhmode\unskip\space\fi MR }
 \providecommand{\MRhref}[2]{%
   \href{http://www.ams.org/mathscinet-getitem?mr=#1}{#2}
 }
 \providecommand{\href}[2]{#2}

\end{document}